\documentclass[11pt]{amsart}

\usepackage{amssymb}
\usepackage{amscd}
\usepackage{amsmath}
\usepackage{amssymb}
\usepackage{amsfonts}
\usepackage{amsmath}
\usepackage{amscd}
\usepackage{amsthm}
\usepackage[all,cmtip]{xy}
\usepackage[all]{xy}
\usepackage{bm}
\usepackage[all]{xy}
\usepackage{color}
\usepackage[hmargin=3cm,tmargin=4cm,bmargin=3cm]{geometry}
\newtheorem{theorem}{Theorem}[section]
\newtheorem{lemma}[theorem]{Lemma}
\newtheorem{prop}[theorem]{Proposition}
\newtheorem{coro}[theorem]{Corollary}
\newtheorem{conj}[theorem]{Conjecture}
\theoremstyle{remark}
\newtheorem{remark}[theorem]{Remark}
\newtheorem{exam}[theorem]{Example}
\newtheorem{definition}[theorem]{Definition}

\newcommand{\Det}{\mathrm{Det}}
\newcommand{\Hom}{\mathrm{Hom}}

\newcommand{\frD}{\mathfrak{D}}

\newcommand{\Fa}{\mathfrak{a}}
\newcommand{\fra}{\mathfrak{a}}

\newcommand{\frc}{\mathfrak{c}}
\newcommand{\frf}{\mathfrak{f}}

\newcommand{\QQ}{\mathbb{Q}}
\newcommand{\IQ}{\mathbb{Q}}
\newcommand{\IC}{\mathbb{C}}
\newcommand{\IZ}{\mathbb{Z}}
\newcommand{\IR}{\mathbb{R}}
\newcommand{\RR}{\mathbb{R}}
\newcommand{\ZZ}{\mathbb{Z}}
\newcommand{\CC}{\mathbb{C}}

\newcommand{\Qu}{\mathbb{Q}}
\newcommand{\Qc}{{\mathbb{Q}^c}}
\newcommand{\Ce}{{\mathbb{C}}}
\newcommand{\Qp}{{\mathbb{Q}_p}}
\newcommand{\Ql}{{\mathbb{Q}_\ell}}
\newcommand{\Qlc}{{\mathbb{Q}_\ell^c}}

\newcommand{\Ze}{\mathbb{Z}}
\newcommand{\Zl}{{\mathbb{Z}_\ell}}
\newcommand{\Zp}{\mathbb{Z}_p}
\newcommand{\sseq}{\subseteq}
\newcommand{\Fp}{\mathbb{F}_p}
\newcommand{\frp}{\mathfrak{p}}
\newcommand{\lra}{\longrightarrow}
\renewcommand{\OE}{\mathcal{O}_E}
\newcommand{\OK}{\mathcal{O}_K}
\newcommand{\OKG}{\OK[G]}
\newcommand{\ind}{\mathrm{ind}}
\newcommand{\ZG}{\Ze[G]}
\newcommand{\QG}{{\Qu[G]}}
\newcommand{\ZlG}{\Zl[G]}
\newcommand{\QlG}{\Ql[G]}
\newcommand{\Cl}{\mathrm{Cl}}
\newcommand{\Gl}{\mathrm{GL}}
\newcommand{\tr}{\mathrm{tr}}

\newcommand{\calF}{\mathcal{F}}
\newcommand{\calL}{\mathcal{L}}
\newcommand{\calM}{\mathcal{M}}
\newcommand{\calO}{\mathcal{O}}
\newcommand{\calA}{\mathcal{A}}
\newcommand{\calN}{\mathcal{N}}
\newcommand{\calW}{\mathcal{W}}

\newcommand{\Nrd}{\mathrm{Nrd}}
\newcommand{\tensor}{\otimes}
\newcommand{\HCl}{\mathrm{HCl}}
\newcommand{\eHCl}{\mathrm{eHCl}}
\newcommand{\cok}{\mathrm{Cok}}
\newcommand{\id}{\mathrm{id}}
\newcommand{\rec}{\mathrm{rec}}
\newcommand{\cl}{\mathrm{cl}}

\newcommand{\mpar}[1]{}

\begin{document}

\title[Refined metric and hermitian structures]{On refined metric and hermitian\\
 structures in arithmetic, I:\\
Galois-Gauss sums and weak ramification}

\author{Werner Bley, David Burns and Carl Hahn}

\address{Math. Inst. Univ. M\"unchen,
Theresienstr. 39,
D-80333 München,
Germany}
\email{bley@math.lmu.de}

\address{King's College London,
Department of Mathematics,
London WC2R 2LS,
U.K.}
\email{david.burns@kcl.ac.uk}

\address{15 Keswick Avenue,
London SW15 3QH,
U.K.}
\email{carlhahn01@gmail.com}


\maketitle

\begin{abstract} We use techniques of relative algebraic $K$-theory to develop a common refinement of the theories of metrised and hermitian Galois structures in arithmetic. As a first application of the general approach, we then use it to prove several new results,
and to formulate several explicit new conjectures, concerning the detailed arithmetic properties of a natural class of wildly ramified Galois-Gauss sums.
\end{abstract}
\tableofcontents

\section{Introduction}\label{intro}\mpar{intro}

This article has essentially two main purposes. Firstly, we shall use techniques of relative algebraic $K$-theory to develop a natural, and very general, algebraic formalism that gives a common, and strong, refinement of the theory of `hermitian modules' and `hermitian class groups' described by Fr\"ohlich in \cite{fro2} and of the theory of `metrised modules and complexes' and `arithmetic classgroups'  introduced by Chinburg, Pappas and Taylor in \cite{cpt}.

Secondly, as a first concrete application of this refined theory, we shall show that it gives considerable new insight on the detailed arithmetic properties of a natural class of wildly ramified Galois-Gauss sums.

To give a few more details we fix a finite group $\Gamma$ and recall that a hermitian $\Gamma$-module is a pair comprising a finitely generated projective $\Gamma$-module together with a non-degenerate $\Gamma$-invariant pairing on this module. Fr\"ohlich showed that such modules are naturally classified by a `discriminant' invariant that lies in the hermitian classgroup ${\rm HCl}(\Gamma)$ of $\Gamma$ and is defined in terms of idelic-valued functions on the ring $R_\Gamma$ of $\QQ^c$-valued virtual characters of $\Gamma$.

This theory was developed with arithmetic applications in mind since for any  tamely ra\-mi\-fied Galois extension of number fields $L/K$ with ${\rm Gal}(L/K) = \Gamma$ the ring of algebraic integers of $L$ constitutes a hermitian $\Gamma$-module when endowed with its natural trace pairing. In this setting, Fr\"ohlich conjectured, and Cassou-Nogu\`es and Taylor subsequently proved (\cite{cn-t2}),
that the corresponding discriminant element uniquely characterises the Artin root numbers of irreducible complex symplectic characters of $\Gamma$. The latter result is commonly regarded as the highlight of classical `Galois module theory', as had been developed in the 1970's and 1980's (for more details see \cite{fro2})

To develop an analogous theory in the setting of arithmetic schemes admitting a tame action of $\Gamma$, Chinburg, Pappas and Taylor subsequently defined a  metrised $\Gamma$-module (respectively, complex of $\Gamma$-modules) to be a pair comprising a finitely generated projective $\Gamma$-module and a collection of suitable metrics on the isotypic components of the complexified module (respectively, a perfect complex of $\Gamma$-modules together with  metrics on the isotypic components of the complexified cohomomology modules). To classify such structures they defined the Arithmetic classgroup  $A(\Gamma)$ of $\Gamma$ in terms of idelic-valued functions on $R_\Gamma$ and showed each metrised $\Gamma$-module (respectively complex) gives rise to an associated invariant in $A(\Gamma)$.

To describe a common refinement of the above algebraic theories we construct canonical homomorphisms $\Pi_\Gamma^{\rm met}$ and
$\Pi_\Gamma^{\rm herm}$ from the relative algebraic $K_0$-group $K_0(\ZZ[\Gamma],\QQ^c[\Gamma])$ of the ring inclusion
$\ZZ[\Gamma] \to \QQ^c[\Gamma]$ to the group $A(\Gamma)$ and to a natural extension of the group ${\rm HCl}(\Gamma)$ respectively.

We then show that $\Pi_\Gamma^{\rm met}$ and $\Pi_\Gamma^{\rm herm}$ send each of the natural generating elements of
$K_0(\ZZ[\Gamma],\QQ^c[\Gamma])$, respectively of the subgroup $K_0(\ZZ[\Gamma],\QQ[\Gamma])$ of
$K_0(\ZZ[\Gamma],\QQ^c[\Gamma])$, to the difference of the natural invariants of two metrised modules in $A(\Gamma)$,
respectively of the discriminants of two hermitian modules in ${\rm HCl}(\Gamma)$.

To define the homomorphisms
$\Pi_\Gamma^{\rm met}$ and $\Pi_\Gamma^{\rm herm}$ we rely on a description of the group
$K_0(\ZZ[\Gamma],\QQ^c[\Gamma])$ in terms of idelic-valued functions on $R_\Gamma$ that is proved by
Agboola and the first author in \cite{AgboolaBurns}.

The strategy to apply this theory in arithmetic settings is then twofold.
In any given setting, one first hopes to identify a canonical
element of $K_0(\ZZ[\Gamma],\QQ^c[\Gamma])$ that at least one of $\Pi_\Gamma^{\rm met}$ or $\Pi_\Gamma^{\rm herm}$ sends
to arithmetic invariants that have been considered previously.  
 Then one can hope to prove, or at least to formulate conjecturally, a precise relation in $K_0(\ZZ[\Gamma],\QQ^c[\Gamma])$
that projects (under either $\Pi_\Gamma^{\rm met}$ or $\Pi_\Gamma^{\rm herm}$ or both) to recover pre-existing results,
or conjectures, in ${\rm A}(\Gamma)$ and ${\rm HCl}(\Gamma)$.

 In any case in which this can be achieved one can reasonably hope to obtain up to three significant outcomes.

 Firstly, one will obtain strong refinements of earlier results in the literature since both of the homomorphisms
$\Pi_\Gamma^{\rm met}$ and $\Pi_\Gamma^{\rm herm}$ have large kernels.

 Secondly, one can hope to obtain an explanation of any parallel aspects of the nature of earlier
results in ${\rm A}(\Gamma)$ and ${\rm HCl}(\Gamma)$.

 Thirdly, and perhaps most importantly, since $K_0(\ZZ[\Gamma],\QQ[\Gamma])$ has a canonical direct sum decomposition as
$\bigoplus_\ell K_0(\ZZ_\ell[\Gamma],\QQ_\ell[\Gamma])$, where $\ell$ runs over all primes, theorems and conjectures in
${\rm A}(\Gamma)$ and ${\rm HCl}(\Gamma)$ that appeared to be intrinsically global in nature are replaced by problems that can admit natural local decompositions and hence become easier to study.

Whilst there is, of course, no guarantee that this strategy can work in all natural settings, in this article we show that it works very well in the setting of hermitian and metrised modules that arise from fractional ideals of number fields and their links to classical Galois-Gauss sums.

In addition, in a subsequent article it will be shown that the same approach can also be used to refine the theory of Chinburg, Pappas and Taylor related to connections between the Zariski cohomology complexes of sheaves of differentials on arithmetic schemes with a tame action of a finite group and the associated epsilon constants and, in particular, to explain the similarity between the results obtained in \cite{cpt} and \cite{cpt2}.

A little more precisely, in the present article we first use the above approach in the setting of tamely ramified extensions of number fields to quickly both refine and extend previous results of Chinburg and the first author in \cite{BurnsChinburg} related to the links between Galois-Gauss sums and the hermitian modules comprising fractional powers of the different of $L/K$ endowed with the natural trace pairing.

In the main body of the article we then consider wildly ramified Galois-Gauss sums. Whilst the arithmetic properties of such sums are still in general poorly understood, significant progress has been made by Erez and others (see, for example, \cite{ErezTaylor}) in the case of Galois extensions $L/K$ that are both of odd degree and `weakly ramified' in the sense of \cite{Erez}.

We recall, in particular, that under these hypotheses there exists a unique fractional ideal $\mathcal{A}_{L/K}$ of $L$ the square of which is equal to the inverse of the different of $L/K$ and that the hermitian Galois structure of $\mathcal{A}_{L/K}$ has been shown in special cases to be closely linked to the properties of Galois-Gauss sums twisted by second Adams operators.

Following the general strategy described above, we shall now show that for any such extension $L/K$, with ${\rm Gal}(L/K) = \Gamma$, there exists a canonical element $\mathfrak{a}_{L/K}$ of $K_0(\ZZ[\Gamma],\QQ^c[\Gamma])$ that simultaneously controls both the hermitian and metrised structures that are naturally associated to $\mathcal{A}_{L/K}$.

We then prove that $\mathfrak{a}_{L/K}$ belongs to, and has finite order in, the subgroup $K_0(\ZZ[\Gamma],\QQ[\Gamma])$ of $K_0(\ZZ[\Gamma],\QQ^c[\Gamma])$ and also that it behaves well functorially under change of extension.

We also show that $\mathfrak{a}_{L/K}$ has a canonical decomposition as a sum of elements constructed from local fields and then use this decomposition to compute $\mathfrak{a}_{L/K}$ explicitly in several important cases.

By using these results we can then derive several unconditional results concerning the hermitian and metrised structures associated to $\mathcal{A}_{L/K}$ and thereby, for example, extend the main results of the celebrated article \cite{ErezTaylor} of Erez and Taylor.

In the general case, these results also combine with extensive numerical computation to motivate us to formulate several new conjectures
 concerning the detailed arithmetic properties of the local Galois-Gauss sums that are attached to weakly ramified extensions.

In the first such conjecture (Conjecture \ref{local conj}) we predict a precise conjectural description of $\mathfrak{a}_{L/K}$ in terms of local `Galois-Jacobi' sums and the fundamental classes of local class field theory.

We show that this explicit conjecture is equivalent to a special case of the `local epsilon constant conjecture' formulated by Breuning in \cite{mb2} and hence provides the first concrete link between the theory of the square root of the inverse different and the general framework of Tamagawa number conjectures that originated with Bloch and Kato in \cite{BK}.

At the same time, this link combines with the methods developed here to give a new, and effective, strategy for proving the epsilon constant conjecture formulated by the first and second authors in \cite{BleyBurns} for certain new families of wildly ramified Galois extensions of number fields.

Then, in Conjecture \ref{second local conj}, we predict that $\mathfrak{a}_{L/K}$ can also be directly computed in terms of a naturally-defined `idelic twisted unramified  characteristic' element. This simple (and, we feel, quite surprising) conjectural formula can be proved completely in certain important special cases and is also supported by extensive numerical computations.

Upon comparing the predictions made in Conjectures \ref{local conj} and \ref{second local conj} one then derives a new, and explicit, conjectural formula for Galois-Jacobi sums in terms of local fundamental classes (for details see Remark \ref{last rem}).

This framework of new conjectures concerning the arithmetic properties of wildly ramified Galois-Gauss sums is surely worthy of further investigation.

However, to finish with an even more concrete example of the insight that comes from using techniques of relative algebraic $K$-theory we recall that in \cite{Vinatier2}
Vinatier conjectures that $\mathcal{A}_{L/K}$ is a free $\Gamma$-module when $K = \QQ$ and is able,
by using the connection to twisted Galois-Gauss sums, to prove this conjecture if the decomposition groups in ${\rm Gal}(L/\QQ)$ of each
wildly ramified prime are abelian (see \cite{Vinatier1}). The conjecture is also known to hold if $L/\Qu$ is tamely ramified by the work of Erez in \cite{Erez}. However, aside from numerical verifications in a small (finite) number of cases (see \cite{Vinatier3}), there is still essentially nothing known about this conjecture in the non-abelian weakly and wildly ramified case.

By contrast, applying our approach in this setting now allows us to show easily that Vinatier's Conjecture naturally decomposes into a family of corresponding conjectures concerning extensions of local fields. This observation leads directly to a general `finiteness result' for Vinatier's Conjecture and hence renders the conjecture accessible to effective computation. In particular, in this way we are able to prove the conjecture for several new, and infinite, families of non-abelian wildly ramified Galois extensions.

Although we do not pursue it here, we believe it likely that the same local approach would also shed light on several of the explicit questions that were recently raised by Caputo and Vinatier in the introduction to \cite{cv}.

Finally, we would like to note that much of this work grew out of the King's College London PhD Thesis \cite{hahn} of the third author.
\\

{\bf Acknowledgements} We are grateful to St\'ephane Vinatier for interesting correspondence and to Claus Fieker for his valuable help with the computation of the automorphism groups for the degree
$63$ extensions that are discussed in \S\ref{deg 63 extensions}. The second author is also grateful to Adebisi Agboola for stimulating discussions concerning related work. Finally, we are very grateful to the anonymous referee for a very careful reading of the article and, in particular, for having pointed out an error in an earlier version of Theorem \ref{thm tame}.

\section*{\large{Part I: The general approach and first examples}}

In this part of the article we shall first review some basic facts concerning relative algebraic $K$-theory and the theories of both arithmetic and hermitian classgroups. We then establish a new link between these theories that will play a key role in subsequent arithmetic applications.

Throughout the section we illustrate abstract definitions and results by means of arithmetic examples that are motivated by our later applications.

For any Galois extension of fields $F/E$ we set $G({F/E}) := {\rm Gal}(F/E)$. We write $\mathbb{Q}^c$ for the algebraic closure of
$\mathbb{Q}$ in $\mathbb{C}$ and for any number field $E \sseq \mathbb{Q}^c$ we also write $\Omega_E$ for the absolute Galois group $G({\mathbb{Q}^c/E})$.

 For any finite group $\Gamma$ we write $\widehat\Gamma$ for the set of irreducible $\mathbb{Q}^c$-valued characters of $\Gamma$.
If $\ell$ denotes a rational prime, then we write $\widehat\Gamma_\ell$ for the set of irreducible $\Qlc$-valued characters.

\section{Relative $K$-theory, metric structures and hermitian structures}

\subsection{Relative algebraic $K$-theory}

We fix a finite group $\Gamma$ and a Dedekind domain $R$ of characteristic zero and write $F$ for the field of fractions of $R$.

For any extension field $E$ of $F$ and any $R[\Gamma]$-module $M$ we set $M_E := E\otimes_R M$ and for any homomorphism $\phi: M \to N$ of $R[\Gamma]$-modules we write $\phi_E: M_E \to N_E$ for the induced homomorphism of $E[\Gamma]$-modules.

\subsubsection{}We write $K_0(R[\Gamma],E[\Gamma])$ for the relative algebraic $K_0$-group that arises
from the inclusion of rings $R[\Gamma]\subset E[\Gamma]$ and we use the description of this group in
 terms of explicit generators and relations that is given by Swan in
 \cite[p. 215]{swan}.

 We recall in particular that in this description each element of
 $K_0(R[\Gamma],E[\Gamma])$ is represented by a triple $[P,\phi, Q]$ where $P$ and
 $Q$ are finitely generated projective left $R[\Gamma]$-modules and
 $\phi \colon P_E  \lra Q_E$  is an isomorphism of (left) $E[\Gamma]$-modules.

We write ${\Cl}(R[\Gamma])$ for the reduced projective classgroup of $R[\Gamma]$ (as discussed in \cite[\S 49]{CurtisReinerII})
and often use the fact that there exists a canonical exact commutative diagram

\begin{equation} \label{E:kcomm}
\begin{CD} K_1(R[\Gamma]) @> >> K_1(E[\Gamma]) @> \partial^1_{R,E,\Gamma} >> K_0(R[\Gamma],E[\Gamma]) @> \partial^0_{R,E,\Gamma} >> {\rm Cl}(R[\Gamma])\\
@\vert @A\iota AA @A\iota' AA @\vert\\
K_1(R[\Gamma]) @> >> K_1(F[\Gamma]) @> \partial^1_{R,F,\Gamma}  >> K_0(R[\Gamma],F[\Gamma]) @> \partial^0_{R,F,\Gamma}  >> {\rm Cl}(R[\Gamma]).
\end{CD}
\end{equation}
Here the map $\iota$ is induced by the inclusion $F[\Gamma]\subseteq E[\Gamma]$ and $\iota'$
sends each element $[P,\phi, Q]$ to $[P,E\otimes_F\phi,Q]$. These maps are injective and will usually be regarded as inclusions.
The map $\partial^0_{R,E,\Gamma}$ sends each element $[P,\phi, Q]$ to $[P]-[Q]$. (For details of all the other homomorphisms that occur above see \cite[Th. 15.5]{swan}.)


We write $K_0T(R[\Gamma])$ for the Grothendieck group of finite $R[\Gamma]$-modules that are of finite projective dimension and recall that there are natural isomorphisms of abelian groups
\begin{equation}\label{decomp}
K_0T(R[\Gamma]) \cong K_0(R[\Gamma],F[\Gamma])\cong \bigoplus_v K_0(R_v[\Gamma],F_v[\Gamma]).
\end{equation}
We choose the normalisation of the first isomorphism so that for any finite $R[\Gamma]$-module $M$ of finite
projective dimension, and any resolution of the form $0 \to P\xrightarrow{\theta} P' \to M \to 0$,
where the modules $P$ and $P'$ are finitely generated and projective,
the class of $M$ in $K_0T(R[\Gamma])$ is sent to $[P,\theta_F, P']$.
In addition, the direct sum in (\ref{decomp}) runs over all non-archimedean places $v$ of $F$ and the second
isomorphism is the diagonal map induced by the homomorphisms
\begin{equation}\label{localization map}
\pi_{\Gamma,v}: K_0(R[\Gamma],F[\Gamma])\to K_0(R_v[\Gamma],F_v[\Gamma])
\end{equation}
that sends each element $[X,\xi, Y]$ to $[X_v, \xi_v, Y_v]$, where we set $X_v := R_v\otimes_RX$ and $\xi_v := F_v\otimes_F\xi$.

We write $\zeta(A)$ for the centre of a ring $A$. Then to compute in $K_1(E[\Gamma])$ one uses the `reduced norm' homomorphism
\begin{equation*}
 \mathrm{Nrd}_{E[\Gamma]}: K_1(E[\Gamma]) \to \zeta(E[\Gamma])^{\times}
\end{equation*}
which sends the class of each pair $(V,\phi)$, where $V$
is a finitely generated free $E[\Gamma]$-module and $\phi$ is an automorphism of $V$ (as $E[\Gamma]$-module), to the reduced norm of
 $\phi$, considered as an element of the semisimple $E$-algebra
${\rm End}_{E[\Gamma]}(V)$. If $E\subseteq \mathbb{Q}^c$ is a number field and $|\Gamma|$ is odd,
then $\mathrm{Nrd}_{E[\Gamma]}$ is bijective by the Hasse-Schilling-Maass Norm Theorem (cf. \cite[Th.~(45.3)]{CurtisReinerII}).
The same is true for algebraically closed fields and $p$-adic fields. In particular we write
\begin{equation}\label{ebh}
\delta_\Gamma: \zeta(\mathbb{Q}^c[\Gamma])^\times \to K_0(\mathbb{Z}[\Gamma],\mathbb{Q}^c[\Gamma])
\end{equation}
for the composite $\partial^1_{\mathbb{Z},\mathbb{Q}^c,\Gamma}\circ ({\rm Nrd}_{\mathbb{Q}^c[\Gamma]})^{-1}$.
 For a rational prime  $\ell$ we write

\[
\delta_{\Gamma,\ell} \colon \zeta(\Qlc[\Gamma])^\times \to K_0(\Zl[\Gamma],\Qlc[\Gamma])
\]
for the composite $\partial^1_{\Zl,\Qlc,\Gamma}\circ ({\rm Nrd}_{\Qlc[\Gamma]})^{-1}$.

\subsubsection{}In the sequel we make much use of the fact that $K_0(\mathbb{Z}[\Gamma],\mathbb{Q}^c[\Gamma])$ can be explicitly described in terms of idelic-valued functions on the characters of $\Gamma$.

To recall this description we write $R_\Gamma$ for the free abelian group on $\widehat\Gamma$. Then the Galois group $\Omega_\mathbb{Q}$ acts on $R_\Gamma$ {via the rule $(\omega \circ \chi)(\gamma)
= \omega (\chi(\gamma))$ for every $\omega \in \Omega_\mathbb{Q}$, $\chi \in \widehat \Gamma$ and $\gamma \in \Gamma$}.

For each $a$ in ${\rm GL}_n(\mathbb{Q}^c[\Gamma])$ we define an element $\Det (a)$ of
$\Hom(R_\Gamma,\mathbb{Q}^{c\times})$ in the following way: if $T$ is a
representation over $\mathbb{Q}^c$ which has character $\phi$, then $\Det (a)(\phi) := \det( T(a))$. This definition depends
only on $\phi$ and not on the choice of representation $T$.
Analogously, if $w$ denotes a finite place of $\Qc$, then each element $a$ of ${\rm GL}_n( \mathbb{Q}_w^c[\Gamma] )$ defines
a homomorphism $\Det(a) \colon R_\Gamma   \lra ( \mathbb{Q}_w^c )^{\times}$ .

{We write $J_f(\mathbb{Q}^c[\Gamma])$ for the group of finite ideles of $\mathbb{Q}^c[\Gamma]$
and view $\mathbb{Q}[\Gamma]^\times$ as a subgroup of $J_f(\mathbb{Q}^c[\Gamma])$ via the
natural diagonal embedding. In particular, if $a$ is any element of ${\rm GL}_n(J_f(\mathbb{Q}^c[\Gamma]))$ the above approach
allows one} to define an element $\Det (a)$ of $\Hom(R_\Gamma, J_f(\mathbb{Q}^c))$ which is easily seen to be $\Omega_\Qu$-equivariant.
We set
\begin{equation*}
{U_f(\mathbb{Z}[\Gamma]) := \prod_{\ell}\mathbb{Z}_\ell[\Gamma]^\times \subset J_f(\mathbb{Q}[\Gamma])},
\end{equation*}
with the product taken over all primes $\ell$, and then define a homomorphism
\begin{equation}\label{Delta rel}
\Delta^{{\rm rel}}_{\Gamma} \colon
\Det(\QG^\times) \to
\frac{\Hom(R_\Gamma ,  J_f(\mathbb{Q}^c))^{\Omega _\mathbb{Q}} }{\Det(U_f(\mathbb{Z}[\Gamma]))} \times
\Det(\mathbb{Q}^c[\Gamma]^\times);
\qquad \theta \mapsto ([\theta], \theta^{-1})
\end{equation}
where $[\theta]$ denotes the class of $\theta$ modulo $\Det(U_f(\mathbb{Z}[\Gamma]))$. We recall that by the Hasse-Schilling-Maass
norm theorem
\[
\Det(\QG^\times) = \Hom^+(R_\Gamma ,\mathbb{Q}^{c\times})^{\Omega _\mathbb{Q}}
\]
where the right hand expression denotes Galois equivariant homomorphisms whose values on $R_\Gamma^s$, the group of
virtual symplectic characters, are totally positive. In particular,  if $\Gamma$ has odd order,
 then $\Det(\QG^\times) = \Hom(R_\Gamma ,\mathbb{Q}^{c\times})^{\Omega _\mathbb{Q}}$

It is shown in \cite[Th. 3.5]{AgboolaBurns} that there is a natural isomorphism of abelian groups
\begin{equation}\label{rel isom}
h^{\rm rel}_{\Gamma}: K_0(\mathbb{Z}[\Gamma] ,\mathbb{Q}^c[\Gamma]) \xrightarrow{\sim} \mathrm{Cok}(\Delta^{\rm rel}_{\Gamma}).
\end{equation}
We shall often use the explicit description of this map given in the following result (taken from \cite[Rem.~3.8]{AgboolaBurns}).

In the sequel for any ordered set of $d$ elements $\{e^j\}_{1 \le j \le d}$ we write $\underline{e^j}$ for the $d \times 1$ column vector with $j$-th entry $e^j$.

In addition, for any $\Gamma$-modules $X$ and $Y$ we write $\mathrm{Is}_{\Qu[\Gamma]}(X_\Qu, Y_\Qu)$ for the set of
isomorphisms of $\Qu[\Gamma]$-modules $X_\Qu\to Y_\Qu$.

\begin{lemma}\label{rep for relative group}\mpar{rep for relative group}
Let $c = [X,\xi, Y]$ be an element of $K_0(\Ze[\Gamma], \Qc[\Gamma])$ with locally free $\Ze[\Gamma]$-modules
$X$ and $Y$ of rank $d$. Choose a $\Qu[\Gamma]$-basis $\{ y_0^j \}$ of $Y_\Qu$ and, for each rational prime $p$,
a $\Zp[\Gamma]$-basis $\{ y_p^j\}$ of $Y_p$ and an $\Zp[\Gamma]$-basis $\{x_p^j\}$ of $X_p$ and define $\mu_p$
to be the element of $\Gl_d(\Qp[\Gamma])$ which satisfies $\underline{y_p^j} = \mu_p\cdot\underline{y_0^j}$.
Fix $\theta$ in $\mathrm{Is}_{\Qu[\Gamma]}(X_\Qu, Y_\Qu)$, note $\{\theta^{-1}(y_0^j)\}$ is a $\Qu[\Gamma]$-basis of
$X_\Qu$ and write $\lambda_p$ for the matrix in $\Gl_d(\Qp[\Gamma])$ with $\underline{x_p^j} = \lambda_p \cdot\underline{\theta^{-1}(y_0^j)}$.
Finally, write $\mu$ for the matrix in $\Gl_d(\Qc[\Gamma])$ that represents $\xi \circ (\theta^{-1} \tensor_\Qu \Qc)$
with respect to the $\Qc[\Gamma]$-basis $\{y^j\}$ of $Y_\Qc$.

Then the element $h^{{\rm rel}}_{\Gamma}(c)$  is represented by the homomorphism pair
\[
\left( \prod_{p} \Det(\lambda_p\cdot \mu_p^{-1}) \right)  \times \Det(\mu) \in
\Hom(R_\Gamma , J_f(\mathbb{Q}^c))^{\Omega _\mathbb{Q}}  \times
\Det(\mathbb{Q}^c[\Gamma]^\times).
\]
\end{lemma}

\subsubsection{}\label{example0} We give a first example of elements of relative algebraic $K$-groups that naturally arise in arithmetic contexts.

To do this we fix a finite Galois extension of number fields $L/K$ and set $G := G(L/K)$. Since $\QQ^c\subset \CC$ we identity the set $\Sigma(L)$ of field embeddings $L \to \IQ^c$ with the set of embeddings $L \to \IC$ and we write $H_L := \prod_{\Sigma(L)}\IZ$.

Then the natural action of $G$ on $\Sigma(L)$ endows $H_L$ with the structure of a $G$-module (explicitly, if  $\{ w_\sigma : \sigma \in \Sigma(L)\}$ is the canonical $\Ze$-basis of $H_L$, then $\gamma w_\sigma = w_{\sigma\circ\gamma^{-1}}$). This module is free of rank $[K:\QQ]$ since, if one fixes an extension $\hat \sigma$ in $\Sigma(L)$ of each $\sigma$ in $\Sigma(K)$, then the set $\{ w_{\hat\sigma}\}_{\sigma \in \Sigma(K)}$ is a basis of $H_L$ over $\ZZ[G]$.

In addition, the map
\[
\kappa_L:\mathbb{Q}^c\otimes_{\mathbb{Q}}L \rightarrow\prod_{\Sigma(L)}\mathbb{Q}^c = \QQ^c\otimes_\ZZ H_L
\]
that sends each element $z\otimes \ell$ to $(\sigma(\ell)z)_{\sigma \in \Sigma(L)}$ is then an isomorphism of $\mathbb{Q}^c[G]$-modules.

As a result, any full projective $\ZZ[G]$-sublattice $\mathcal{L}$ of $L$ gives rise to an associated element
\[ [\mathcal{L},\kappa_L,H_L]\]
of $K_0(\ZZ[G],\QQ^c[G])$.

In the case that $\mathcal{L}$ is an $\mathcal{O}_K[G]$-module the recipe in Lemma \ref{rep for relative group} gives rise to a useful description of this element that we record in the next result.

In this result (and the sequel) we use the following notation. For each element $b$ of $L$ with $L = K[G]\cdot b$ and each character $\chi$ in $\widehat G$ that is represented by a homomorphism of the form $T_\chi: G \to {\rm GL}_{n_\chi}(\QQ^c)$, one defines a resolvent element
\[ (b\mid \chi) := {\rm det}(\sum_{g \in G}g(b) T_\chi(g^{-1}))\]
and then an associated `norm-resolvent' by setting
\[
\calN_{K/\Qu}(b\mid \chi):= \prod_\omega (b \mid \chi^{\omega^{-1}})^\omega,
\]
where $\omega$ runs through a transversal of $\Omega_\Qu$ modulo $\Omega_K$.

For each finite place $v$ of $K$ we write $K_v$ for the completion of $K$ at $v$ and note that $L_v := L \tensor_K K_v \simeq \prod_{w|v}L_w$ is a
free $K_v[G]$-module of rank one. Then, in the same way as above, for each element $b_v$ in $L_v$ such that $L_v = K_v[G]\cdot b_v$ we define an idelic-valued resolvent $(b_v \mid \chi)$ and an idelic-valued norm resolvent $\calN_{K/\Qu}(b_v \mid \chi)$ (for more details see \cite[\S~4.1]{BurnsChinburg}). For an $\mathcal{O}_K$-module $\calL$ we also set $\calL_v := \calL \tensor_{\OK} \calO_{K_v}$.

\begin{lemma}\label{norm resolvent lemma}\mpar{norm resolvent lemma}
Fix a $\ZZ$-basis $\{a_\sigma\}_{\sigma \in \Sigma(K)}$ of $\OK$, an element $b$ of $L$ such that $L = K[G]\cdot b$ and,
for each finite place $v$ of $K$, an element $b_v$ of $L_v$ such that $\calL_v = \calO_{K_v}[G]\cdot b_v$.

Then the element $h_G^{\mathrm{rel}}([\calL, \kappa_L, H_L])$ is represented by the homomorphism pair $(\theta_1\theta_2^{-1}, \theta_2 \theta_3)$ where for $\chi$ in $\widehat G$ one has
\[
\theta_1(\chi) := \prod_v \calN_{K/\Qu}(b_v \mid \chi), \quad \theta_2(\chi) := \calN_{K/\Qu}(b \mid \chi), \quad \theta_3(\chi) := \delta_K^{\chi(1)}
\]
with
$\delta_K := \det( \tau(a_\sigma))_{\sigma, \tau \in  \Sigma(K)}$.
\end{lemma}

\begin{proof}
 Since $H_L$ is a free $G$-module, in terms of the notation of Lemma \ref{rep for relative group} we can and will use the basis $\{y_0^j\} = \{y_p^j\} = \{ w_{\hat\sigma}\}_{\sigma \in \Sigma(K)}$ so that $\mu_p$ is the identity matrix for every prime $p$.

We write $\theta_b: L \lra H_{L, \Qu}$ for the $\QG$-linear isomorphism that sends each element $a_\sigma \cdot b$ to $w_{\hat\sigma}$.

For each prime $p$ we set $\calO_{K,p} := \Zp \tensor_\Ze \OK \simeq \prod_{v \mid p}\calO_{K_v}$ and $\calL_p := \Zp \tensor_\Ze \calL \simeq \prod_{v \mid p}\calL_v$. We note that the element $b_p := \left( b_v \right)_{v\mid p}$ is a $\calO_{K,p}[G]$-generator of $\calL_p$ and that the homomorphism of $\Zp[G]$-modules $\theta_{b_p}: \calL_p \lra H_{L,p}$ that sends each element $a_\sigma\cdot  b_p$ to $w_{\hat\sigma}$ is bijective.

For the basis $\{x_p^j\}$ that occurs in the statement of Lemma \ref{rep for relative group} we choose $\{a_\sigma \cdot b_p\}_{\sigma \in \Sigma(K)}$ and then write $\lambda_p$ for the matrix in $\Gl_d(\Qp[G])$ which satisfies $\underline{a_\sigma \cdot b_p} = \lambda_p \cdot \underline{\theta_{b}^{-1}(w_{\hat\sigma})}$. We note, in particular, that $\lambda_p$ is the coordinate matrix of the $\Qp[G]$-linear map $(\QQ_p^c\otimes_{\QQ}\theta_b) \circ (\QQ_p^c\otimes_{\ZZ_p}\theta_{b_p})^{-1}$ with respect to the basis $\{ w_{\hat\sigma}\}$.

Then Lemma \ref{rep for relative group} implies that $h_G^{\mathrm{rel}}([\calL, \kappa_L, H_L])$ is represented by the homomorphism pair
\[ \left( \prod_p \Det(\lambda_p) \right) \times \Det(\mu)\]
where $\mu$  is the coordinate matrix in $\Gl_d(\Qc[G])$ of $\kappa_L \circ (\QQ^c\otimes_\QQ\theta_{b})^{-1}$ with respect to the basis $\{ w_{\hat\sigma}\}$.

In addition, from \cite[(16) and (17)]{BleyBurns}, one knows that $\Det(\mu)(\chi) = \delta_K^{\chi(1)}\cdot\calN_{K/\Qu}(b \mid \chi)$ for each character $\chi$.

Finally to compute each homomorphism $\Det(\lambda_p) $ we note that
\[ (\QQ_p^c\otimes_\QQ\theta_{b})\circ (\QQ^c_p\otimes_{\QQ_p}\theta_{b_p})^{-1} = ((\QQ_p^c\otimes_\QQ\theta_{b})\circ (\QQ_p^c\otimes_{\QQ^c}\kappa_L)^{-1}) \circ ((\QQ_p^c\otimes_{\QQ^c}\kappa_L)\circ (\QQ^c_p\otimes_{\ZZ_p}\theta_{b_p})^{-1})\]
and write $\lambda_{p,2}$ for the coordinate matrix
of $(\QQ_p^c\otimes_{\QQ^c}\kappa_L)\circ (\QQ^c_p\otimes_{\ZZ_p}\theta_{b_p})^{-1}$.

Then using similar computations to those used to derive \cite[(16) and (17)]{BleyBurns} one finds that for each character $\chi$ one has
$$
 \Det(\lambda_{p,2})(\chi)= \calN_{K/\Qu}(b_p \mid \chi)  = \prod_{v \mid p}\calN_{K/\Qu}(b_v \mid \chi),
$$
as required to complete the proof.
\end{proof}


\subsection{Hermitian Modules and Classgroups}

In this section we recall some of the basic theory of hermitian modules and classgroups.
For more details see Fr\"ohlich \cite[Chap. II]{fro2}. Note however, that in contrast to the convention used in loc. cit.
we consider all modules as left modules.

\begin{definition}\label{herm def} A \emph{hermitian form} on a $\Gamma$-module $X$ is a nondegenerate bilinear map
\[
h \colon X_\Qu \times X_\Qu \to \IQ[\Gamma]
\]
that is $\IQ[\Gamma]$-linear in the first variable and satisfies $h(x,y)=h(y,x)^\sharp$ with $z \mapsto z^\sharp$
the $\IQ$-linear anti-involution of $\IQ[\Gamma]$ which inverts elements of $\Gamma$.

A \emph{hermitian} $\Gamma$-module is a pair $(X,h)$ comprising a finitely generated projective
$\Gamma$-module $X$ and a hermitian form $h$ on $X$.
\end{definition}

\begin{exam}\label{example3}
For any number field $K$ and any finite group
$\Gamma$ we extend the field-theoretic trace ${\tr}_{K/\IQ}: K \to \QQ$ to a linear map
$K[\Gamma] \to \QQ[\Gamma]$ by applying it to  the coefficients of each element of $K[\Gamma]$.

This action then gives rise to a hermitian form
\[
t_{K[\Gamma]} \colon  K[\Gamma]\times K[\Gamma] \to \IQ[\Gamma]
\]
by setting $t_{K[\Gamma]}(x,y) = {\rm tr}_{K/\IQ}(xy^\sharp)$. In particular, since $\mathcal{O}_K$ is a free $\IZ$-module
the pair $(\mathcal{O}_K[\Gamma], t_{K[\Gamma]})$ is a hermitian $\Gamma$-module.
\end{exam}

\begin{exam}\label{example2}
For any finite Galois extension $L/K$ of number fields, with $G = G({L/K})$, one obtains a hermitian form
\[
t_{L/K}: L\times L \to \IQ[G]
\]
by setting $t_{L/K}(x,y) = \sum_{g \in G}{\tr}_{L/\IQ}(x \cdot g(y))g$.
For each full projective $G$-sublattice $\mathcal{L}$ of $L$ the pair $(\mathcal{L}, t_{L/K})$
is then a hermitian $G$-module.
\end{exam}

\begin{exam}\label{pullback2}
Let $X_1$ and $X_2$ be finitely
generated projective $\Gamma$-modules and $\xi$ an isomorphism of $\QQ[\Gamma]$-modules
$X_{2, \Qu}\cong X_{1, \Qu}$. For any hermitian form $h$ on $X_1$ we define the `pullback of $h$
through $\xi$' to be the hermitian form $\xi^*(h)$ on $X_2$ that satisfies
\[
\xi^*(h)(x_2,y_2) = h(\xi(x_2),\xi(y_2))
\]
for all $x_2, y_2 \in X_2$.
\end{exam}

To classify general hermitian $\Gamma$-modules Fr\"ohlich defined (see, for example, \cite[Chap. II, (5.3)]{fro2}) the `hermitian classgroup'
${\HCl}(\Gamma)$ of $\Gamma$ to be the cokernel of the homomorphism
\begin{equation}\label{hcl def}
\Delta^{\rm herm}_\Gamma \colon
\Det(\Qu[\Gamma]^\times)  \to
\frac{{\rm Hom}({\rm R}_\Gamma,{\rm J}_{\rm f}(\IQ^c))^{\Omega_\IQ}}
{{\rm Det}({\rm U_f}(\IZ[\Gamma]))}\times {\rm Hom}({\rm R}^s_\Gamma,\IQ^{c\times})^{\Omega_\IQ}
;\;\;\;\theta\mapsto ([\theta]^{-1},\theta^s)
\end{equation}
where $R^s_\Gamma$ denotes the subgroup of $R_\Gamma$ generated by the set of irreducible symplectic characters of $\Gamma$ and $\theta^s$ denotes  the restriction of $\theta$ to $R^s_\Gamma$.

To each hermitian $\Gamma$-module $(X,h)$ Fr\"ohlich then associated a canonical 'discriminant' element ${\rm Disc}(X,h)$ in ${\rm HCl}(\Gamma)$ that is defined explicitly as follows.

\begin{definition}\label{pfaffian def}\mpar{pfaffian def}
Let $(X,h)$ be a hermitian $\Gamma$-module and write $d$ for the rank of the
free $\IQ[\Gamma]$-module $X_\Qu$.
Choose a $\IQ[\Gamma]$-basis $\{ x_0^j\}$ of $X_\Qu$ and,
for each prime $p$, a $\IZ_p[\Gamma]$-basis $\{x_p^j\}$ of $X_p$.
Then there exists an element $\lambda_p$ of ${\rm GL}_d(\IQ_p[\Gamma])$ with
$\underline{x_p^j} = \lambda_p\cdot\underline{x_0^j}$ and the `discriminant class'  ${\rm Disc}(X,h)$
is the element of ${\rm HCl}(\Gamma)$ represented by the pair
\[
\left( \prod_p {\rm Det}(\lambda_p), {\rm Pf}(h(x_0^i,x_0^j)) \right).
\]
Here ${\rm Pf}$ is the `Pfaffian' function in ${\rm Hom}({\rm R}^s_\Gamma,\IQ^{c\times})$
defined in \cite[Chap. II, Prop. 4.3]{fro2}.
\end{definition}

We end this section with a new definition that will be useful in the sequel.

\begin{definition} The `extended hermitian classgroup' ${\eHCl}(\Gamma)$ of $\Gamma$ is defined to the
cokernel of the homomorphism that is defined just as $\Delta^{\rm herm}_\Gamma$
except that the term ${\rm Hom}({\rm R}^s_\Gamma,\IQ^{c\times})^{\Omega_\IQ}$ on the right hand side
of (\ref{hcl def}) is replaced by ${\rm Hom}({\rm R}^s_\Gamma,\IQ^{c\times})$.
We regard ${\rm HCl}(\Gamma)$ as a subgroup of ${\rm eHCl}(\Gamma)$ in the obvious way.\end{definition}

\subsection{Metrised modules and Classgroups}

We quickly recall the definition of metrised modules and class groups. For further details we refer the reader to
\cite[\S 2 and  \S 3.1]{cpt}.

For each $\phi$ in $\widehat\Gamma$ we write $W_\phi$ for the Wedderburn
component of $\IQ^c[\Gamma]$ which corresponds to the contragredient
character $\overline{\phi}$ of $\phi$. Thus $W_\phi$ has character $\phi(1)\overline{\phi}$.

For any $\IQ^c[\Gamma ]$-module
$X$ we then set
\begin{equation*}
X_\phi := {\bigwedge}^{\rm top}_{\IQ^c}(X\otimes_{\IQ^c}W_\phi)^\Gamma,
\end{equation*}
where `$\bigwedge^{\rm top}_{\IQ^c}$' denotes the highest exterior power
over $\IQ^c$ which is non-zero, and $\Gamma$ acts diagonally on the tensor product. We recall from \cite[Lem.~2.3]{cpt} that $X_\phi \simeq \overline{W}_\phi X$.

Recall that $\Qc$ is the algebraic closure of $\Qu$ in $\Ce$. We write
$\sigma_\infty :\IQ^c \rightarrow \IC$ for the inclusion and
$\overline{z}$ for the conjugate of a complex number $z$.


\begin{definition}
A \emph{metrised} $\Gamma$-module is a pair $\left(X, \left\{ ||\cdot||_\phi \right\}_{\phi \in \widehat\Gamma} \right)$ comprising a finitely generated projective
$\Gamma$-module $X$ and a set $\{||\cdot||_\phi\}_{\phi \in \widehat\Gamma}$ of metrics on the complex lines
$\IC\otimes_{\IQ^c}X_\phi$ induced by positive definite hermitian forms $\mu_\phi$ on the spaces
$\IC\otimes_{\IQ^c}X_\phi$.

In this situation, we usually abbreviate $\left(X, \left\{ ||\cdot||_\phi \right\}_{\phi \in \widehat\Gamma} \right)$ to $\left(X, \mu_\bullet \right)$ and note that for each $\phi$ in $\widehat\Gamma$ and each element $x$ of $\IC\otimes_{\IQ^c}X_\phi$ one has $|| x ||_\phi^2 = \mu_\phi(x,x)$.\end{definition}

\begin{exam}\label{highest ext power exam} An important special case occurs when $\mu_\phi$ arises as the `highest exterior power' of a positive definite hermitian form $\tilde\mu_\phi$ on the space
\[ (X \tensor_\Ze W_\phi)^\Gamma \tensor_\Qc \Ce = ((X \tensor_\Ze \Ce) \tensor_\Ce (W_\phi \tensor_\Qc \Ce))^\Gamma.\]
In this case, for any $\Ce$-basis $v_1,  \ldots, v_d$ of this space one has
\[ || v_1 \wedge \ldots \wedge v_d ||_\phi^2 = \det(\left( \tilde\mu_\phi(v_i, v_j)_{1 \le i,j \le d} \right)).\]
\end{exam}

\

Let $\Gamma$ be a finite group.  Then the standard $\Gamma$-equivariant positive definite
hermitian form $\mu_{\IC [\Gamma]}$ on $\IC [\Gamma]$ is defined (for example, in \cite[\S~2.1]{cpt}) by setting
\[
\mu_{\IC [\Gamma]}\left( \sum_{g \in \Gamma}x_gg, \sum_{h \in \Gamma}y_hh \right) =
\sum_{g\in \Gamma}x_g\overline{y_g}.
\]

The associated $\Ce[\Gamma]$-valued hermitian form is the so-called `multiplication form'
\[ \hat\mu_{\Ce[\Gamma]} \colon \Ce[\Gamma] \times \Ce[\Gamma] \lra \Ce[\Gamma]\]
that sends each pair $(x,y)$ to $x \cdot \overline{y}$, where we extend complex conjugation to an anti-involution on $\Ce[\Gamma]$ by setting
\[
\overline{\sum_{\gamma \in \Gamma}a_\gamma \gamma } := \sum_{\gamma \in \Gamma}\overline{a_\gamma} \gamma^{-1}.
\]

\begin{exam}\label{example1}\mpar{example1}
In this example we use the hypotheses and notation of \S\ref{example0}.

(i) We write $\mu_L$ for the (unique) $\Gamma$-equivariant positive definite hermitian form
 on $\IC \otimes_\IZ H_L$ that satisfies
\[ \mu_L \left(
\sum_{\sigma \in \Sigma(L)}x_\sigma  w_\sigma, \sum_{\sigma \in \Sigma(L)}y_\sigma w_\sigma \right) = \sum_{\sigma \in \Sigma (L)}x_\sigma \overline{y_\sigma}.
\]
For each $\phi \in \widehat \Gamma$ the form $\mu_L$ together with the restriction of $\mu_{\IC[\Gamma]}$ on $\Ce \tensor_\Qc W_\phi$ induces a
positive definite hermitian form $\tilde\mu_{L,\phi}$ on the tensor product
\[
\left( \left( \Ce \tensor_\Ze H_L \right) \tensor_ \Ce \left( \Ce \tensor_\Qc W_\phi \right) \right)^\Gamma = \Ce \tensor_\Qc \left( H_L \tensor_\Ze W_\phi \right)^\Gamma.
\]
We then write $\mu_{L, \phi}$ for the positive definite hermitian form on
\[
\Ce \tensor_\Qc {\bigwedge}_\Qc^{\rm top} \left( H_L \tensor_\Ze W_\phi \right)^\Gamma =
{\bigwedge}_\Ce^{\rm top}\left( \Ce \tensor_\Qc (H_L \tensor_\Ze W_\phi)^\Gamma \right)
\]
that is obtained as the
highest exterior power of $\tilde\mu_{L, \phi}$ (as per the discussion in Example \ref{highest ext power exam}). The induced metric
\[ \mu_{L,\bullet} := \{ \mu_{L, \phi}\}_{\phi\in \widehat \Gamma}\]
on $H_L$ plays an important role in the sequel.
%

\smallskip

(ii) There is a $\Gamma$-equivariant positive definite hermitian form $h_L$
on $\Ce \otimes_\IQ L$ defined by
\[ h_L(z_1\otimes m,z_2\otimes n) = z_1\overline{z_2}\sum_{\sigma \in
\Sigma(L)}\sigma(m)
\overline{\sigma (n)} .\]
(This form is a scalar multiple of the `Hecke form' defined by Chinburg, Pappas and Taylor in \cite[\S5.2]{cpt}.) For each $\phi$ in $\widehat \Gamma$ we write $h_{L,\phi}$ for the positive definite hermitian form
on $(\IC\otimes _\IQ L)_\phi$ that is obtained as the highest
exterior power of the form on $(L\otimes_\IQ W_\phi)^G$ which is
induced by $h_L$ on $\Ce \otimes_\IQ L$ and by the restriction of
$\mu_{\IC [\Gamma]}$ on $\Ce \tensor_\Qc W_\phi$.

We set
\[ h_{L,\bullet} := \{ h_{L,\phi }\}_{\phi \in \widehat \Gamma}\]
and note that if
$\mathcal{L}$ is any full projective $\IZ [\Gamma]$-sublattice of $L$, then
the pair $(\mathcal{L}, h_{L, \bullet})$ is naturally a metrised $\Gamma$-module.
\end{exam}

\begin{exam}\label{pullback exam} Let $E \sseq \Qc$ be a subfield and let $X_1$ and $X_2$ be finitely
generated locally-free $\Ze[\Gamma]$-modules. Let $\xi$ denote an isomorphism of $E[\Gamma]$-modules $X_{2,E}\cong X_{1,E}$.
For each $\phi$ in $\widehat\Gamma$ we write
\[
\xi_{\phi} : (X_2\otimes_{\Ze}\Qc)_\phi\otimes_{\Qc}\CC
 \cong
(X_1\otimes_{\Ze}\Qc)_\phi\otimes_{\Qc}\CC
\]
for the isomorphism of complex lines which is induced by $\xi$. If $h$
is any metric on $X_1$, then we define the `pullback' of $h$ under $\xi$ to be the (unique) metric $\xi^*(h)$ on $X_2$ which satisfies
\[
\xi^*(h)_{\phi}(z) = h_{\phi}(\xi_{ \phi}(z))
\]
for all $\phi\in \widehat\Gamma$ and $z \in (X_2\otimes_{\Ze} \Qc)_\phi\otimes_{\Qc}\CC$.
\end{exam}

%
%
%
%

In order to classify metrised $\Gamma$-modules Chinburg, Pappas and Taylor (\cite[\S3.1 and \S3.2]{cpt})
defined the \emph{arithmetic classgroup} $A(\Gamma)$ of $\Gamma$ to be the cokernel of the homomorphism
 \[
\Delta^{\rm met}_\Gamma \colon
\Det(\Qu[\Gamma]^\times) \to
\frac{{\rm Hom}({\rm R}_\Gamma,{\rm J}_{\rm f}(\IQ^c))^{\Omega_\IQ}}{{\rm Det}({\rm U_f}(\IZ[\Gamma]))}
\times {\rm Hom}({\rm R}_\Gamma,\IR^\times_{>0});\;\;\;\theta\mapsto ([\theta],|\theta|)
\]
where we write $|\theta|$ for the homomorphism which sends each character $\phi$ in $\hat{\Gamma}$ to $|\theta(\phi)|^{-1}$.
Note that we adopt here the convention of \cite[\S4.2 and Rem.~4.4]{AgboolaBurns}, i.e. our $|\theta|$ is the inverse of the
map $|\theta|$ used in \cite{cpt}.

To each metrised $\Gamma$-module $(X,h)$ one can then associate a canonical `arithmetic class' $[X,h]$ in $A(\Gamma)$.

We next recall the explicit definition of this element from \cite[\S 3.2]{cpt} (see also \cite[Rem.~4.6]{AgboolaBurns}) and to do this we use the notation of Lemma \ref{rep for relative group}.

\begin{definition}\label{metrised element def} Let $(X,\mu_\bullet)$ be a metrised $\Gamma$-module, with $X$ locally-free over
$\IZ[\Gamma]$ of rank $d$.
Choose a $\IQ[\Gamma]$-basis $\{ x_0^j\}$ of $X_\Qu$ and,
for each prime $p$, a $\IZ_p[\Gamma]$-basis $\{x_p^j\}$ of $X_p$.
Then there exists an element $\lambda_p$ of ${\rm GL}_d(\IQ_p[\Gamma])$ such that
$\underline{x_p^j} = \lambda_p\cdot\underline{x_0^j}$.

For each $x$ in $X_\Qu$ we set
\[ r(x) := \sum_{\gamma \in \Gamma}\gamma (x)\otimes \gamma \in X \otimes _{\IZ} \IQ^c[\Gamma].\]
We note that for each $w$ in $W_\phi$ one has $r(x)(1\otimes w) \in (X \otimes_{\Ze}W_\phi)^\Gamma$
 where for each $w$ in $W_\phi$ the action of $r(x)$ on $1 \tensor w$
is defined by $r(x)(1 \tensor w) := \sum_{\gamma \in \Gamma} \gamma(x) \tensor \gamma(w)$.

Let $\{ w_{\phi,k}\}_{1\le k\le \phi(1)^2}$ be a $\IQ^c$-basis of $W_\phi$ that is
orthonormal with respect to the restriction of $\mu_{\Ce [\Gamma]}$ to
$W_\phi$. Then the set $\{ r(x_0^j)(1\otimes w_{\phi,k})\}_{j, k}$ is an
$\IQ^c$-basis of $(X \otimes_{\IZ} W_\phi )^\Gamma$ and so $\bigwedge_j \bigwedge_k r(x_0^j)(1\otimes w_{\phi,k})$ is an
$\IQ^c$-basis of $(X \otimes_{\IZ}\IQ^c)_\phi$.

We then define $[X, \mu_\bullet]$ to be the element of
$\mathrm{A}(\Gamma)$ that is represented by the homomorphism on
 $R_\Gamma$ which sends each character $\phi \in \widehat \Gamma$ to
\begin{equation}\label{arithmetic rep}
\prod_{p}{\rm Det}(\lambda_p)(\phi) \times
\left|\left|  \left(\bigwedge_j \bigwedge_k r(x_0^j)(1\otimes w_{\phi,k})\right)\otimes_{}1 \right|\right|_\phi^{1/\phi(1)}
\in J_f(\IQ^c)\times \IR^\times_{>0} .
\end{equation}
We note that it is straightforward to show that $[X,\mu]$ is independent of the precise choices
of bases $\{ x_0^j\}$, $\{x_{p}^{j}\}$ and $\{w_{\phi,k}\}$.
\end{definition}

\

As a concrete example, we now apply the above recipe in the setting of Example \ref{example1}(i). To do this we recall from \S\ref{example0} that $\{w_\sigma \colon \sigma \in \Sigma(L)\}$ denotes the canonical $\Ze$-basis of the $G$-module
$H_L = \prod_{\Sigma(L)}\Ze$. Moreover, in Example  \ref{example1}(i) we have defined a metric $\mu_{L, \bullet}$ on $H_L$ so that the pair $(H_L,\mu_{L,\bullet})$ gives rise to an element $[H_L,\mu_{L,\bullet}]$ of $A(G)$.

The following result will play an important role in a later argument.

\begin{lemma}\label{betti example}
The element $[H_L,\mu_{L,\bullet}]$ of $A(G)$ is represented
by the pair $(1,\theta)$ where $\theta$ sends each character $\phi$ of $\widehat{G}$ to $|G|^{[K:\IQ]\frac{\phi(1)}{2}}$.
\end{lemma}

\begin{proof} If $X = H_L$ and $\mu_\bullet = \mu_{L,\bullet}$, then in the notation of Definition \ref{metrised element def} we can take both $\{ x_0^j\}$ and $\{x_p^j\}$ to be the basis $\{ w_{\hat\sigma}\}_{\sigma \in \Sigma(K)}$ described in \S\ref{example0} and so $\lambda_p = 1$.

In addition, for a character $\phi$ in $\widehat{G}$, embeddings $\sigma$ and $\tau$ in $\Sigma(K)$ and integers $k$ and $\ell$ with $1\le k,\ell\le \phi(1)^2$ one has

\begin{align*} &\, (\mu_L\otimes \mu_{\CC[G]})(r(w_{\hat\sigma})(1\otimes
w_{\phi,k}), r(w_{\hat\tau})(1\otimes
w_{\phi,\ell})) \\
= &\, (\mu_L\otimes \mu_{\CC[G]})(\sum_{g \in G}g(w_{\hat\sigma})\otimes g(w_{\phi,k}),\sum_{h \in G}h(w_{\hat\tau})\otimes h(w_{\phi,\ell}))\\
= &\, \sum_{g,h} \mu_L(g(w_{\hat\sigma}),h(w_{\hat\tau}))\cdot\mu_{\CC[G]}(g(w_{\phi,k}),h(w_{\phi,\ell}))\\
= &\, \sum_{g,h} \delta_{g,h}\delta_{\hat\sigma,\hat\tau}\cdot\mu_{\CC[G]}(g(w_{\phi,k}),h(w_{\phi,\ell}))\\
= &\, \delta_{\hat\sigma,\hat\tau}\delta_{k,\ell}\cdot |G|.
\end{align*}

From the explicit description given in Example \ref{highest ext power exam} it thus follows that the second component of the representative (\ref{arithmetic rep}) is equal to the $\phi(1)$-st root of
\[
{\rm det}\!\left((\delta_{\hat\sigma,\hat\tau}\delta_{k,\ell}\cdot |G|)_{(\sigma,k),(\tau,\ell)}\right)^{1/2} = |G|^{[K:\QQ]\cdot\phi(1)^2/2},
\]
as suffices to give the claimed result. \end{proof}

\section{Canonical homomorphisms and the universal diagram} In this section we establish a direct link between relative algebraic $K$-theory and the theories of metrised and hermitian modules reviewed above. The existence of such a link will then play a key role in subsequent arithmetic results.

For any finite group $\Gamma$ we abbreviate ${\rm Cl}(\ZZ[\Gamma])$ to ${\rm Cl}(\Gamma)$ and we recall that there is a natural isomorphism of abelian groups

\begin{equation}\label{red isom}
h_\Gamma^{\rm red}: {\rm Cl}(\Gamma) \cong {\rm Cok}(\Delta^{\rm red}_\Gamma)
\end{equation}
where $\Delta^{\rm red}_{\Gamma}$ denotes the homomorphism
\[
\Delta^{\rm red}_{\Gamma} \colon
\Hom(R_\Gamma  ,\mathbb{Q}^{c\times})^{\Omega _\mathbb{Q}} \to
\frac{\Hom(R_\Gamma  ,J_f(\mathbb{Q}^c))^{\Omega _\mathbb{Q}} }
       {\Det(U_f(\mathbb{Z}[\Gamma]))}; \qquad \theta \mapsto [\theta].
\]

\begin{remark}\label{rep for classgroup}\mpar{rep for classgroup} We normalise the isomorphism $h_\Gamma^{\rm red}$ as in \cite[Rem. 1, p. 21]{Frohlich}. To be specific, if $X$ is a finitely generated projective $\Ze[\Gamma]$-module, then one can give
an explicit representative of the class $h^{{\rm red}}_{\Gamma}([X])$ as follows. We choose a $\Qu[\Gamma]$-basis $\{x_0^j\}$ of $X_\Qu$ and, for each rational prime $p$, a $\Zp[\Gamma]$-basis $\{ x_p^j\}$ of $X_p$. Let $\lambda_p$ be the matrix in $\Gl_d(\Qp[\Gamma])$ which satisfies
$\underline{x_p^j} = \lambda_p\cdot\underline{x_0^j}$. Then $h^{{\rm red}}_{\Gamma}([X])$  is represented by the function
$( \prod_{p} \Det(\lambda_p))$.
\end{remark}

In the next result we shall use the following canonical homomorphisms (of abelian groups)
\[
\begin{array}{ccll}
  \partial_\Gamma^{1,1} &\colon& \mathrm{Cok} (\Delta_\Gamma^{\rm rel})  \lra \mathrm{A}(\Gamma), & \hskip 0.5truein ([\theta_1], \theta_2) \mapsto ([\theta_1], |\theta_2|)\\
  \\
  \partial_\Gamma^{2,1} &\colon& \mathrm{Cok} (\Delta_\Gamma^{\rm rel})  \lra \eHCl(\Gamma), & \hskip 0.5truein ([\theta_1], \theta_2) \mapsto ([\theta_1], \theta_2^s)\\
  \\
  \partial_\Gamma^{1,2} &\colon& \mathrm{A} (\Gamma)  \lra \cok(\Delta_\Gamma^{\mathrm{red}}), & \hskip 0.5truein ([\theta_1], \theta_2) \mapsto [\theta_1]\\
  \\
\partial_\Gamma^{2,2} &\colon& \eHCl (\Gamma)  \lra \cok(\Delta_\Gamma^{\mathrm{red}}), & \hskip 0.5truein ([\theta_1], \theta_2) \mapsto [\theta_1].
\end{array}
\]

We shall also use the following composite homomorphisms (defined using the isomorphisms $h_\Gamma^{\mathrm{rel}}$ and
 $h_\Gamma^{\mathrm{red}}$)
\[
\begin{array}{lll}
   \Pi_\Gamma^{\mathrm{met}} := \partial_\Gamma^{1,1} \circ h_\Gamma^{\mathrm{rel}} & \colon &
  K_0(\Ze[\Gamma], \Qc[\Gamma]) \lra \mathrm{A}(\Gamma), \\
  \\
  \Pi_\Gamma^{\mathrm{herm}} := \partial_\Gamma^{2,1} \circ h_\Gamma^{\mathrm{rel}} &\colon&
  K_0(\Ze[\Gamma], \Qc[\Gamma]) \lra \eHCl(\Gamma), \\
 \\
 \partial_\Gamma^{\mathrm{met}} := \left( h_\Gamma^{\mathrm{red}} \right)^{-1} \circ \partial_\Gamma^{1,2} &\colon&
  \mathrm{A}(\Gamma)  \lra \Cl(\Gamma), \\
\\
\partial_\Gamma^{\mathrm{herm}} := \left( h_\Gamma^{\mathrm{red}} \right)^{-1} \circ \partial_\Gamma^{2,2} &\colon&
  \eHCl(\Gamma)  \lra \Cl(\Gamma).
\end{array}
\]

For convenience we shall use the same notation $\partial_\Gamma^{\mathrm{herm}}$ to denote the restriction of $\partial_\Gamma^{\mathrm{herm}}$ to the subgroup $\HCl(\Gamma)$.

\begin{theorem}\label{prop key diag} \
\begin{itemize}
\item[(i)] The homomorphism $\Pi_\Gamma^{\rm met}$ sends each class $[X,\xi, Y]$
to $[X,\xi^*(\mu)] - [Y,\mu]$ for any choice of metric $\mu$ on $Y$.
\item[(ii)] The homomorphism $\Pi_\Gamma^{\rm herm}$ sends each element $[X,\xi, Y]$
of the subgroup $K_0(\mathbb{Z}[\Gamma],\mathbb{Q}[\Gamma])$ to ${\rm Disc}(X,\xi^*(h)) -
{\rm Disc}(Y,h)$ for any  choice of hermitian form $h$ on $Y$.\end{itemize}
\begin{itemize}
\item[(iii)] The homomorphism $\partial_\Gamma^{\rm met}$ sends the class  $[X,h]$ of a metrised module
$(X, h)$ to the class $[X]$.

   \item[(iv)] The homomorphism $\partial_\Gamma^{\rm herm}$ sends the discriminant  $\mathrm{Disc}(X,h)$
of a hermitian module $(X, h)$ to the class $[X]$.

\item[(v)] The following diagram commutes.

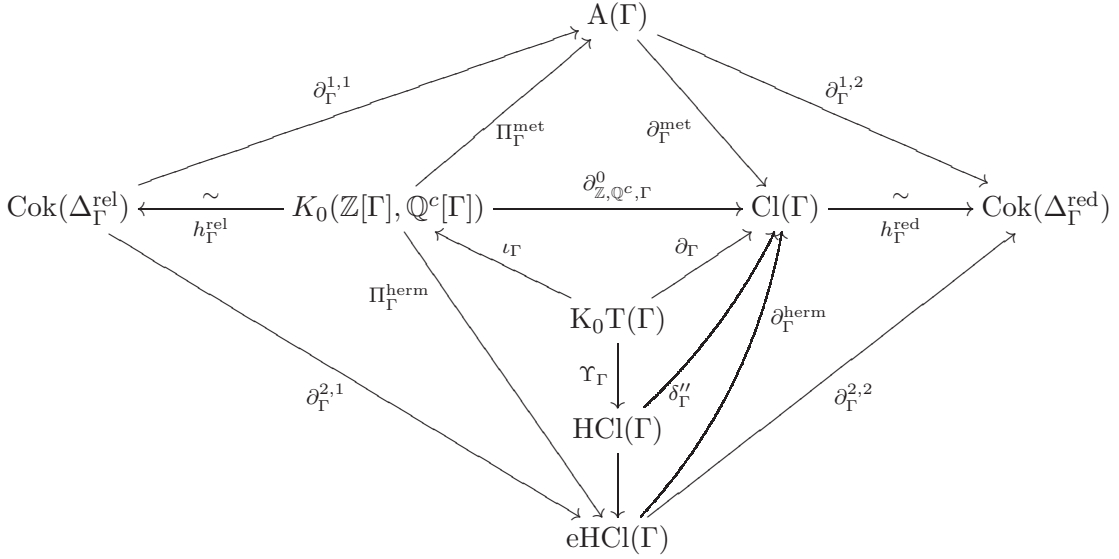
\begin{equation*}
\xymatrix{& &  & {\rm A}(\Gamma)\ar[ddr]_{\,\,\,\,\,\,\,\,\,\partial^{\rm met}_\Gamma}\ar[ddrrr]^{\partial_{\Gamma}^{1,2}}\\
& \\
{\rm Cok}(\Delta^{\rm rel}_\Gamma) \ar[dddrrr]_{\partial_{\Gamma}^{2,1}} \ar[uurrr]^{\partial_{\Gamma}^{1,1}}  & &
K_0(\mathbb{Z}[\Gamma],\mathbb{Q}^c[\Gamma])  \ar[uur]_{\!\!\!\Pi^{\rm met}_\Gamma} \ar[dddr]_<<<<<<<<{\Pi^{\rm herm}_\Gamma} \ar[ll]^-{h_\Gamma^{\rm rel}}_-{\sim} \ar[rr]^-{\partial^0_{\mathbb{Z},\mathbb{Q}^c,\Gamma}} &
&{\rm Cl}(\Gamma) \ar[rr]^-{\sim}_-{h_\Gamma^{\rm red}} & &{\rm Cok}(\Delta_\Gamma^{\rm red})\\
& & & {\rm K}_0{\rm T}(\Gamma)\ar[d]_{\Upsilon_\Gamma} \ar[ul]_{\iota_\Gamma} \ar[ur]^{\partial_\Gamma}\\
 & & & {\rm HCl}(\Gamma)\ar[d] \ar@/_/[ruu]_<<<<<{\!\!\!\!\delta''_\Gamma}\\
 & & &  {\rm eHCl}(\Gamma) \ar@/_1pc/[ruuu]_>>>>>>>>>{\partial^{\rm herm}_{\Gamma}}\ar[uuurrr]_{\partial_{\Gamma}^{2,2}}}
\end{equation*}
Here the unlabeled arrow is the natural inclusion ${\rm HCl}(\Gamma) \to {\rm eHCl}(\Gamma)$ and the remaining homomorphisms that are not defined above are as follows.
\begin{itemize}
\item  $\iota_\Gamma$ is the composition of the first isomorphism in (\ref{decomp}) and
the natural inclusion $K_0(\ZZ[\Gamma],\QQ[\Gamma]) \to K_0(\ZZ[\Gamma],\QQ^c[\Gamma])$,
\item $\Upsilon_\Gamma$ is the homomorphism defined by Fr\"ohlich in \cite[Chap. 2, \S6]{fro2},
\item $\partial_\Gamma$ is the canonical map (as described in \cite[Chap. 1, (1.3)]{fro2}) ,
\item  $\delta''_\Gamma$ the homomorphism described in \cite[Chap. 2, (6.16)]{fro2}.
\end{itemize}
(For further details of these maps see the argument below.)
 \end{itemize}
\end{theorem}

\begin{proof} Claim (i) is proved by Agboola and the first author in \cite[Th. 4.11]{AgboolaBurns}.

 To prove claim (ii)  we write $d$ for the $\QQ[\Gamma]$-rank of $X_\Qu \cong Y_\Qu$ and,
just as in Definition \ref{metrised element def},  we fix a $\QQ[\Gamma]$-basis
$\{x_0^j\}_{1\le j\le d}$ of $X_\Qu$ and also, for each prime $p$,
 $\ZZ_p[\Gamma]$-bases $\{x_p^j\}_{1\le j\le d}$ of $X_p$ and $\{y_p^j\}_{1\le j\le d}$ of $Y_p$.

We write $\lambda_p$ and $\mu_p$ for
 the (unique) elements of ${\rm GL}_d(\QQ_p[\Gamma])$ with
$\underline{x_p^j} = \lambda_p\cdot\underline{x_0^j}$ and $\underline{y_p^j} = \mu_p\cdot
  \underline{\xi(x_0^j)}$, where in the last equality we use the fact that
$\{\xi(x_0^j)\}_{1\le j\le d}$ is a $\QQ[\Gamma]$-basis of   $Y_\Qu$.

Then the explicit definition of $h_\Gamma^{\rm rel}$ as described in Lemma \ref{rep for relative group} ensures that $h_\Gamma^{\rm rel}([X,\xi, Y])$ is
represented by the pair

\[
\left(\prod_{p} \Det(\lambda_p)\cdot\Det(\mu_p)^{-1}\right) \times 1 \in \Hom (R_\Gamma ,J_f(\QQ^c))^{\Omega_F}\times
 \Hom(R_\Gamma , (\QQ^c)^{\times}).
\]

The assertion of claim (ii) thus follows because Definition \ref{pfaffian def} implies that for any
hermitian form $h$ on $X$ the element
${\rm Disc}(X,\xi^*(h)) - {\rm Disc}(Y,h)$ is also represented by

\begin{align*}
&(\prod_p {\rm Det}(\lambda_p), {\rm Pf}(\xi^*(h)(x_0^i,x_0^j)))\times (\prod_p {\rm Det}(\mu_p)^{-1}, {\rm Pf}
(h(\xi(x_0^i),\xi(x_0^j)))^{-1}) \\
 = &(\prod_p {\rm Det}(\lambda_p)\cdot{\rm Det}(\mu_p)^{-1}, {\rm Pf}(h(\xi(x_0^i),\xi(x_0^j)))
{\rm Pf}(h(\xi(x_0^i),\xi(x_0^j)))^{-1})\\
  = &(\prod_p {\rm Det}(\lambda_p)\cdot{\rm Det}(\mu_p)^{-1},1)
\end{align*}
where the first equality follows immediately from the definition of the pullback $\xi^*(h)$.

Claim (iii) and (iv) are immediate consequences of the respective Hom-descriptions of the groups $\mathrm{A}(\Gamma), \eHCl(\Gamma)$ and $\Cl(\Gamma)$.

Turning to claim (v) we note at the outset that the upper and lower left and right hand most triangles commute by definition of the maps involved and that the outer quadrilateral commutes since both of the composites $\partial^{1,2}_\Gamma\circ\partial^{1,1}_\Gamma$ and $\partial^{2,2}_\Gamma\circ\partial^{2,1}_\Gamma$ send each pair
$([\theta_1],\theta_2)$ to the class of $[\theta_1]$.

We next note that the commutativity of the upper central triangle, namely the equality $\partial_{\Ze, \Qc, \Gamma}^0
= \partial_\Gamma^{\mathrm{met}} \circ \pi_\Gamma^{\mathrm{met}}$,
will follow if we show that the composites $\partial_\Gamma^{1,2}\circ \partial_\Gamma^{1,1}\circ h^{\rm rel}_\Gamma$ and
$h^{\rm red}_\Gamma\circ \partial^0_{\IZ,\QQ^c,\Gamma}$ coincide.

This is true because the explicit description of
$h^{\rm rel}_\Gamma$ implies that $\partial_\Gamma^{1,2}\circ \partial_\Gamma^{1,1}\circ h^{\rm rel}_\Gamma$ sends each element
$[X,\xi, Y]$ of $K_0(\IZ[\Gamma],\IQ^c[\Gamma])$ to the class represented by the homomorphism
\[ \prod_p({\rm Det}(\lambda_p)\cdot {\rm Det}(\mu_p)^{-1}) = (\prod_p{\rm Det}(\lambda_p))\cdot (\prod_p{\rm Det}(\mu_p))^{-1}\]
whilst
\[
(h^{\rm red}_\Gamma\circ \partial^0_{\IZ, \Qc,\Gamma})([X,\xi, Y])
= h^{\rm red}_\Gamma([X]-[Y]) = h^{\rm red}_\Gamma([X])h^{\rm red}_\Gamma([Y])^{-1}
\]
and Remark \ref{rep for classgroup} implies that the classes $h^{\rm red}_\Gamma([X])$ and $h^{\rm red}_\Gamma([Y])$ are respectively represented by the products $\prod_p{\rm Det}(\lambda_p)$ and $\prod_p{\rm Det}(\mu_p)$.

The above facts combine to directly imply commutativity of the lower central triangle, namely the equality $\partial_{\Ze, \Qc, \Gamma}^0
= \partial_\Gamma^{\mathrm{herm}} \circ \pi_\Gamma^{\mathrm{herm}}$, and so it only suffices to prove commutativity of the four triangles inside this triangle.

We shall now discuss these triangles clockwise, starting from the uppermost.

The commutativity of the first triangle follows directly from the fact that for any finite $\Gamma$-module $M$ of finite projective dimension, and any resolution of the form
\[ 0 \to P\xrightarrow{\theta} P' \to M \to 0,\]
where $P$ is finitely generated and locally-free and $P'$ is finitely generated and free, the class of $M$ in $K_0T(\ZZ[\Gamma])$ is sent by $\iota_\Gamma$ to $[P,\theta_\QQ,P']$ and by $\partial_\Gamma$ to $[P]-[P']$ $(= [P]$ as $P'$ is free).

If for the above sequence we fix a $\ZZ[\Gamma]$-basis $\{x^i\}$ of $P'$ and then for each prime $p$ choose a matrix $\lambda_p$
in ${\rm GL}_d(\QQ_p[\Gamma])$ so that the components of the vector $\lambda_p\cdot \underline{\theta_\QQ^{-1}(x^i)}$ are a $\ZZ_p[\Gamma]$-basis of $P_p$,
then the image of the class of $M$ in $K_0T(\ZZ[\Gamma])$ under $\Upsilon_\Gamma$ is represented by $(\prod_p{\rm Det}(\lambda_p),1)$.
This implies the commutativity of the second triangle since Remark \ref{rep for classgroup} implies the class of $\partial_\Gamma(M) = [P]$
is represented by $(\prod_p{\rm Det}(\lambda_p))$
whilst the definition of $\delta_\Gamma''$ implies that it is induced by sending each pair $([\theta_1],\theta_2)$ to $(h_\Gamma^{\rm red})^{-1}([\theta_1])$.

The latter fact also directly implies commutativity of the third triangle and the fourth triangle commutes since, in terms of the above notation,
the composite $h_\Gamma^{\rm rel}\circ \iota_\Gamma$ sends the class of $M$ to the element represented by the pair $((\prod_p{\rm Det}(\lambda_p)),1)$.\end{proof}

In the next result we describe an explicit link between the elements in relative algebraic $K$-theory constructed in \S\ref{example0}, the hermitian modules described in Example \ref{example2} and the metrised modules defined in Example \ref{example1}.

This link explains the relevance of Theorem \ref{prop key diag} to our later results.

\begin{prop}\label{almost there} Let $L/K$ be a finite Galois extension of number fields with group $G$. Then for any full projective $\mathcal{O}_K[G]$-submodule $\mathcal{L}$ of $L$ the following claims are valid.

\begin{itemize}
\item[(i)] The image of $[\mathcal{L}, \kappa_L, H_L]$ under $\Pi_G^{\rm met}$ is equal to $[\mathcal{L},h_{L,\bullet}]-[H_L,\mu_{L,\bullet}]$.

\item[(ii)] The image of $[\mathcal{L}, \kappa_L, H_L]$ under $\Pi_G^{\rm herm}$ is equal to ${\rm Disc}(\mathcal{L},t_{L/K})$.
\end{itemize}
\end{prop}

\begin{proof} The pullback with respect to $\kappa_L$ of the metric $\mu_{L,\bullet}$ defined in Example \ref{example1}
is equal to $h_{L,\bullet}$ (cf. \cite[Exam. 4.10(i)]{AgboolaBurns}). This fact combines with Theorem \ref{prop key diag}(i) to directly imply the equality in claim (i).

To prove claim (ii) we use the representative $(\theta_1\theta_2^{-1}, \theta_2 \theta_3)$ of $h_G^{\mathrm{rel}}([\calL, \kappa_L, H_L])$ described in Lemma \ref{norm resolvent lemma}. We also recall that, with this notation, the general result of Fr\"ohlich in \cite[Cor. to Th. 27]{fro2} implies the element ${\rm Disc}(\mathcal{L},t_{L/K}) - {\rm Disc}(\mathcal{O}_K[G],t_{K[G]})$ of ${\rm HCl}(G)$ is represented by  $(\theta_{1}\cdot\theta_2^{-1},\theta_2^s)$, where the form $t_{K[G]}$ is as defined in Example \ref{example3}.

Comparing these results one deduces that the element
\[
\Pi_G^{\rm herm}([\mathcal{L},\kappa_L, H_L]) - {\rm Disc}(\mathcal{L},t_{L/K}) + {\rm Disc}(\mathcal{O}_K[G],t_{K[G]})
\]
of ${\rm HCl}(G)$ is represented by the pair $(1,\theta_3^s)$.

To deduce claim (ii) from this it is thus enough to show that the pair
$(1,\theta_3^s)$ also represents the element ${\rm Disc}(\mathcal{O}_K[G],t_{K[G]})$.

To check this we need only note that, in the terminology of \cite[Chap. II, \S5]{fro2},
the Pfaffian of the matrix $(t_{K[G]}(a_\sigma,a_\tau))_{\sigma,\tau \in \Sigma(K)}$ sends each character
$\chi$ in $R_G^s$ to $\delta_K^{\chi(1)} = \theta_3(\chi)$.

Then, by applying the recipe of Definition \ref{pfaffian def} with $\{ x_0^j\} = \{ x_p^j \} = \{ a_\sigma\}_{\sigma \in \Sigma(K)}$ one finds
that ${\rm Disc}(\mathcal{O}_K[G],t_{K[G]})$ is indeed represented by the pair $(1,\theta_3^s)$, as required.
\end{proof}

\section*{\large{Part II: Weak ramification and Galois-Gauss sums}}

In this part of the article we describe a first arithmetic application of the approach described in earlier sections by using Theorem \ref{prop key diag} (and Proposition \ref{almost there}) to refine existing results concerning links between Galois-Gauss sums and certain metric and hermitian structures that arise naturally in arithmetic.

In this way, in \S\ref{tame section} we refine the main results of Chinburg and the second author in \cite{BurnsChinburg}
concerning relations between hermitian-metric structures involving fractional powers of the inverse different of a tamely ramified
Galois extension of number fields and the associated Galois-Gauss sums (twisted by appropriate Adams operations).

In the remainder of the article we then focus on weakly ramified Galois extensions (of odd degree) and use Theorem \ref{prop key diag} to refine key aspects of the extensive existing theory of the square root of the inverse different for such extensions.

\section{Tamely ramified Galois-Gauss sums}\label{tame section}

\subsection{Galois-Gauss sums, Adams operators and Galois-Jacobi sums}\label{general cons}\mpar{general cons} For the reader's convenience in this section we fix notation regarding various variants of Galois-Gauss sums that will play a role in the sequel.

To do this we fix an arbitrary finite Galois extension $L/K$ of number fields in $\mathbb{Q}^c$ and set $G :=G({L/K})$.

For each character $\chi$ in $\widehat{G}$ we obtain a primitive central idempotent of $\Qu^c[G]$ by setting
\[
e_\chi := \frac{\chi(1)}{|G|} \sum_{g \in G} \chi(g)g^{-1}.
\]

We use the fact that each element of $\zeta(\QQ^c[G])$ can then be written uniquely in the form
\begin{equation}\label{explicit decomp} x = \sum_{\chi \in \widehat G}e_\chi\cdot x_\chi\end{equation}
with each $x_\chi$ in $\QQ^c$.

For convenience we extend the assignment $x \mapsto x_\chi$ to arbitrary elements $\chi$ of $R_G$ by multiplicativity.

\subsubsection{}We define the `equivariant global Galois-Gauss sum' for $L/K$ by setting
\[ \tau_{L/K} := \sum_{\chi \in \widehat G}e_\chi\cdot\tau(K,\chi) \in \zeta(\QQ^c[G])\]
where each (global) Galois-Gauss sum $\tau(K,\chi)$ belongs to $\QQ^c$ and is as defined, for example, by Fr\"ohlich in \cite[Ch. I, (5.22)]{Frohlich}.

We also define an `equivariant unramified characteristic' in $\zeta(\QQ[G])$ by setting
\[ y_{L/K} :=  \sum_{\chi \in \widehat G}e_\chi\cdot \prod_{v \mid d_L}y(K_v,\chi_v).\]
Here $\chi_v$ is the restriction of $\chi$ to the decomposition subgroup of some fixed place $w$ of $L$ above $v$ and (following \cite[Ch. IV, \S 1]{Frohlich}) for any finite Galois extension of local fields $F/E$ of group $D$ and each $\phi$ in $\widehat D$ we set

\mpar{explicit unram char}
\begin{equation}\label{explicit unram char}
y(E, \phi) :=
\begin{cases}
  1, & \text{ if } \phi |_{I} \ne 1, \\ - \phi(\sigma), & \text{ if } \phi |_{I} = 1,
\end{cases}
\end{equation}
where $I$ is the inertia subgroup of $D$ and $\sigma$ is a lift to $D$ of the Frobenius element in $D/I$.

We then define the `modified equivariant (global) Galois-Gauss sum' for $L/K$ by setting
\[ \tau'_{L/K} := \tau_{L/K}\cdot y_{L/K}^{-1}.\]


Since we rely on certain results from \cite{BleyBurns} we will also use the `absolute (global) Galois-Gauss sum for $L/K$' that is obtained by setting
\[
\tau^\dagger_{L/K} := \sum_{\chi\in \hat{G}} e_\chi \tau(\Qu,{\rm ind}^{\Qu}_K\chi) \in \zeta(\Qc[G])^\times.
\]

In particular, it is useful to note that the inductivity property of Galois-Gauss sums combines with the fact $\tau(K, 1_K) = 1$ to imply
\mpar{dagger tau}
\begin{equation}\label{dagger tau}
\tau^\dagger_{L/K} = \tau_K^G\cdot \tau_{L/K}\end{equation}
where $\tau^G_K$ is the invertible element of $\zeta(\mathbb{Q}^c[G])$ obtained by setting
\[
\tau^G_K :=\text{Nrd}_{\mathbb{Q}[G]}(\tau(\mathbb{Q},\text{ind}_K^\mathbb{Q}{\bf 1}_{K}))\]
so that $(\tau^G_K)_\chi = \tau(\mathbb{Q},\text{ind}_K^\mathbb{Q}{\bf 1}_{K})^{\chi(1)}$ for all $\chi$ in $\widehat G$.

\subsubsection{}For each integer $k$ that is coprime to $|G|$ we write $\psi_k$ for the $k$-th Adams operator on $R_G$ (for the relevant properties of which we refer to \cite[Lem. 3.1]{BurnsChinburg}).

We use this operator to construct endomorphisms of $\zeta(\QQ^c[G])$ in the following way. For each pair of integers $m$ and $n$
we  write $(m+n\cdot \psi_{k,*})(x)$ for the unique element of $\zeta(\QQ^c[G])$ with $(m+n\cdot \psi_{k,*})(x)_\chi := (x_\chi)^m\cdot (x_{\psi_k(\chi)})^n$ for every $\chi$ in $\widehat G$.

We then define the `$k$-th Galois-Jacobi sum' for the extension $L/K$ by setting
\[ J_{k,L/K} := (\psi_{k,\ast} - k)(\tau_{L/K}). \]

In the sequel we shall often use the following key property of these sums.

\begin{lemma}\label{jacobi lemma} For each integer $k$ prime to $|G|$ one has $J_{k,L/K}\in \zeta(\IQ[G])^\times$.\end{lemma}

\begin{proof} An element $x$ of $\zeta(\IQ^c[G])$ belongs to $\zeta(\IQ[G])$ if and only if one has
$(x_\chi)^\omega=x_{\chi^\omega}$ for all $\chi\in \hat{G}$ and all $\omega\in \Omega_{\IQ}$.

To verify that the elements $J_{k,L/K}$ satisfy this criterion we recall how the absolute Galois group acts on Gauss sums. We let
$\mathrm{Ver}_{K/\Qu} \colon \Omega_\Qu^{ab} \lra  \Omega_K^{ab}$ denote the transfer map and write $v_{K/\Qu}$ for the cotransfer
map from abelian characters of $\Omega_K$ to abelian characters of $\Omega_\Qu$. Thus, for each $\chi \in \widehat G$ the map
$v_{K/\Qu}\det_\chi$ is an abelian character of $\Omega_\Qu$. Then, by  \cite[Th.~20B(ii)]{Frohlich}, one has $\tau(K, \chi^{\omega^{-1}})^\omega = \tau(K, \chi) \cdot \left( v_{K/\Qu} \det\nolimits_\chi \right) (\omega)$ for all $\chi \in \widehat G$ and all
$\omega \in \Omega_\Qu$.

Hence it suffices to show that $\left(  \left( v_{K/\Qu} \det_\chi \right) (\omega) \right)^k = \left( v_{K/\Qu} \det_{\psi_k(\chi)} \right) (\omega)$ and this is true because $\det_{\psi_k(\chi)} = (\det_{\chi})^k$ (see \cite[Lem. 3.1]{BurnsChinburg}).\end{proof}

{With the results of \cite{BleyBurns} in mind we finally note} that if $G$ has odd order,
then an explicit comparison of the respective definitions shows that
\mpar{label product}
\begin{equation}\label{label product}
\tau^G_K\cdot (\psi_{2,\ast}-1)(\tau'_{L/K})\cdot (\tau^\dagger_{L/K})^{-1} = J_{2,L/K}\cdot (\psi_{2,\ast}-1)(y_{L/K}^{-1}).
\end{equation}

\begin{remark}\label{local analogue} If $F/E$ is a finite Galois extensions of $p$-adic fields (for some $p$) with group $D$,
then one can use the canonical local Gauss sum $\tau(E,\phi)$ (as discussed, for example,
in \cite[Ch. III, \S2, Th. 18 and Rem. 1]{Frohlich}) for each $\phi$ in $\widehat D$ to define natural analogues $\tau_{F/E}, y_{F/E}, \tau'_{F/E}, \tau^\dagger_{F/E}, \tau_E^D$ and $J_{k,F/E}$ in $\zeta(\QQ^c[D])$ of the elements defined above.
Then in the same way as above one can show that for each integer $k$ that is coprime to $|D|$ the element $J_{k,F/E}$
belongs to $\zeta(\QQ[D])^\times$ and can also prove the following local analogues of the equalities (\ref{dagger tau}) and (\ref{label product}):
\mpar{local dagger tau}
\begin{equation}\label{local dagger tau}
\tau^\dagger_{F/E} = \tau_E^D\cdot \tau_{F/E}
\end{equation}
and
\mpar{local label product}
\begin{equation}\label{local label product}
\tau^D_E\cdot (\psi_{2,\ast}-1)(\tau'_{F/E})\cdot (\tau^\dagger_{F/E})^{-1} = J_{2,F/E}\cdot (\psi_{2,\ast}-1)(y_{F/E}^{-1}).
\end{equation}
\end{remark}

\subsubsection{}In the next result we write $W_{L/K}$ for the so-called `Cassou-Nogu\`es-Fr\"ohlich root number class' in ${\rm Cl}(G)$.

We recall that this element plays a critical role in classical Galois module theory (as discussed by Fr\"ohlich in \cite{Frohlich} and \cite{fro2}).

\begin{lemma}\label{cnf lift} There exists a canonical element $W_{L/K}^{\rm rel}$ of $K_0(\ZZ[G],\QQ^c[G])$ that has all of the following properties.
\begin{itemize}
\item[(i)] The image of $W_{L/K}^{\rm rel}$ under the connecting homomorphism $\partial^0_{\ZZ,\QQ^c,G}$ is $W_{L/K}$.
\item[(ii)] $W_{L/K}^{\rm rel}$ is trivial if the Artin root number of each symplectic character in $\widehat{G}$ is positive.
\item[(iii)] In all cases the element $2\cdot W_{L/K}^{\rm rel}$ is trivial.
\end{itemize}
\end{lemma}

\begin{proof} The element $W_{L/K}$ is defined directly in terms of the Artin root numbers of symplectic characters in $\widehat{G}$ by means of the isomorphism $h_G^{\rm red}$ in (\ref{red isom}).

One can use the isomorphism $h_G^{\rm rel}$ in (\ref{rel isom}) to define $W_{L/K}^{\rm rel}$ in a similarly explicit way. However, for later purposes, it is useful to adopt a different approach to the definition of $W_{L/K}^{\rm rel}$.

To do this we recall the element $\epsilon_{L/K}$ of $\zeta(\RR[G])^\times$ that is defined in terms of epsilon constants in \cite[just after (9)]{BleyBurns}.

Then, in view of the description of ${\rm im}({\rm Nrd}_{\RR[G]})$ that is given by the Hasse-Schilling-Maass Norm Theorem, we can use the Weak Approximation Theorem to choose an element $\lambda$ of $\zeta(\QQ[G])^\times$ with the property that $\lambda\cdot\epsilon_{L/K}$ belongs to ${\rm im}({\rm Nrd}_{\RR[G]})$.

We then obtain an element of $K_0(\ZZ[G],\QQ^c[G])$ by setting
\[ W_{L/K}^{\rm rel} := \delta_G(\lambda) - \sum_p\delta_{G,p}(\lambda)\]
where $p$ runs over all primes and each $\delta_{G,p}(\lambda)$ is regarded as an element of $K_0(\ZZ[G],\QQ^c[G])$ by means of the composite
inclusion
\[ K_0(\ZZ_p[G],\QQ_p[G]) \subset K_0(\ZZ[G],\QQ[G]) \subset K_0(\ZZ[G],\QQ^c[G]).\]

This recipe is independent of the choice of $\lambda$ since if $\lambda'$ is any choice, then $\lambda^{-1} \lambda'$ belongs to ${\rm im}
({\rm Nrd}_{\QQ[G]})$ and so one has
\begin{align*}\delta_G(\lambda')-\delta_G(\lambda) =\, &\delta_G(\lambda^{-1}\lambda')\\
 =\, &(\partial^1_{\ZZ,\QQ,G}\circ ({\rm Nrd}_{\QQ[G]})^{-1})(\lambda^{-1}\lambda')\\
 =\, &\sum_p(\partial^1_{\ZZ_p,\QQ_p,G}\circ ({\rm Nrd}_{\QQ_p[G]})^{-1})(\lambda^{-1}\lambda')\\
 =\, & \sum_p\delta_{G,p}(\lambda') - \sum_p\delta_{G,p}(\lambda).\end{align*}

Given this definition of $W_{L/K}^{\rm rel}$, the property in claim (i) follows directly from the argument of \cite[Prop. 3.1]{BleyBurns}.

In addition, claim (ii) is true because the given hypotheses imply that $\epsilon_{L/K}$ belongs to ${\rm im}({\rm Nrd}_{\RR[G]})$ so that one can compute $W_{L/K}^{\rm rel}$ by using the element $\lambda = 1$.

Finally, claim (iii) follows easily from the fact that the square of any element of $\zeta(\QQ[G])^\times$ belongs to ${\rm im}({\rm Nrd}_{\QQ[G]})$.
\end{proof}

\subsection{Tame Galois-Gauss sums and fractional powers of the different}\label{tame ggs section} We now assume the Galois extension $L/K$ is tamely ramified
and fix a natural number $k$ that is both coprime to $|G|$ and so that the order of each inertia subgroup of
$G$ is congruent to $1$ modulo $k$.

In any such case it follows immediately from Hilbert's formula for the different in terms of ramification invariants (cf. \cite[Ch.~IV, Prop.~4]{LocalFields}) that there exists a unique fractional ideal $\mathfrak D_{L/K}^{-1/k}$ of $\mathcal O_L$ whose $k$-th power is equal to the inverse of the different $\mathfrak{D}_{L/K}$ of $L/K$ and for any integer $i$
we set $\mathfrak D_{L/K}^{-i/k} = (\mathfrak D_{L/K}^{-1/k})^i$.

Each ideal $\mathfrak D_{L/K}^{-i/k}$ is stable under the natural action of $\mathcal{O}_K[G]$ and, since $L/K$ is assumed to be tamely ramified, the $\mathcal{O}_K[G]$-module $\mathfrak D_{L/K}^{-i/k}$ is known to be locally-free (by Ullom \cite{U}).

In particular, since $\mathfrak D_{L/K}^{-i/k}$ is a full sublattice of $L$, the construction of \S\ref{example0} gives rise to a well-defined
element $[\mathfrak D_{L/K}^{-i/k},\kappa_L, H_L]$ of $K_0(\mathbb{Z}[G],\mathbb{Q}^c[G])$.

We next write $\psi_k^\vee$ for the map which sends a function $h$ on $R_G$
to $h\circ \psi_k$ and recall that, since $k$ is prime to $|G|$, Cassou-Nogu\`es and Taylor have in \cite{cn-t} shown that
the assignment 
\[ (\theta,\theta') \mapsto (\psi_k^\vee(\theta),\psi_k^\vee(\theta'))\]
induces
(via the map (\ref{Delta rel}) and isomorphism (\ref{rel isom})) a well-defined
endomorphism $\Psi_k$ of the group $K_0(\ZZ[G],\QQ^c[G])$.

We can now state the main result of this section. This result uses the invertible elements $\tau_K^G$ and $\tau'_{L/K}$ of $\zeta(\QQ^c[G])$ that are defined in \S\ref{general cons} as well as the relative Cassou-Nogu\`es-Fr\"ohlich root number class $W_{L/K}^{\rm rel}$ defined in Lemma \ref{cnf lift}.
%


\begin{theorem}\label{thm tame}
Let $L/K$ be a tamely ramified Galois extension of number fields with group $G$
and $k$ any natural number that is both coprime to $|G|$ and such that the order of each inertia subgroup of
$G$ is congruent to $1$ modulo $k$. Then in $K_0(\ZZ[G],\QQ^c[G])$ one has
\mpar{tame result}
\begin{equation}\label{tame result}
\sum_{i=0}^{k-1}[\mathfrak D_{L/K}^{-i/k},\kappa_L, H_L] = \delta_G((\tau_K^G)^{k}\cdot\psi_{k,\ast}(\tau'_{L/K})) + \Psi_k(W_{L/K}^{\rm rel}).
\end{equation}
\end{theorem}

Before proving this result we use it to derive certain explicit consequences concerning the metric and hermitian structures that arise in this setting.

In particular, the following result extends to all integers $k$ as in Theorem \ref{thm tame} the results on the hermitian modules $(\mathcal{O}_{L},t_{L/K})$,
corresponding to $k=1$, and $(\mathfrak{D}_{L/K}^{-1/2},t_{L/K})$, corresponding to $k=2$ and $G$ of odd order,
that are obtained by Erez and Taylor in \cite{ErezTaylor}.

We recall the definition of the element $\delta_K$ from Lemma \ref{norm resolvent lemma} and write $d_K$ for the discriminant of $\OK$.

In the sequel we will often use the fact that
\mpar{disc pfaff}
\begin{equation}\label{disc pfaff}
\tau(\mathbb{Q},\text{ind}_K^\mathbb{Q}{\bf 1}_{K})^2 = d_K = \delta_K^2,
\end{equation}
as follows by combining \cite[Th. (11.7)(iii)]{Neukirch} together with \cite[(5.23)]{Frohlich}.

\begin{coro}\label{coro 1} Assume the notation and hypotheses of Theorem \ref{thm tame}. Then both of the following claims are valid.

\begin{itemize}
\item[(i)] In ${\rm A}(G)$ one has
\[
\sum_{i=0}^{k-1}[\mathfrak D_{L/K}^{-i/k},h_{L,\bullet}] = \varepsilon^{\rm met}_{L/K,k} + \Pi_G^{\rm met}(\Psi_k(W_{L/K}^{\rm rel}))\]
where  $h_{L,\bullet}$ is the metric defined in Example \ref{example1} and $\varepsilon^{\rm met}_{L/K,k}$
is represented by the pair $(1, |\theta_k|)$ with
$\theta_k(\phi)= (|G|^{[K:\IQ]} {|d_K|})^{k\frac{\phi(1)}{2}} \cdot  {\tau}(K,\psi_k(\phi))$  for all $\phi$ in $R_G$.

\item[(ii)] In ${\rm HCl}(G)$ one has
\[ \sum_{i=0}^{k-1}{\rm Disc}(\mathfrak D_{L/K}^{-i/k},t_{L/K}) = \varepsilon^{\rm herm}_{L/K,k} + \Pi_G^{\rm herm}(\Psi_k(W_{L/K}^{\rm rel}))\]
where the hermitian form $t_{L/K}$ is as defined in Example \ref{example2} and $\varepsilon^{\rm herm}_{L/K,k}$ is represented by the pair $(1,\tilde\theta_k)$ with $\tilde\theta_k(\phi)= d_K^{k\frac{\phi(1)}{2}}\cdot  {\tau}(K,\psi_k(\phi))$ for all $\phi$ in $R^s_G$.
\end{itemize}
\end{coro}

\begin{proof} To prove claim (i) we note first Proposition \ref{almost there}(i) implies that for each $i$ one has
\[
\Pi_G^{\rm met}([\mathfrak D_{L/K}^{-i/k}, \kappa_L, H_L]) = [\mathfrak D_{L/K}^{-i/k},h_{L,\bullet}]-[H_L,\mu_{L,\bullet}].\]

We next recall that for $\alpha = \left( \alpha_\chi \right)_{\chi \in \widehat G} \in \zeta(\Qc[G])^\times$ the element $h_G^{\mathrm{rel}}(\delta_G(\alpha))$ is represented
by the function $\chi \mapsto (1, \alpha_\chi)$.

This implies, in particular, that $h^{\rm rel}_G(\delta_G((\tau_K^G)^k\cdot\psi_{k,\ast}(\tau'_{L/K})))$ is represented by the pair
$(1,\theta_k')$ with
$\theta_k'(\phi) := \tau(\mathbb{Q},{\rm ind}^\mathbb{Q}_K{\bf 1}_K)^{k\phi(1)}\cdot \tau '(K,\psi_k(\phi))$ for each $\phi$ in $\widehat G$.

Finally we recall that the element $[H_L,\mu_{L,\bullet}]$ has been explicitly computed in Lemma \ref{betti example}.

Putting these facts together with the equality in Theorem \ref{thm tame} one finds that the element

\begin{align*} &\sum_{i=0}^{k-1}[\mathfrak D_{L/K}^{-i/k},h_{L,\bullet}] -\Pi_G^{\rm met}(\Psi_k(W_{L/K}^{\rm rel}))\\
 = \,\,& k\cdot [H_L,\mu_{L,\bullet}] + \Pi_G^{\rm met}\bigl(\sum_{i=0}^{k-1}
[\mathfrak D_{L/K}^{-i/k},\kappa_L, H_L] - \Psi_k(W_{L/K}^{\rm rel})\bigr)\\
= \,\, & k\cdot [H_L,\mu_{L,\bullet}] + \Pi_G^{\rm met}(\delta_G((\tau_K^G)^k\cdot\psi_{k,\ast}(\tau'_{L/K})))\end{align*}
of $A(G)$ is represented by the homomorphism pair $(1,|\theta_k|)$ where for each $\phi$ in $\widehat G$ one has
\[
\theta_k(\phi) := |G|^{[K:\IQ]\frac{k\phi(1)}{2}} \cdot \tau(\IQ,{\rm ind}^\IQ_K{\bf 1}_K)^{k\phi(1)}\cdot \tau'(K,\psi_k(\phi)).
\]
But, taking account of both (\ref{disc pfaff}) and the fact that $y(K_v,\phi_v)$ is a root of unity for all $\phi$ in $\widehat G$,
one finds that
\[
|\theta_k|(\phi) = (|G|^{[K:\IQ]} {|d_K|})^{k\frac{\phi(1)}{2}} \cdot |\tau(K,\psi_k(\phi))|
\]
and this proves claim (i).

It is enough to prove the equality of claim (ii) in ${\rm eHCl}(G)$ and to do this we note that the description in
Proposition \ref{almost there}(ii) combines with Theorem \ref{thm tame} to imply that
\begin{align*}
&\sum_{i=0}^{k-1}{\rm Disc}(\mathfrak D_{L/K}^{-i/k},t_{L/K})-\Pi_G^{\rm herm}(\Psi_k(W_{L/K}^{\rm rel}))\\
 =\, &\Pi_G^{\rm herm}\bigl( \sum_{i=0}^{k-1}[\mathfrak D_{L/K}^{-i/k},\kappa_L,H_L] - \Psi_k(W_{L/K}^{\rm rel})\bigr)\\
 =\, &\Pi_G^{\rm herm}(\delta_G((\tau_K^G)^k\cdot\psi_{k,\ast}(\tau'_{L/K}))).
\end{align*}

In addition, by the definition of $\Pi_G^{\mathrm{herm}}$ one has the following equality in ${\rm eHCl}(G)$
\[
\Pi^{\rm herm}_G(\delta_G((\tau_K^G)^k\cdot\psi_{k,\ast}(\tau'_{L/K}))) = (\partial^{2,1}_G\circ h^{\rm rel}_G)(\delta_G((\tau_K^G)^k\cdot\psi_{k,\ast}(\tau'_{L/K}))).
\]

Hence one deduces that the difference
\[ \sum_{i=0}^{k-1}{\rm Disc}(\mathfrak D_{L/K}^{-i/k},t_{L/K}) -\Pi_G^{\rm herm}(\Psi_k(W_{L/K}^{\rm rel}))\]
is represented by the pair $(1,(\theta'_k)^s)$, where $\theta_k'$ is as defined in the proof of claim (i).

To deduce claim (ii) from this it is now enough to note that for $\phi$ in $R_G^s$ one has
\begin{equation}\label{almost there0}
\theta_k'(\phi) = \tau(\IQ,{\rm ind}^\IQ_K{\bf 1}_K)^{k\phi(1)}\cdot \tau'(K,\psi_k(\phi)) = d_K^{k\phi(1)/2}\cdot  {\tau}(K,\psi_k(\phi)),
\end{equation}
where, to derive the second equality, we have used (\ref{disc pfaff}) and the fact that for
every $\phi$ in $R_G^s$ the integer $\phi(1)$ is even and $y(K_v,\phi_v)=1$ by \cite[Th.~29 (i)]{Frohlich}.
\end{proof}

In the remainder of this section we shall prove Theorem \ref{thm tame} by combining results of the first and second author from \cite{BleyBurns} and of Chinburg and the second author from \cite{BurnsChinburg}.

To do this we fix a $K[G]$-generator $b$ of $L$ and a $\Ze$-basis $\{a_\sigma\}_{\sigma \in \Sigma(K)}$ of $\mathcal{O}_K$.
For each integer $i$ with $0 \le i < k$ and each non-archimedean place $v$ of $K$ we also fix an $\mathcal{O}_{K,v}[G]$-generator $b_{i,v}$ of
$( \mathfrak D_{L/K}^{-i/k})_v$.

Then by Lemma \ref{norm resolvent lemma} the element
$h^{\rm rel}_G([\mathfrak D_{L/K}^{-i/k}, \kappa_L, H_L])$ is represented by the pair of homomorphisms
$(\theta_{i,1}\cdot\theta^{-1}_{2},\theta_2\cdot\theta_{3})$ where for each $\chi$ in $R_G$ one has
\begin{equation}\label{homs def}
\theta_{i,1}(\chi) :=
\prod_v\mathcal{N}_{K/\IQ}(b_{i,v}\mid \chi), \,\,\, \theta_2(\chi) := \mathcal{N}_{K/\IQ}(b\mid \chi), \,\,\,\theta_3(\chi) := \delta_K^{\chi(1)}.
\end{equation}

With this notation, it is straightforward to check that
\[
(\theta_2^{-k},\theta_2^k) \equiv (\psi_k^\vee(\theta_2)^{-1},\psi_k^\vee(\theta_2)) \,\,\,({\rm mod}\,\, {\rm im}(\Delta_G^{\rm rel})),
\]
(see, for example, the end of the proof of \cite[Prop. 3.3]{BurnsChinburg}) and it is also clear $\psi_k^\vee(\theta_3) = \theta_3$.

In particular, if we denote the sum on the left hand side of (\ref{tame result}) by $\Sigma_k$, then these observations
combine with the above description of each element $h^{\rm rel}_G([\mathfrak D_{L/K}^{-i/k}, \kappa_L, H_L])$ and the congruence proved in Lemma \ref{adapted cong} below to imply that $h_G^{\rm rel}(\Sigma_k)$ is represented by the pair
\[
(\psi_k^\vee(\theta_{0,1}\cdot\theta_2^{-1}),\psi_k^\vee(\theta_2\cdot\theta^k_3)) =
(\psi_k^\vee(\theta_{0,1}\cdot\theta_2^{-1}),\psi_k^\vee(\theta_2\cdot\theta_3))\cdot (1, \theta^{k-1}_3).
\]

It follows that, writing $x_k$ for the element of $K_0(\ZZ[G],\QQ^c[G])$ for which $h_G^{\rm rel}(x_k)$ is represented
by the pair $(1,\theta^{k-1}_3)$, one has
\[ \Sigma_k = \Psi_k([\mathcal{O}_L,\kappa_L, H_L]) + x_k.\]

We claim next that the results of \cite{BleyBurns} imply that
\begin{equation}\label{bb eq} [\mathcal{O}_L,\kappa_L, H_L] = \delta_G(\tau_K^G\cdot  \tau_{L/K}')+ W_{L/K}^{\rm rel}.\end{equation}

Before proving this equality we note that, if true, it would combine with the previous equality to imply that the element
$h_G^{\rm rel}(\Sigma_k - \Psi_k(W_{L/K}^{\rm rel}))$ is represented by the homomorphism pair $(1,\theta_3^{k-1}\cdot\theta_3'\cdot\psi_k^\vee(\theta_4))$,
 where for each $\chi$ in $R_G$ one has
$\theta_3'(\chi) = \tau(\mathbb{Q},\text{ind}_K^\mathbb{Q}{\bf 1}_{K})^{\chi(1)}$ and $\theta_4(\chi) = (\tau'_{L/K})_\chi$.

On the other hand, from (\ref{disc pfaff}) one has
$\tau(\mathbb{Q},\text{ind}_K^\mathbb{Q}{\bf 1}_{K}) = \pm \delta_K$ so that
\[ (1,\theta_3) \equiv (1,\theta_3')  \,\,\,({\rm mod}\,\, {\rm im}(\Delta_G^{\rm rel}))\]
and hence $h_G^{\rm rel}(\delta_G((\tau_K^G)^{k}\cdot\psi_{k,\ast}(\tau'_{L/K})))$ is also represented by the pair
$(1,\theta_3^{k-1}\cdot\theta_3'\cdot\psi_k^\vee(\theta_4))$.

It would thus follow that $\Sigma_k - \Psi_k(W_{L/K}^{\rm rel}) = \delta_G((\tau_K^G)^{k}\cdot\psi_{k,\ast}(\tau'_{L/K}))$, as claimed.

To complete the proof of Theorem \ref{thm tame} it is therefore enough to prove (\ref{bb eq}). To do this we note that the notation $\mathcal{E}_{L/K}$ introduced in \cite[\S3.1]{BleyBurns} denotes the element
\[ \delta_G(\lambda\cdot \epsilon_{L/K}) - \sum_p \delta_{G,p}(\lambda) = \delta_G(\epsilon_{L/K}) + W_{L/K}^{\rm rel}\]
of $K_0(\ZZ[G],\QQ^c[G])$, where $\lambda$ in $\zeta(\QQ[G])^\times$ is chosen as in the proof of Lemma \ref{cnf lift}.
 Note that here and in the sequel, to be able to apply the results of \cite{BleyBurns} we are implicitly working 
in the group $K_0(\ZG, \CC[G])$, regarding both $K_0(\ZG, \RR[G])$ and $K_0(\ZG, \Qc[G])$ as subgroups in the obvious way.

The definition of the element $\delta_{L/K}(\mathcal{O}_L)$ of $K_0(\ZZ[G],\QQ^c[G])$ given in \cite[\S3.2]{BleyBurns} ensures that
\begin{align*}\delta_G(\tau^\dagger_{L/K}) - [\mathcal{O}_L,\kappa_L,H_L] =\, &\delta_G(\epsilon_{L/K}) - \delta_{L/K}(\mathcal{O}_L)\\
=\, &\mathcal{E}_{L/K} - \delta_{L/K}(\mathcal{O}_L) - W_{L/K}^{\rm rel}\\
=\, &\delta_G(y_{L/K}) - W_{L/K}^{\rm rel}.\end{align*}
The first equality here is a consequence of  \cite[Rem. 3.5]{BleyBurns} and the fact that the element $\tau_{L/K}$ and map $\rho_L$ in loc. cit. correspond, in our notation, to $\tau^\dagger_{L/K}$ and $\kappa_L$. In addition, the third equality follows directly from \cite[Cor. 7.7]{BleyBurns}.

To derive the required equality (\ref{bb eq}) from the last displayed formula is it is then enough to note that (\ref{dagger tau}) implies  $\tau^\dagger_{L/K}\cdot y_{L/K}^{-1}$ is
equal to $\tau_K^G\cdot  \tau_{L/K}\cdot y_{L/K}^{-1} = \tau_K^G\cdot  \tau_{L/K}'$.

\begin{lemma}\label{adapted cong} For the homomorphisms $\theta_{i,1}$ for $0 \le i < k$ that are defined in (\ref{homs def}) one has \[
\prod_{i=0}^{k-1}\theta_{i,1} \equiv \psi_k^\vee(\theta_{0,1})  \,\,\,({\rm mod}\,\, {\rm Det}(U_f(\ZZ[G]))).
\]\end{lemma}

\begin{proof} This is proved by a slight adaptation of the arguments in \cite{BurnsChinburg} (and is implicitly used in the proof of Corollary 2.2 in loc. cit.). To be precise, we shall use the notation of \cite[\S4.3.1]{BurnsChinburg} with our integer $k$ corresponding the integer $\ell$ used in loc. cit.

Then the present hypotheses (on $k$) allow us to choose the integer $\ell'$ to be $(1-e)/\ell$. In particular, if we set $N := 0$, then $N_\ell = 0$ and, for each $i$ with $0 \le i < \ell$, also $N_i = -i\ell'(e-1) = -i\ell'e + N_i'$ with $N_i' := -i(e-1)/\ell$. Each element $a_{N_i}$ can therefore be written as $c_i\cdot a_{N'_i}$ with $c_i$ an element of $B$ with $v_{\mathfrak{p}}(c_i) = -i\ell'$.

With this choice of $\ell'$ an explicit computation shows that the integer $M_{\mathfrak{p},\ell,\ell'}$ defined in (2.4) of loc. cit. is equal to $\sum_{i=0}^{\ell-1}i\ell'$ and so one can take the element $c$ chosen in [loc. cit., Cor. 4.5] to be the product $\prod_{i=0}^{\ell-1}c_i$.
For this element there is for every $\chi$ in $\Hom (\Lambda, B^{c\times})$ an equality
\[ (ca_0\mid \psi_\ell\chi)\prod_{i=0}^{\ell-1}(a_{N_i}\mid \chi)^{-1} = (a_0\mid \psi_\ell\chi) \prod_{i=0}^{\ell-1}(a_{N_i'}\mid \chi)^{-1}.\]
and so [loc. cit., Cor. 4.5] asserts that the $\mathfrak{p}$-adic valuation of this element is zero.

It is now straightforward to derive the claimed congruence by combining this fact with the argument of [loc. cit., \S5].
\end{proof}

\section{Weakly ramified Galois-Gauss sums and the relative element $\fra_{L/K}$}

In the remainder of the article we study links between Galois-Gauss sums and hermitian and metric structures that arise in weakly ramified Galois extensions of odd degree. In this first section we define a canonical element in relative algebraic $K$-theory that is key to the theory we develop and then state some of the main results about this element that we establish in later sections.

At the outset we fix a finite odd degree Galois extension  of number fields $L/K$ that is `weakly ramified' in the sense of Erez \cite{Erez} (that is, the second lower ramification subgroups in $G$ of each place of $L$ are trivial) and set $G := G(L/K)$.

Since $L/K$ is of odd degree there exists a unique fractional $\mathcal{O}_L$-ideal $\mathcal{A}_{L/K}$ whose square is the
inverse of the different $\mathfrak{D}_{L/K}$ (see the discussion at the beginning of \S\ref{tame ggs section}).

In addition, since $L/K$ is weakly ramified, Erez has shown that $\mathcal{A}_{L/K}$ is a locally-free module with respect to
the restriction of the natural action of $\mathcal{O}_K[G]$ on $L$ (see \cite{Erez}).

We may therefore use the general construction of \S\ref{example0} to define a canonical element of $K_0(\ZZ[G],\QQ^c[G])$ by setting
\begin{equation}\label{a-def}
\mathfrak{a}_{L/K} := [\mathcal{A}_{L/K},\kappa_L, H_L]-\delta_G(\tau^G_K\cdot (\psi_{2,\ast}-1)(\tau'_{L/K}))
\end{equation}
where the Galois-Gauss sums $\tau^G_K$ and $\tau'_{L/K}$ are as defined in \S\ref{general cons}.

Proposition \ref{almost there} implies the projection of $[\mathcal{A}_{L/K},\kappa_L, H_L]$ to each of the groups $A(G), \HCl(G)$ and $\Cl(G)$ recovers arithmetical invariants related to $\mathcal{A}_{L/K}$ that have been studied in previous articles. By using this fact explicit information about the element $\fra_{L/K}$ can often constitute a strong refinement of pre-existing results or conjectures concerning the metric and hermitian structures that are associated to $\mathcal{A}_{L/K}$ and this observation motivates the systematic study of $\fra_{L/K}$ that we undertake in later sections.

In the next result (which will be proved in \S\ref{glob consequences}) we collect some of the main results that we prove concerning $\mathfrak{a}_{L/K}$.

In the sequel we write $\mathcal{W}_{L/K}$ for the set of finite places $v$ of $K$ that ramify wildly in an extension $L/K$ and
$\mathcal{W}_{L/K}^\Qu$ for the set of rational primes that lie below any place in $\mathcal{W}_{L/K}$.

We also let $A_{\rm tor}$ denote the  torsion subgroup of an abelian group $A$.

\begin{theorem}\label{main result}\mpar{main result}
Let $L/K$ be a finite odd degree weakly ramified Galois extension of number fields of group $G$. Then the following assertions are valid.

\begin{itemize}
\item[(i)] The element $\mathfrak{a}_{L/K}$ belongs to the subgroup
\[ \bigoplus_{\ell\in \mathcal{W}_{L/K}^\Qu}K_0(\mathbb{Z}_\ell[G],\mathbb{Q}_\ell[G])_{\rm tor}\]
of $K_0(\mathbb{Z}[G],\mathbb{Q}[G])$.
In particular, if $L/K$ is tamely ramified, then $\mathfrak{a}_{L/K} = 0$.

\item[(ii)] In $A(G)$ one has
\[ [\mathcal A_{L/K},h_{L,\bullet}] =  \Pi_G^{\rm met}(\Fa_{L/K})+ \varepsilon^{\rm met}_{L/K} \]
where the metric $h_{L,\bullet}$ is as defined in Example \ref{example1} and $\varepsilon^{\rm met}_{L/K}$
is represented by the pair $(1,\theta)$ with $\theta(\phi)=
( {|G|^{[K:\IQ]}|{\rm d}_K|})^{\phi(1)/2}\cdot | {\tau}(K,\psi_2(\phi)-\phi)|$ for all $\phi$ in $R_G.$

\item[(iii)] In ${\rm HCl}(G)$ one has
\[ {\rm Disc}(\mathcal A_{L/K},t_{L/K}) = \Pi_G^{\rm herm}(\Fa_{L/K}) + \varepsilon^{\rm herm}_{L/K}\]
where the hermitian form $t_{L/K}$ is as defined in Example \ref{example2} and $\varepsilon^{\rm herm}_{L/K}$ is represented by the pair $(1,\tilde\theta)$ with $\tilde\theta(\phi)=  {{\rm d}_K^{\phi(1)/2}\cdot \tau(K,\psi_2(\phi) - \phi)}$ for all $\phi$ in $R^s_G.$
\item[(iv)] In ${\rm Cl}(G)$ one has $\partial_{\Ze, \Qc, G}^0(\mathfrak{a}_{L/K}) = [\mathcal{A}_{L/K}]$.
\end{itemize}
\end{theorem}

\begin{remark}
In addition to the result of Theorem \ref{main result}(i) it is also possible to explicitly compute $\fra_{L/K}$ for certain (weakly)
wildly ramified extensions $L/K$ (see, for example, Theorem \ref{vinatier thm} below). These results show, in particular,
that $\fra_{L/K}$ does not in general vanish.
\end{remark}

In Conjecture \ref{local conj} below we shall offer a precise  conjectural description of $\fra_{L/K}$ in terms of local (second) Galois-Jacobi sums and invariants related to fundamental classes arising in local class field theory. This description is related to certain `epsilon constant conjectures' that are already in the literature and hence to the general philosophy of Tamagawa number conjectures that originated with Bloch and Kato.

This connection gives a new perspective to the theory of the square root of the inverse different but does not itself help to compute $\mathfrak{a}_{L/K}$ explicitly in any degree of generality.

Nevertheless, our methods combine with extensive numerical experiments to suggest that, rather surprisingly, it might also be possible in general to describe $\mathfrak{a}_{L/K}$ very explicitly (see \S\ref{second conj sect}). This possibility is definitely worthy of further investigation, not least because it could be used to obtain significant new evidence in the context of certain wildly ramified Galois extensions in support of the formalism of Tamagawa number conjectures.

In a different direction, Theorem \ref{main result} leads to effective `finiteness results' on the natural arithmetic invariants related to $\fra_{L/K}$ that arise as the extension $L/K$ varies.

To give a simple example of such a result, for each number field $K$ and finite abstract group $\Gamma$ of odd order we write
${\rm WR}_K(\Gamma)$ for the set of fields $L$ that are weakly ramified odd degree Galois extensions of $K$
and for which there exists an isomorphism of groups $\iota \colon G(L/K)\cong \Gamma.$

For each field $L \in {\rm WR}_K(\Gamma)$ we then write ${\rm Is}_L(\Gamma)$ for the set of group isomorphisms
\mbox{$\iota : G(L/K) \cong \Gamma$}, and for each $\iota \in \mathrm{Is}_L(\Gamma)$ we consider
the induced isomorphism of relative algebraic $K$-groups
\[
\iota_{\ast}: K_0(\IZ[G(L/K)],\IQ^c[G(L/K)])\cong K_0(\IZ[\Gamma],\IQ^c[\Gamma]).\]
%

We then define a subset of `realisable classes' in $K_0(\IZ[\Gamma],\IQ^c[\Gamma])$ by setting
\[
R^{\rm wr}_K(\Gamma):=\{\iota_{\ast}(\Fa_{L/K}) \colon  L\in {\rm WR}_K(\Gamma),\iota \in {\rm Is}_L(\Gamma)\}.
\]

Recalling that the group $K_0(\ZG, \QG)_{\mathrm{tor}}$ is finite (see, for example,  \cite[Cor.~2.5]{BleyWilson}) the result of Theorem \ref{main result}(i) leads directly to the following result.

\begin{coro}\label{finiteness coro} The set $R^{\rm wr}_K(\Gamma)$ is finite. \end{coro}

%
%


In \S\ref{vinatier section} we explain how the set $R^{\rm wr}_K(\Gamma)$
can be computed effectively and then apply the general theory in the setting of an explicit conjecture of Vinatier (from \cite[\S1, Conj.]{Vinatier2}) concerning the Galois structure of $\mathcal{A}_{L/K}$.

%


To end this section we prove an important preliminary result.

\begin{prop}\label{rationality thm} \mpar{rationality thm}
Let $L/K$ be a finite odd degree weakly ramified Galois extension of number fields of group $G$.  Then $\mathfrak a_{L/K}$ belongs to the subgroup $ K_0(\IZ[G],\IQ[G])$ of $K_0(\IZ[G],\IQ^c[G])$.
\end{prop}

\begin{proof} For $x$ and $y$ in $K_0(\ZG, \Qc[G])$ we write $x \equiv y$ if $x-y$ belongs to $K_0(\ZG, \QG)$.

Then $\fra_{L/K}$ is equal to
\begin{align*}
 [\calA_{L/K}, \kappa_L, H_L]   - \delta_G(\tau_K^G\cdot(\psi_{2,\ast}-1)(\tau'_{L/K})) \equiv \,& \left(  [\calA_{L/K}, \kappa_L, H_L]  - \delta_G(\tau_{L/K}^\dagger) \right) - \delta_G( J_{2,L/K})\\
           \equiv \, &   [\calA_{L/K}, \kappa_L, H_L]  - \delta_G(\tau_{L/K}^\dagger)
\end{align*}
where the first equivalence follows from (\ref{label product}) and the obvious containment $(\psi_{2,\ast}-1)(y_{L/K}) \in \zeta(\QQ[G])$ and the second from Lemma \ref{jacobi lemma} (with $k = 2$).

It thus suffices to note that the computations in \cite[p.~555-556]{BleyBurns} (which rely heavily on a result of Fr\"ohlich in \cite[\S 9 (i),(ii)]{Frohlich}) show that $ [\calA_{L/K}, \kappa_L, H_L]  \equiv \delta_G(\tau_{L/K}^{\dagger})$.\end{proof}

\section{Functoriality properties of $\fra_{L/K}$}\label{functorial section}

Following Proposition \ref{rationality thm} we know that each element $\mathfrak a_{L/K}$ belongs to $ K_0(\IZ[G],\IQ[G])$. In this section we prove  the following result which establishes the basic functorial properties of these elements as the extension $L/K$ varies.

\begin{theorem}\label{funct thm}\mpar{func thm}
Let $L/K$ be a weakly ramified odd degree Galois extension of number fields of group $G$,
fix an intermediate field $F$ of $L/K$ and set $J := G(L/F)$.
\begin{itemize}\item[(i)]
The restriction map $\rho_J^G:K_0(\IZ[G],\IQ[G])\to K_0(\IZ[J],\IQ[J])$
sends $\Fa_{L/K}$ to $\Fa_{L/F}$ .
\item[(ii)]
Assume $J$ is normal in $G$ and write $\Gamma$ for the quotient $G/J\cong G({F/K})$.
Then the natural coinflation map $\pi_\Gamma^G:K_0(\IZ[G],\IQ[G])\to K_0(\IZ[\Gamma],\IQ[\Gamma])$ sends $\Fa_{L/K}$ to $\Fa_{F/K}$.
\end{itemize}
\end{theorem}

\begin{proof} It is convenient to first prove claim (ii) in the statement of Theorem \ref{funct thm}.
To do this we use the commutative diagram

\begin{equation}\label{comm diag}
\xymatrix{
\zeta (\IQ^c[G])^\times \ar[d]_{\tilde{\pi}_\Gamma^G} \ar[r]^-{\delta_G} &K_0(\IZ[G],\IQ^c[G])\ar[d]_{\pi_\Gamma^G}\\
\zeta(\IQ^c[\Gamma])^\times \ar[r]^-{\delta_\Gamma} &K_0(\IZ[\Gamma],\IQ^c[\Gamma]).}
\end{equation}
in which $\tilde{\pi}^G_{\Gamma}(z)_\phi = z_{\inf_{\Gamma}^G(\phi)}$ for all $z$ in $\zeta (\IQ^c[G])^\times$ and $\phi$ in $\widehat \Gamma$ (see, for example, \cite[p.~577]{BleyBurns}).

Then both of the equalities $\tilde{\pi}^G_{\Gamma}(\tau^G_K)=\tau^\Gamma_K$ and $\tilde{\pi}^G_{\Gamma}((\psi_{2,\ast}-1)(\tau'_{L/K})) = (\psi_{2,\ast}-1)(\tau'_{F/K})$ follow easily from the (well-known) facts that Gauss sums and unramified characteristics are invariant under inflation and Adams operations commute with inflation.

Hence the key point in proving claim (ii) is to prove $\pi^G_\Gamma([\mathcal A_{L/K},\kappa_L, H_L]) = [\mathcal A_{F/K},\kappa_F,H_F]$.
To show this we write ${\rm tr}_{L/F}$ for the field theoretic trace map $L\to F$.
Since $\mathcal A_{L/K}$ is $\ZG$-projective it is also cohomologically trivial and so
$\mathcal A^J_{L/K}={\rm tr}_{L/F}(\mathcal A_{L/K})=\mathcal A_{F/K}$,
where the last equality follows, for example, from the explicit computations of Erez in \cite[p. 246]{Erez}.

In addition, the natural identification of $H^J_L$ with $H_F$ induces a commutative diagram of $\IQ^c[\Gamma]$-modules

\[
\begin{CD}
(\IQ^c\otimes_\IQ L)^J @>\kappa^J_L >> (\IQ^c\otimes_\IZ H_L)^J \\
@\vert @\vert \\
\IQ^c\otimes_\IQ F @> \kappa_F >> \IQ^c\otimes_\IZ H_F
\end{CD}
\]
and, taken together, these facts imply that
\[
\pi^G_\Gamma([\mathcal A_{L/K}, \kappa_L, H_L]) = [\mathcal A^J_{L/K},\kappa_L^J, H^J_L] =[\mathcal A_{F/K},\kappa_F, H_F],
\]
as required to complete the proof of claim (ii) of Theorem \ref{funct thm}.

To prove Theorem \ref{funct thm}(i)  we use the commutative diagram (see, for example, \cite[p.~575]{BleyBurns})
\begin{equation}\label{comm diag2}
\begin{CD}
\zeta(\IQ^c[G])^\times @> \delta_G >> K_0(\IZ[G],\IQ^c[G])\\
@V\tilde{\rho}^G_J VV @V \rho^G_J VV \\
\zeta(\IQ^c[J])^\times @> \delta_J>> K_0(\IZ[J],\IQ^c[J]).
\end{CD}
\end{equation}
Here, for each $z$ in $\zeta(\IQ^c[G])^\times$ and $\phi$ in $\widehat{J}$, one has
$\tilde{\rho}^G_J(z)_\phi = \prod_{\chi\in \widehat{G}}z_\chi^{\langle\chi,{\rm ind}^G_J(\phi)\rangle_G}$
where we write $\langle \cdot, \cdot \rangle_G$ for the natural pairing on $R_G$.

For each number field $E$ we now set $\tau_E := \tau(\Qu, \ind_E^\Qu {\bf 1}_E)$.
We claim that
\begin{equation}
  \label{eq:112}
\tilde{\rho}^G_J(\tau^G_K) = {\rm Nrd}_{\IQ[J]}(\tau^{[G:J]}_K).
\end{equation}
In fact, for all $\phi \in \widehat{J}$ one has
\[
{\rm Nrd}_{\IQ[J]}(\tau^{[G:J]}_K)_\phi = \tau_K^{\phi(1)[G:J]} \,\text{ and } \, \tilde{\rho}^G_J(\tau^G_K)_\phi =
\prod_{\chi \in \widehat G}\tau_K^{\chi(1) \langle \chi, \ind_J^G\phi \rangle_G}\]
and so the claimed equality is valid since $\sum_{\chi \in \widehat G}{\chi(1) \langle \chi, \ind_J^G\phi \rangle_G} = {\phi(1)[G:J]}$.

We next note that, since $|G|$ is odd, one has
\[ \ind_J^G(\psi_2(\phi)) = \psi_2 ( \ind_J^G(\phi) )\]
for all $\phi$ in $\widehat G$ (see, for example, \cite[Prop.-Def. 3.5]{Erez}). Thus, given the commutativity of (\ref{comm diag2}) and the (well-known) inductivity in degree zero of both Galois-Gauss sums
and non-ramified characteristics one deduces that
\begin{equation}
  \label{eq:113}
\rho^G_J(\delta_G((\psi_{2,\ast}-1)(\tau'_{L/K}))) =\delta_J(\tilde{\rho}^G_J((\psi_{2,\ast}-1)(\tau'_{L/K}))) =
\delta_J((\psi_{2,\ast}-1)(\tau'_{L/F})).
\end{equation}

By combining (\ref{eq:112}) and (\ref{eq:113}) we obtain an equality
\begin{align*}\rho^G_J(\delta_G(\tau^G_K\cdot (\psi_{2,\ast}-1)(\tau'_{L/K}))) =\, &\delta_J({\rm Nrd}_{\IQ[J]}(\tau^{[G:J]}_K))+\delta_J((\psi_{2,\ast}-1)(\tau'_{L/F}))\\
=\, &\delta_J({\rm Nrd}_{\IQ[J]}(\tau_F^{-1}\cdot\tau^{[G:J]}_K)+\delta_J(\tau^J_F\cdot (\psi_{2,\ast}-1)(\tau'_{L/F})).
\end{align*}

To consider the corresponding behaviour of the term $[\mathcal A_{L/K}, \kappa_L, H_L]$
under restriction the key point is that in the subgroup $K_0(\IZ[J],\IQ[J])$ of $K_0(\IZ[J],\IQ^c[J])$ there are equalities
\begin{multline*}\label{eqn3}
\rho^G_J([\mathcal A_{L/K},\kappa_L, H_L])\!-\! [\mathcal A_{L/F},\kappa_L, H_L]
\!=\! [\mathcal A_{L/K},\kappa_L, H_L]\! -\![\mathcal A_{L/F},\kappa_L, H_L]\! =\! [\mathcal A_{L/K}, {\rm id},   \mathcal A_{L/F}]\notag\\
\!=\![\mathcal A_{L/F}\mathcal A_{F/K}, {\rm id}, \mathcal A_{L/F}]\! =\! \delta_J({\rm Nrd}_{\IQ[J]}(\tau^{-1}_F\cdot \tau^{[G:J]}_K)).
\end{multline*}
Here the first equality is obvious, the second is by the defining relations of $K_0(\IZ[J],\IQ^c[J])$,
the third follows from the (well-known) multiplicativity property $\mathcal A_{L/K}=\mathcal A_{L/F}\mathcal A_{F/K}$
and the fourth from the result of Lemma \ref{lemma Nrd below} below.

Comparing the last two displayed equalities it follows directly that $\rho^G_J(\Fa_{L/K})=\Fa_{L/F}$, as claimed.
\end{proof}

\begin{lemma}\label{lemma Nrd below} With the subgroup $J$ and field $F$ as above one has 
$[\mathcal A_{L/F}\mathcal A_{F/K}, {\rm id},   \mathcal{A}_{L/F}]=\delta_J({\rm Nrd}_{\IQ[J]}(\tau_F^{-1}\cdot \tau_K^{[G:J]}))$ in $K_0(\IZ[J],\IQ[J])$.
\end{lemma}

\begin{proof} By Lemma \ref{prepare proof} below it suffices to show that $N_{F/\Qu}(\calA_{F/K}) = \pm \tau_F^{-1}\cdot\tau_K^{[G:J]}$.

This equality is, in turn, a direct consequence of the fact that
\begin{multline*}
  N_{F/\Qu}(\calA_{F/K})^2 = N_{F/\Qu}(\frD_{F/K})^{-1} = N_{F/\Qu}(\frD_{F/\Qu}^{-1} \frD_{K/\Qu}) = d_{F/\Qu}^{-1} \cdot N_{K/\Qu}(\frD_{K/\Qu})^{[F:K]} \\
  = d_{F/\Qu}^{-1} \cdot   d_{K/\Qu}^{[F:K]} = \tau_F^{-2} \cdot \tau_K^{2[G:J]},
 \end{multline*}
where the last equality follows from (\ref{disc pfaff}).
\end{proof}

\begin{lemma}\label{prepare proof}\mpar{prepare proof}
  Let $E$ be a number field and $G$ a finite group. Let $N$ be a locally free $\OE[G]$-module of rank one. Let $\fra$ denote a
fractional $\OE$-ideal. Then in $K_0(\ZG, \QG)$ one has
\[
[\fra N, \mathrm{id}, N] = \delta_G(\Nrd_\QG(N_{E/\Qu}(\fra))).
\]
\end{lemma}

\begin{proof} Recall that for each prime $p$ and each $\Ze$-module $X$ we write $X_p$ for the $\Zp$-module $\Zp \tensor_\Ze X$.

In particular, there is an isomorphism $N_p \simeq \left( \calO_{E,p}[G] \right)^d$ of $\calO_{E,p}[G]$-modules and hence in $K_0(\Zp[G], \Qp[G])$ an equality
\[
[(\fra N)_p, \id, N_p] = [\fra_p[G]^d, \id, \calO_{E,p}[G]^d] = d  [\fra_p[G], \id, \calO_{E,p}[G]].
\]
It follows that $[\fra N, \id, N] = d [\fra[G], \id, \OE[G]]$ in $K_0(\ZG, \QG)$.

Now set $n := [E:\Qu]$ and choose $\Ze$-bases $\omega_1, \ldots, \omega_n$ for $\OE$ and $\alpha_1, \ldots, \alpha_n$ for $\fra$.
Then

\[
\fra[G] = \bigoplus_{i=1}^n \ZG\alpha_i, \quad \OE[G] = \bigoplus_{i=1}^n \ZG\omega_i.
\]
With respect to these bases the identity is represented by the matrix $B \in \Gl_n(\Qu) \sseq \Gl_n(\QG)$ defined by
$B = (b_{ji})$ where $\alpha_i = \sum_{j=1}^nb_{ji}\omega_j$. Note that $|\det(B)| = N_{E/\Qu}(\fra)$.

By the defining relations in relative $K$-groups and the definitions of $\partial_{\Ze, \Qu, G}$ and $\delta_G$
we obtain
\[
[\fra[G], \id, \OE[G]] = [\ZG^n, B, \ZG^n] =  \partial_{\Ze, \Qu, G}([\QG^n, B]) = \delta_G(\Nrd_\QG(B)).
\]
Now $\Nrd_\QG(B) = \sum_{\chi \in \widehat G}x_\chi e_\chi$ with $x_\chi = \det(T_\chi(B)) = \det(B)^{\chi(1)}$, where
$T_\chi$ is a representation with character $\chi$. Hence $\Nrd_\QG(B) = \Nrd_\QG(\det(B))$ and
\[
[\fra[G], \id, \OE[G]] = \delta_G(\Nrd_\QG(\det(B))) =  \delta_G(\Nrd_\QG(|\det(B)|)) = \delta_G(\Nrd_\QG(N_{E/\Qu}(\fra))).
\]
where the second equality follows from $\delta_G(\Nrd_\QG(-1)) = 0$.
\end{proof}

\section{A canonical local decomposition of $\fra_{L/K}$}

In this section we follow the approach of Breuning in \cite{mb2} to give a canonical decomposition of the
term $\mathfrak{a}_{L/K}$ as a sum of terms which depend only upon the local extensions $L_w/K_v$ for
places $v$ of $K$ which ramify wildly (and weakly) in $L/K$.

\subsection{The local relative element $\fra_{F/E}$}

We first define the canonical local terms that will occur in the decomposition of $\mathfrak{a}_{L/K}$.

To do this we fix a rational prime $\ell$ and an odd degree weakly ramified Galois extension $F/E$ of fields
which are contained in $\mathbb{Q}_\ell^c$ and of finite degree over $\mathbb{Q}_\ell$ and we set $\Gamma := G({F/E})$.

We also fix an embedding of fields $j_\ell: \Qc \lra \Qlc$ and by abuse of notation also write $j_\ell :\zeta(\Qc[\Gamma]) \lra \zeta(\Qlc[\Gamma])$ for the induced ring embedding. We then write
$$j_{\ell,*} :  K_0(\Ze[\Gamma], \Qc[\Gamma]) \lra K_0(\Zl[\Gamma], \Qlc[\Gamma])$$ for the homomorphism of abelian groups that sends each element $[P,\iota,Q]$ to $[P_\ell,\QQ^c_\ell\otimes_{\QQ^c,j_\ell}\iota,Q_\ell]$ and we note that $j_{\ell, *} \circ \delta_\Gamma = \delta_{\Gamma,\ell}\circ j_\ell$.

We write $\Sigma(F)$ for the set of embeddings $F\hookrightarrow\mathbb{Q}_\ell^c$ and
\[
\kappa_F: \mathbb{Q}_\ell^{c}\otimes_{\mathbb{Q}_\ell} F \rightarrow\prod_{\Sigma(F)}\mathbb{Q}_\ell^c
\]
for the isomorphism of $\mathbb{Q}_\ell^c[\Gamma]$-modules sending $x\otimes f$ to $(\sigma(f)x)_{\sigma \in \Sigma(F)}$ for $f\in F$ and $x \in \mathbb{Q}_\ell^c.$

We also write $H_F$ for the submodule $\prod_{\Sigma(F)}\mathbb{Z}_\ell$ of $\prod_{\Sigma(F)}\mathbb{Q}_\ell^c$ and note that $[\mathcal{A}_{F/E},\kappa_F,H_F]$ is then a well-defined element of $K_0(\mathbb{Z}_\ell[\Gamma],\mathbb{Q}_\ell^c[\Gamma])$.

We next write $U_{F/E}$ for the canonical `unramified' element of $K_0(\Zl[\Gamma],\mathbb{Q}_\ell^c[\Gamma])$
defined (for any Galois extension of local fields) by Breuning in \cite{mb2} and then define an element of
$K_0(\Zl[\Gamma],\mathbb{Q}_\ell^c[\Gamma])$ by setting
\[
\mathfrak{a}_{F/E} := [\calA_{F/E}, \kappa_F, H_F] - \delta_{\Gamma,\ell}(j_\ell( \tau^\Gamma_E\cdot (\psi_{2,\ast}-1)(\tau'_{F/E}))) - U_{F/E},
\]
where the elements $\tau^\Gamma_E$ and $\tau'_{F/E}$ of $\zeta(\QQ^c[\Gamma])^\times$ are constructed
from local Galois-Gauss sums as in Remark \ref{local analogue}.

The point of introducing the element $U_{F/E}$ is that it guarantees that $\mathfrak{a}_{F/E}$ is `rational' in the sense of the following lemma.


\begin{prop}\label{dagger proof}
$\mathfrak{a}_{F/E}$ is independent of the choice of $j_\ell$ and belongs to $K_0(\mathbb{Z}_\ell[\Gamma],\mathbb{Q}_\ell[\Gamma])$.
\end{prop}

\begin{proof} The first assertion follows immediately from \cite[Lem. ~2.2]{mb2} and the containment
\[ \tau^\Gamma_E\cdot (\psi_{2,\ast}-1)(\tau'_{F/E})\cdot (\tau_{F/E}^\dagger)^{-1}\in \zeta(\QQ[\Gamma])^\times,\]
which itself follows directly from (\ref{local label product}) and the local analugue of Lemma \ref{jacobi lemma} (see Remark \ref{local analogue}).

The second claim follows by combining the same containment with the containment
\[
[\calA_{F/E}, \kappa_F, H_F] - \delta_{\Gamma,\ell}(j_\ell(\tau^\dagger_{F/E}))- U_{F/E} \in K_0(\IZ_\ell[\Gamma],\IQ_\ell[\Gamma]).
\]
proved by Breuning's argument in \cite[Prop.~3.4]{mb2}. \end{proof}

\subsection{$\fra_{F/E}$ and fundamental classes}
\label{local conj section}

In this section we reformulate the local epsilon constant conjecture formulated by Breuning in \cite[Conj.~3.2]{mb2} in terms of the explicit element $\fra_{F/E}$.

To this end we recall that for any finite Galois extension of $\ell$-adic fields $F/E$, of group $\Gamma$, Breuning's conjecture is an equality in $K_0(\Zl[\Gamma], \Ql[\Gamma])$ of the form
\begin{equation}\label{bc first}
T_{F/E} + C_{F/E} + U_{F/E} - M_{F/E} = 0.\end{equation}
Here, in addition, to the element $U_{F/E}$ used in the previous section, the following elements also occur:
where
\begin{itemize}
  \item $T_{F/E} := \delta_{\Gamma,\ell}(j_\ell(\tau_{F/E}^\dagger))$ is the equivariant local epsilon constant.
\item $C_{F/E} = \mathcal{E}(\exp_\ell(\calL))_\ell - [\calL, \kappa_F, H_F]$, with $\calL$ is any full projective
$\Zl[\Gamma]$-sublattice of $\calO_F$ that is contained in a sufficiently large power of the maximal ideal $\frp_F$ of $\calO_F$ to ensure the
$\ell$-adic exponential map ${\rm exp}_\ell$ converges on $\calL$. For the precise definition of $\mathcal{E}(\exp_\ell(\calL))_\ell$ we refer the reader to \cite[\S~2.4]{mb2} and \cite[\S~3.2]{BleyBurns}. For the moment, we point out only that this element relies on local fundamental classes and is
 very difficult to compute explicitly in any degree of generality.
\item $M_{F/E}$ is a simple and explicitly defined correction term (see \cite[\S~2.6]{mb2}).
\end{itemize}

To reinterpret (\ref{bc first}) we assume $F/E$ is weakly ramified. In this case the lattice $\mathcal{L}$ that occurs above can be taken to be $p^N\cdot \calA_{F/E}$ for any sufficiently large integer $N$ and the element
\[
{\mathcal{E}}_{F/E} := \mathcal{E}(\exp_\ell(p^N\cdot\calA_{F/E}))_\ell - \delta_{\Gamma, \ell}(\Nrd_{\Ql[\Gamma]}(p^{N[E:\Ql]}))
\]
of $K_0(\ZZ_\ell[\Gamma],\QQ_\ell[\Gamma])$ is easily seen to be independent of the choice of $N$.

We next define an element of $K_0(\ZZ_\ell[\Gamma],\QQ_\ell[\Gamma])$ by setting
%
\begin{equation}\label{twisted unram char def}
\frc_{F/E} :=  \delta_{\Gamma,\ell} ((1 - \psi_{2,\ast})(y_{F/E})).
\end{equation}

Then by combining Lemma \ref{prepare proof} with (\ref{local label product}) one finds that
 Breuning's conjectural equality (\ref{bc first}) is equivalent to the following conjecture.

\begin{conj}\label{local conj}
Let $F/E$ be a weakly ramified Galois extension of $\ell$-adic fields with group $\Gamma$. Then in $K_0(\ZZ_\ell[\Gamma],\QQ_\ell[\Gamma])$ one has
\[
\fra_{F/E} = {\mathcal{E}}_{F/E}   - \delta_{\Gamma, \ell}(J_{2,F/E}) - \frc_{F/E} - M_{F/E},
\]
where the second Galois-Jacobi sum $J_{2,F/E}$ of $F/E$ is as discussed in Remark \ref{local analogue}.
\end{conj}

\begin{remark}\label{explicit rems twisted} For later purposes we note that (\ref{explicit unram char}) implies that
$(1-\psi_{2,\ast})(y_{F/E}) = (1-e_{\Gamma_0}) + \sigma^{-1} e_{\Gamma_0}$,
with $\Gamma_0$ the inertia subgroup of $\Gamma$ and $\sigma$ an element of
$\Gamma$ that projects to the Frobenius in $\Gamma/\Gamma_0$,
and hence that $\frc_{F/E} =  \delta_{\Gamma,l} ((1-e_{\Gamma_0}) + \sigma^{-1} e_{\Gamma_0})$.

In particular, $\frc_{F/E}$ vanishes if $F/E$ is tame (since then
 $(1-e_{\Gamma_0}) + \sigma^{-1} e_{\Gamma_0} \in \Zl[\Gamma]^\times$
 and, in all cases, $\Gamma/\Gamma_0$ is abelian and so
 $\Nrd_{\Ql[\Gamma]}( (1-e_{\Gamma_0}) + \sigma^{-1} e_{\Gamma_0}) = (1-e_{\Gamma_0}) + \sigma^{-1} e_{\Gamma_0}$).
\end{remark}

\subsection{The decomposition result}
We can now state and prove the main result of this section. In this result we use for each prime $\ell$, each extension $E$ of $\IQ_\ell$ and each subgroup $H$ of $G$,
the natural induction map ${\rm i}^G_{H,E}: K_0(\IZ_\ell[H],E[H]) \to K_0(\IZ_\ell[G],E[G])$ on relative $K$-groups.

\begin{theorem}\label{decomp thm}\mpar{decomp thm}
Let $L/K$ be a weakly ramified odd degree Galois extension of number fields of group $G$.
Then in $K_0(\IZ[G],\IQ[G])$ one has an equality
\[
\Fa_{L/K}= \sum_{\ell}\sum_{v\mid \ell}{\rm i}^G_{G_w,\IQ_{\ell}}(\Fa_{L_w/K_v})
\]
where the sum is over all primes $\ell$ and for each place $v$ of $K$ we fix a place $w$ of
$L$ lying above $v$ and identify the Galois group of $L_w/K_v$
with the decomposition subgroup $G_w$ of $w$ in $G$.
\end{theorem}

\begin{proof} Proposition \ref{rationality thm} implies $\Fa_{L/K}$ decomposes naturally as a sum $\sum_{\ell }\Fa_{L/K,\ell}$ of $\ell$-primary components and so it suffices to prove that for each $\ell$ there is in $K_0(\IZ_\ell[G],\IQ_\ell[G])$ an equality
\begin{equation}
\label{l adic eq} \Fa_{L/K,\ell} = \sum_{v\mid \ell}{\rm i}^G_{G_w,\IQ_{\ell}}(\Fa_{L_w/K_v}).
\end{equation}

To do this we fix a prime $\ell$ and an embedding $j_\ell: \IQ^c \to \IQ_\ell^c$ and write $\mathcal{O}_\ell^t$ for the
valuation ring of the maximal tamely ramified extension of $\IQ_\ell$ in $\IQ_\ell^c$.

We recall that Taylor's Fixed Point Theorem for group determinants (see \cite[Chap. 8, Th. 1.1]{taylor}) implies that the
following composite homomorphism is injective
\mpar{jltstar}
\begin{equation}\label{jltstar}
K_0(\IZ_\ell[G],\IQ_\ell[G]) \to K_0(\IZ_\ell[G],\IQ_\ell^c[G]) \xrightarrow{j_{\ell,*}^t} K_0(\mathcal{O}_\ell^t[G],\IQ_\ell^c[G])
\end{equation}
where the first arrow is the natural inclusion and $j_{\ell,*}^t$ sends $[X,\xi,Y]$ to
$[\mathcal{O}_\ell^t\otimes_\Zl X,\xi,\mathcal{O}_\ell^t\otimes_\Zl Y]$.
It is therefore enough to show that the equality (\ref{l adic eq}) holds after applying $j_{\ell,*}^t$.

The key ingredients required to prove this fact are due to Breuning and are stated in Lemma \ref{decomp alg term} below.

In the sequel we abbreviate ${\rm i}^G_{G_w,\IQ_{\ell}}$ and ${\rm i}^G_{G_w,\IQ^c_{\ell}}$ to ${\rm i}_{w,\ell}$ and
${\rm i}^c_{w,\ell}$ respectively.

In particular, if for any finite Galois extension $F/E$ of either local fields or number fields we set
\[ \tau_{F/E,2} := \tau_E^{G(F/E)}\cdot (\psi_{2,\ast}-1)(\tau'_{F/E}),\]
then Breuning's results as stated below combine with the explicit definitions of the terms $\Fa_{L/K,\ell}$ and $\Fa_{L_w/K_v}$ to imply that

\begin{align}\label{partial} &j_{\ell,*}^t( \Fa_{L/K,\ell}-\sum_{v\mid \ell}{\rm i}_{w,\ell}(\Fa_{L_w/K_v}))\notag\\
= &j_{\ell,*}^t\left( j_{\ell,*}([\calA_{L/K,\ell}, \kappa_{L,\ell}, H_{L, \ell}]) - j_{\ell,*}(\delta_G(\tau_{L/K,2})) \right)\notag\\
&\hskip 0.2truein - \sum_{v\mid \ell}j_{\ell,*}^t\left({\rm i}^c_{w,\ell}( [\calA_{L_w/K_v}, \kappa_{L_w}, H_{L_w}] )-{\rm i}^c_{w,\ell}(\delta_{G_w,\ell}(j_\ell(\tau_{L_w/K_v,2})))- {\rm i}^c_{w,\ell}(U_{L_w/K_v})   \right)\notag\\
= &-j_{\ell,*}^t(j_{\ell,*}(\delta_G(\tau_{L/K,2})))
+ \sum_{v\mid \ell}j_{\ell,*}^t({\rm i}^c_{w,\ell}(\delta_{G_w,\ell}(j_\ell(\tau_{L_w/K_v,2}))))\notag\\
= &-j^t_{\ell,*}(j_{\ell,*}(\delta_G(\tau_{L/K,2}(\tau^\dagger_{L/K})^{-1})))
 + \sum_{v\mid \ell}j_{\ell,*}^t({\rm i}^c_{w,\ell}(\delta_{G_w,\ell}(j_\ell(\tau_{L_w/K_v,2}(\tau^\dagger_{L_w/K_v})^{-1}))))\\
 & \hskip 1truein  - j^t_{\ell,*}(\delta_{G,\ell}(\prod_{v \mid d_L\atop{v\nmid \ell}}\tilde i_w(y_{L_w/K_v}))).\notag
 \end{align}
Here the first equality follows directly from the definitions and the second uses Lemma \ref{decomp alg term}(i) and (ii). In addition, the third equality follows from Lemma \ref{decomp alg term}(iii) below and uses the map
\[ \tilde i_w: \zeta (\Qlc[G_w])^\times \to \zeta(\Qlc[G])^\times\]
that satisfies $\tilde i_w(x)_\chi = \prod_{\varphi \in \widehat G_w}x_\varphi^{\langle \mathrm{res}_{G_w}^G\chi, \varphi \rangle_{G_w}}$ for all $x$ in  $\zeta (\Qlc[G_w])^\times$ and $\chi$ in $\widehat G$.

Now, by (\ref{label product}), the first term in the expression (\ref{partial}) is equal to
\[
 -(j_{\ell,*}^t \circ \delta_{G,\ell}\circ j_{\ell})(J_{2,L/K} \cdot (\psi_{2,\ast}-1)(y_{L/K}^{-1})).
\]
%
In the same way, equality (\ref{local label product}) implies that the second term in (\ref{partial}) is
\[
(j_{\ell,*}^t \circ i_{w, l}^c \circ \delta_{G_w, \ell} \circ j_\ell)(\prod_{v \mid \ell}J_{2,L_w/K_v}\cdot (\psi_{2,\ast}-1)(y_{L_w/K_v}^{-1})).
\]


These two expressions combine with the commutative diagram

\begin{equation}\label{ind comm diagram}
\xymatrix{
\zeta (\Qlc[G_w])^\times \ar[d]_{\tilde{i}_w} \ar[r]^-{\delta_{G_w,\ell}} &K_0(\Zl[G_w],\Qlc[G_w])\ar[d]_{i_{w,\ell}^c}\\
\zeta(\Qlc[G])^\times \ar[r]^-{\delta_{G,\ell}} &K_0(\Zl[G],\Qlc[G]),
}
\end{equation}
the fact that $j_\ell(J_{2,L/K}) = \prod_{v\mid d_L}\tilde i_w(j_\ell(J_{2,L_w/K_v}))$ by the decomposition of global Galois-Gauss sums as a product of local Galois-Gauss sums and the explicit definition of $y_{L/K}$ to show that the sum in (\ref{partial}) is equal to the image under
$j^t_{\ell,*}\circ \delta_{G, \ell}$ of
%
\begin{eqnarray*}
&& j_\ell(J_{2,L/K})^{-1}\cdot\prod_{v\mid \ell}\tilde i_w(j_\ell(J_{2,L_w/K_v}))\cdot
\prod_{v \mid d_L\atop {v\nmid \ell}}\tilde i_w(j_\ell((\psi_{2,\ast}-2)(y_{L_w/K_v}))) \\
&=& \prod_{v \mid d_L\atop {v\nmid \ell}}\tilde i_w(j_\ell(J_{2,L_w/K_v}^{-1}\cdot (\psi_{2,\ast}-2)(y_{L_w/K_v}))).
\end{eqnarray*}
%
%
%
%

It is thus enough to note the image under $j^t_{\ell,*}\circ \delta_{G,\ell}$ of the latter element vanishes as a consequence of \cite[(9) and Lem. 5.3]{BreuPhd} and the second displayed equation on page 68 of loc. cit.

This completes the proof of Theorem \ref{decomp thm}.
\end{proof}

\begin{lemma}\label{decomp alg term}\
\begin{itemize}
\item[(i)] For each prime $\ell$ one has
\[
j_{\ell,*}([\calA_{L/K,\ell}, \kappa_{L,\ell}, H_{L,\ell}])=\sum_{v\mid \ell}{\rm i}^G_{G_w,\IQ^c_\ell}([\calA_{L_w/K_v}, \kappa_{L_w}, H_{L_w}]).
\]
\item[(ii)] For each $v\mid \ell$ the element ${\rm i}^G_{G_w,\IQ^c_{\ell}}(U_{L_w/K_v})$ belongs to $\ker(j^t_{\ell,*})$.

\item[(iii)] One has
\[
j_{\ell,*}(\delta_G(\tau^\dagger_{L/K})) - \sum_{v\mid \ell}{\rm i}^G_{G_w,\IQ_{\ell}}(\delta_{G_w,\ell}(j_\ell(\tau^\dagger_{L_w/K_v}))) \equiv
\delta_{G,\ell}(\prod_{v \mid d_L\atop{v\nmid \ell}}\tilde i_w(y_{L_w/K_v})) \,\, ({\rm mod}\,\,\ker(j^t_{\ell,*})).\]
\end{itemize}
\end{lemma}

\begin{proof} To prove claim (i) one can just follow the proof of \cite[Lem. 5.4]{BreuPhd} verbatim,
merely substituting $\mathcal A_{L/K}$ for the projective $\IZ[G]$-sublattice $\mathcal L$ of $\mathcal O_L$ that is used in loc. cit.

The property stated in claim (ii) is part of the axiomatic characterisation used by Breuning to define the elements $U_{L_w/K_v}$ in \cite[Prop. 4.4]{BreuPhd}.

To prove claim (iii) we note that elements $\delta_G(\tau^\dagger_{L/K})$ and $\delta_{G_w,\ell}(\tau^\dagger_{L_w/K_v})$
are respectively denoted by $\tau_{L/K}$ and
$T_{L_w/K_v}$ in \cite{mb2} and that the claimed congruence is thus equivalent to the equality of \cite[(36)]{BreuPhd}.
\end{proof}

\section{Results in special cases}\label{absolute case}

In this section we compute $\fra_{L/K}$ explicitly in some important special cases and also give a proof of Theorem \ref{main result}.

\subsection{Local results} The following result uses the element $\frc_{F/E}$ defined in (\ref{twisted unram char def}).

\begin{theorem}\label{special cases prop}\mpar{special cases prop} Let $E/\Ql$ be a finite extension and $F/E$ a weakly ramified Galois extension of odd degree with Galois group $\Gamma = G(F/E)$. Then $\fra_{F/E} = \frc_{F/E}$ if either $F/E$ is tamely ramified or if $E/\Ql$ is unramified and $F/E$ is both abelian and has cyclic ramification subgroup.
\end{theorem}

\begin{proof} We fix an embedding $j_\ell \colon \Qc \lra \Qlc$ and use it to identify $\widehat\Gamma$ with the set of
irreducible $\Qlc$-valued characters of $\Gamma$.

By Proposition \ref{dagger proof} and Taylor's fixed point theorem it  suffices to show that
\begin{equation}\label{eq 1}
j_{\ell,*}^t ( [\calA_{F/E}, \kappa_F, H_F] - \delta_{\Gamma, \ell}(j_\ell(\tau_E^\Gamma\cdot (\psi_{2,*}-1)(\tau_{F/E}'))) -  U_{F/E} - \frc_{F/E}) = 0
\end{equation}
with $j_{\ell,*}^t$ as in (\ref{jltstar}).

At the outset we note that $j_{\ell,*}^t(U_{F/E}) = 0$ (by \cite[Prop.~4.4]{BreuPhd}) and that if $\theta$ is any element of $F$ with $\calA_{F/E} = \calO_E[\Gamma]\cdot\theta$, then \cite[Lem.~4.16]{BreuPhd} implies
\[
[\calA_{F/E}, \kappa_F, H_F] = \delta_{\Gamma, \ell}\left( \sum_{\chi \in \widehat\Gamma}e_\chi\delta_E^{\chi(1)}\cdot\calN_{E/\Ql}(\theta \mid \chi) \right).
\]

We now assume $F/E$ is tamely ramified. In this case Remark \ref{explicit rems twisted} implies both that $\frc_{F/E}$ vanishes and
$\delta_{\Gamma,\ell}((\psi_{2,*}-1)\tau_{F/E}') = \delta_{\Gamma,\ell}((\psi_{2,*}-1)\tau_{F/E})$ and so the element on the left hand side of  (\ref{eq 1}) is equal to the image under $j_{\ell,*}^t\circ\delta_{\Gamma,\ell}$ of $x_1\cdot x_2$ where for each $\chi$ in $\widehat{\Gamma}$ one has (in terms of the notation in (\ref{explicit decomp}))
\[
x_{1,\chi} := \frac{\delta_E^{\chi(1)}}{j_\ell(\tau(\Ql, \ind_E^\Ql 1_E)^{\chi(1)})}\,\,\text{ and }\,\, x_{2,\chi} := \frac{\calN_{E/\Ql}(\theta \mid \chi)}{j_\ell(\tau(E, \psi_2(\chi) - \chi))}.
\]

The equality (\ref{eq 1}) is therefore true in this case since both $(j_{\ell,*}^t\circ\delta_{\Gamma, \ell})(x_1) = 0$ (as a consequence of the obvious local analogue of (\ref{disc pfaff})) and $(j_{\ell,*}^t\circ\delta_{\Gamma, \ell})(x_2) = 0$, as indicated by Erez in the proof of
\cite[Prop.~8.2]{Erez}.

In the remainder of the argument we assume that $E/\Ql$ is unramified and $F/E$ is both abelian and has cyclic ramification subgroup.
 The proof in this case will heavily rely on the computations of \cite{BleyCobbe}
(which in turn rely on the work of Pickett and Vinatier in \cite{PV}) and so, for convenience, we switch to
the notation introduced in \S3.1 of loc. cit. (so that $F, E$ and $\Gamma = G({F/E})$ are now replaced by $N$, $K$ and $G$ respectively).

In particular, we define $\alpha_M$ as in \cite[just before Lem.~5.1.4]{BleyCobbe}, let $\theta_2 \in K'$ be such that $\OKG \cdot\theta_2 = \calO_{K'}$ and $T_{K'/K}(\theta_2) = 1$ and recall that the product $\theta = \alpha_M\cdot\theta_2$ satisfies $\calA_{N/K} = \OKG\cdot\theta$.
 (In this regard we observe that the assumption made in \cite{BleyCobbe} that $[K:\Qp]$ and $[K':K]$ are coprime is not needed for the results obtained in \S5 of loc. cit.)

Each character $\psi \in \widehat{G}$ is of the form $\chi\phi$ with an unramified character
$\phi$ of $G_{N/M}$ and $\chi$ a character of $G_{N/K'}$ and from \cite[Prop.~5.1.5]{BleyCobbe} one has
\[
\frac{\calN_{K/\Ql}(\theta \mid \chi\phi)}{\tau(K, \chi\phi)} =
\begin{cases}
  \calN_{K/\Ql}(\theta_2 \mid \phi), & \text{ if } \chi = \chi_0, \\
  p^{-m}\cdot\chi(4)\cdot\calN_{K/\Ql}(\theta_2 \mid \phi)\cdot \phi(p^2), & \text{ if } \chi \ne \chi_0,
\end{cases}
\]
where here and in the following we omit each occurrence of $j_\ell$ in our notation.

Now the proof of \cite[Prop.~5.2.1]{BleyCobbe} shows that $\tau(K, \chi\phi) = \tau(\Ql, i_K^\Ql(\chi\phi))$ and so
(\ref{local label product}) implies
\[
\tau_K^G\cdot (\psi_{2,*}-1)(\tau_{N/K}') = ( \sum_{\chi, \phi} e_{\chi\phi}\tau(K, \chi\phi))\cdot  J_{2, N/K} \cdot (\psi_{2,*}-1)(y_{N/K}^{-1}).
\]
It follows that for each $\chi$ and $\phi$ one has
\begin{eqnarray*}
\frac{\calN_{K/\Ql}(\theta \mid \chi\phi)}{( \tau_K^G\cdot(\psi_{2,*}-1)(\tau_{N/K}') )_{\chi\phi}}
&=& \frac{\calN_{K/\Ql}(\theta \mid \chi\phi)}{\tau(K, \chi\phi)\cdot \tau(K, \psi_2(\chi\phi) - 2\chi\phi)\cdot y(K, \chi\phi - \psi_2(\chi\phi))}\\
&=&
    \begin{cases}
      \calN_{K/\Ql}(\theta_2 \mid \phi) \frac{\tau(K, 2\chi\phi -\psi_2(\chi\phi))}{y(K, \chi\phi -\psi_2(\chi\phi))}, & \text{ if } \chi = \chi_0, \\
      p^{-m}\cdot\chi(4)\cdot\phi(p^2)\cdot \calN_{K/\Ql)}(\theta_2 \mid \phi) \frac{\tau(K, 2\chi\phi -\psi_2(\chi\phi))}{y(K, \chi\phi -\psi_2(\chi\phi))}, & \text{ if } \chi \ne \chi_0.
    \end{cases}
\end{eqnarray*}

Furthermore, one has $\phi(p^2) = \phi( (p^2, K'/K)) = \phi(\sigma^2)$ (by \cite[XIII, \S 4, Prop.~13]{LocalFields}) and so $\frc_{N/K}$ is equal to the element $x_3$ of $\QQ_\ell^c[G]^\times$ that is characterized by the following equalities for each $\chi$ and $\phi$
\[
 x_{3, \chi\phi} =  \begin{cases}
    y(K, \chi\phi -\psi_2(\chi\phi))^{-1}, & \text{ if } \chi = \chi_0, \\
    \phi(p^2) y(K, \chi\phi -\psi_2(\chi\phi))^{-1}, & \text{ if } \chi \ne \chi_0.
  \end{cases}
\]

Taken together, these facts imply that  (\ref{eq 1}) is valid if $\ker(j_{\ell,*}^t \circ \delta_{G, \ell})$ contains the
element $x_4$ of $\QQ_\ell^c[G]^\times$ defined by
\[
x_{4,\chi\phi} =
\begin{cases}
        \calN_{K/\Ql}(\theta_2 \mid \phi)\cdot \tau(K, 2\chi\phi -\psi_2(\chi\phi)), & \text{ if } \chi = \chi_0, \\
      p^{-m}\cdot\chi(4)\cdot\calN_{K/\Ql}(\theta_2 \mid \phi)\cdot\tau(K, 2\chi\phi -\psi_2(\chi\phi)), & \text{ if } \chi \ne \chi_0.
\end{cases}
\]

Now, as in the proof of \cite[Th.~6.1, p. 1243]{BleyCobbe}, one can show that $\ker(j_{\ell,*}^t \circ \delta_{G, \ell})$ contains the element
$x_4'$ of $\QQ_\ell^c[G]^\times$ for which at all $\chi$ and $\phi$ one has
\[
x_{4,\chi\phi}' =
\begin{cases}
        \calN_{K/\Ql}(\theta_2 \mid \phi),  & \text{ if } \chi = \chi_0, \\
      \chi(4)\cdot \calN_{K/\Ql}(\theta_2 \mid \phi), & \text{ if } \chi \ne \chi_0.
\end{cases}
\]
In addition, \cite[Lem.~5.1.2]{BleyCobbe} implies that for all $\chi$ and $\phi$ one has $\tau(K, 2\chi\phi - \psi_2(\chi\phi)) = \tau(K, 2\chi - \chi^2)$.

The required equality $(j_{\ell,*}^t \circ \delta_{G, \ell})(x_4)=0$ is thus true if and only if $(j_{\ell,*}^t \circ \delta_{G, \ell})(x_5) =0$ with
$x_5$ the element of $\QQ_\ell^c[G]^\times$ for which at each $\chi$ and $\phi$ one has
\[
x_{5,\chi\phi} =
\begin{cases}
        1, & \text{ if } \chi = \chi_0, \\
      p^{-m}\tau(K, 2\chi -\chi^2), & \text{ if } \chi \ne \chi_0.
\end{cases}
\]

But, by the last displayed formula in the proof of  \cite[Prop.~3.9]{PV}, for each non-trivial character $\chi$ one has
\[ \tau(K, \chi) = p^m\cdot \chi(c_\chi^{-1})\cdot \psi_K(c_\chi^{-1}) \,\,\text{ and }\,\, \tau(K, \chi^2) = p^m \cdot\chi^2((c_\chi/2)^{-1}) \cdot\psi_K((c_\chi/2)^{-1}),\]
with $\psi_K$ the standard additive character and $c_\chi$ as described in \cite[Prop.~3.9]{PV}.

It follows that $\tau(K, 2\chi - \chi^2) = p^m\cdot \chi(4)^{-1}$ for non-trivial characters $\chi$ and hence that $x_{5,\chi\phi} = \chi(4)^{-1}$ for all $\chi$ and $\phi$. Given this description, it is clear that $x_5 \in \ker(j_{\ell,*}^t \circ \delta_{G, \ell})$, as required to complete the proof of (\ref{eq 1}) in this case.
\end{proof}

\subsection{Global results}\label{glob consequences}\mpar{glob consequences} In this section we derive several consequences of Theorem \ref{special cases prop}, including a proof of Theorem \ref{main result}.

\subsubsection{}\label{proof of main result}We shall first give a proof of Theorem \ref{main result}.

Following Proposition \ref{rationality thm}, for each prime $\ell$ we write $\mathfrak{a}_{L/K,\ell}$ for the image of $\mathfrak{a}_{L/K}$ in $K_0(\mathbb{Z}_\ell[G],\mathbb{Q}_\ell[G])$.

Then Theorem \ref{decomp thm} combines with the vanishing of $\fra_{F/E}$ for each tamely ramified extension $F/E$ of local fields (as proved in Theorem \ref{special cases prop}) to reduce the proof of Theorem \ref{main result}(i) to showing that for each $\ell$ for which there is an $\ell$-adic place $v$ in $\mathcal W_{L/K}$ the element $\Fa_{L/K,\ell}$ belongs to
$K_0(\IZ_{\ell}[G],\IQ_{\ell}[G])_{\rm tor}$.

In view of the explicit description of $K_0(\IZ_{\ell}[G],\IQ_{\ell}[G])_{\rm tor}$ given by the second author in \cite[Th. 4.1]{Burns}, it is thus enough to prove that for each such prime $\ell$ one has $\pi_{H/J}^H(\rho_H^G(\Fa_{L/K,\ell}))=0$ for every cyclic subgroup $H$ of $G$ and every subgroup $J$ of $H$ with $|H/J|$ prime to $\ell$.

Invoking the result of Theorem \ref{funct thm} it is thus enough to show that $\Fa_{F/E,\ell}$
vanishes for all towers of number fields $K\subseteq E\subseteq F\subseteq L$ with $L/E$ cyclic and the degree $[F:E]$
prime to $\ell$. However, in any such case, all $\ell$-adic places of $E$ are tamely ramified in $F/E$ and so
Theorem \ref{special cases prop} in conjunction with Theorem \ref{decomp thm} (or (\ref{l adic eq}))
implies $\Fa_{F/E,\ell}$ vanishes, as required.

Claims (ii) and (iii) of Theorem \ref{main result} will follow from the same argument used to prove Corollary \ref{coro 1}.

Finally we note that claim (iv) follows directly from the definition of $\fra_{L/K}$ and the facts that $H_L$ is a free $G$-module and
 $\partial_{\Ze, \Qu, G}\circ \delta_G$ is the zero homomorphism.

This completes the proof of Theorem \ref{main result}.

\subsubsection{}In order to describe a global consequence of Theorem \ref{special cases prop} we define an `idelic twisted unramified characteristic' by setting
\begin{equation}\label{idelic twisted unram char def}
\frc_{L/K} := \sum_\ell \sum_{v \mid \ell} i_{G_w, \Ql}^G( \frc_{L_w/K_v}).
\end{equation}
If $v$ is at most tamely ramified in $L/K$, then $\frc_{L_w/K_v}$ vanishes. This shows $\frc_{L/K}$ is a well-defined element in $K_0(\ZG, \QG)$ and that
\[
\frc_{L/K,\ell}=
\begin{cases}
  0, & \text{ if } \ell \not\in \calW_{L/K}^\Qu, \\
  \sum_{v \mid \ell} i_{G_w, \Ql}^G( \frc_{L_w/K_v} ), &  \text{ if } \ell \in \calW_{L/K}^\Qu.
\end{cases}
\]

In particular, by combining Theorems \ref{decomp thm} and \ref{special cases prop} one obtains the following result.

\begin{coro}\label{vinatier thm}\mpar{vinatier thm}
Let $L/K$ be a weakly ramified odd degree Galois extensions of number fields.
 Then $\mathfrak{a}_{L/K} = \frc_{L/K}$ whenever all of the following conditions are satisfied at each $v$ in $\calW_{L/K}$:
\begin{itemize}
\item [(i)] The decomposition subgroup of $v$ is abelian;
\item[(ii)] The inertia subgroup of $v$ is cyclic;
\item[(iii)] The extension $K_v/\mathbb{Q}_\ell$ is unramified, where $\ell = \ell(v)$ denotes the residue characteristic.
\end{itemize}
\end{coro}

\begin{remark}
Extensive numerical computations suggest that the equality $\fra_{L/K} = \frc_{L/K}$ proved in Corollary \ref{vinatier thm}
may well be valid in all cases (see \S\ref{subsub comp for fra} for more details).
\end{remark}

Corollary \ref{vinatier thm} immediately combines with Theorem \ref{main result}(ii) and (iii) to give the following explicit consequence concerning the structures  discussed in Examples \ref{example2} and \ref{example1}.

\begin{coro}\label{hopeful}
Under the hypotheses of Corollary \ref{vinatier thm} one has
\[
[\mathcal{A}_{L/K},h_{L,\bullet}] = \Pi_G^{\mathrm{met}}(\frc_{L/K}) + \varepsilon_{L/K}^{\rm met} \text{ and }
{\rm Disc}(\mathcal{A}_{L/K},t_{L/K}) =  \Pi_G^{\mathrm{herm}}(\frc_{L/K}) + \varepsilon^{\rm herm}_{L/K}.
\]
\end{coro}

It is therefore of interest to know when the classes $\Pi_G^{\mathrm{met}}(\frc_{L/K})$ and $\Pi_G^{\mathrm{herm}}(\frc_{L/K})$ vanish and the next result shows that this is often the case.


\begin{lemma}\label{cwr vanish}\mpar{cwr vanish}
The images of $\frc_{L/K}$ in each of the groups $\Cl(G), A(G)$ and $\HCl(G)$ all vanish if for
each $v \in \calW_{L/K}$ one has either $I_w = G_w$ or $I_w$ is of prime power order.\end{lemma}

\begin{proof} We show that each of the individual terms in the definition of $\frc_{L/K}$ projects to zero.
We fix $v$ in $\calW_{L/K}$ and set $\ell := \ell(v)$ and
 $\lambda_w :=(1 - e_{I_w}) + \sigma_w^{-1}e_{I_w} $. If $G_w = I_w$, then $\lambda_w = 1$. In the other case $I_w$ is necessarily of $\ell$-power order. Hence for any prime $p \ne \ell$ we have
$\delta_{G_w}(\lambda_w) = 0$ in $K_0(\Zp[G_w], \Qp[G_w])$ since $\lambda_w \in \mathrm{Nrd}_{\Qp[G_w]}(\Zp[G_w]^\times)$.

Therefore one has $\pi_{G, \ell}( i_{G_w}^G ( \delta_{G_w}(\lambda_w)) =
  i_{G_w}^G ( \delta_{G_w}(\lambda_w ))$ in $K_0(\Ze[G], \Qu[G])$, where $\pi_{G, \ell}$ is the homomorphism between relative $K$-groups defined
in (\ref{localization map}).

We next show that $\frc_v := i_{G_w}^G ( \delta_{G_w}(\lambda_w ))$ belongs to both
$\ker(\partial^{1,1}_G\circ h^{\rm rel}_G)$ and $\ker(\partial^{2,1}_G\circ h^{\rm rel}_G).$

To do this we recall first that for $\alpha = (\alpha_\chi)_{\chi\in \widehat G}$ in $\Qu^c[G]^\times$ the element
$h_G^{\mathrm{rel}}(\delta_G(\alpha))$ is represented by the function
$\chi \mapsto (1, \alpha_\chi)$. Thus the global analogue of the commutative diagram (\ref{ind comm diagram})
implies that $h_G^{\mathrm{rel}}(\frc_v)$ is represented by the pair $(1, \theta)$ with
\[
\theta(\chi) = \prod_{\phi\in \widehat{G_w/I_w}} \phi(\sigma_w^{-1})^{\langle\mathrm{res}_{G_w}^G(\chi), \phi\rangle_{G_w}}.
\]

The elements $\partial^{1,1}_G(h^{\rm rel}_G(\frc_v))$ and $\partial^{2,1}_G(h^{\rm rel}_G(\frc_v))$
are therefore represented by the pairs $(1,|\theta|)$ and $(1,\theta^s)$, respectively,
and so it is enough to show that the maps $|\theta|$ and $\theta^s$
are both trivial.

Since $\theta(\chi)$ is a root of unity one has $|\theta|(\chi) = |\theta(\chi)| = 1$, and so $|\theta|$ is trivial.

In addition, the triviality of $\theta^s$ follows from the fact that if $\chi$ is a symplectic character of $G$, then both
$\langle\mathrm{res}_{G_w}^G(\chi), \phi\rangle_{G_w} = \langle \mathrm{res}_{G_w}^G(\chi), \overline{\phi}\rangle_{G_w}$ and $\phi(\sigma_w) \overline{\phi}(\sigma_w) = 1.$
\end{proof}

\begin{remark}\label{not vanish} In connection with Lemma \ref{cwr vanish} we note that if $L_w/K_v$ is weakly ramified and abelian, then class field theory implies $I_w$ is of prime-power order (as a consequence of \cite[Cor.~2, p. 70]{LocalFields}). In fact, at this stage we know of no example in which the projection of $\frc_{L/K}$ to any of the groups $\Cl(G), A(G)$ and $\HCl(G)$ does not vanish. It is, however, not difficult to show that the element $\frc_{L/K}$ itself does not always vanish. For example, if $G$ is abelian, then $K_0(\ZG, \QG)$ identifies with the group of invertible $\ZG$-sublattices of $\QG$. In particular, if $L/\QQ$ is an abelian $p$-extension in which, for any $p$-adic place $w$ of $L$, one has $I_{w} \subsetneq G_{w} = G$,
then $(1-e_{I_w}) + \sigma_{w}^{-1}e_{I_w}$ does not belong to $\mathbb{Z}[G]$ and so $\frc_{L/\QQ}\not= 0$.
\end{remark}

\begin{remark}\label{simpler variant} The element $\frc_{L/K}$ is in general different from,
and better behaved than, the simpler variant
 $\delta_G((1 - \psi_{2,\ast})(y_{L/K}))$. In particular, whilst it is straightforward to show that $\frc_{L/K}$
enjoys the same functorial properties under change of extension as those described in Theorem \ref{funct thm},
the same is not true of $\delta_G((1 - \psi_{2,\ast})(y_{L/K}))$. \end{remark}

\section{Effective computations and Vinatier's Conjecture}\label{vinatier section}

In this section we first refine Corollary \ref{finiteness coro} by explaining how to make an effective computation of
the set of realisable classes $R^{\rm wr}_K(\Gamma)$.

We then apply this observation to consider a conjecture of Vinatier in the setting of two natural infinite families of extensions which will then be investigated numerically in \S\ref{numerical examples}.

In \S\ref{p cubed section} we consider the family of extensions of smallest degree for which Vinatier's Conjecture is not currently known to be valid and, whilst studying this case, we obtain evidence (described in Proposition \ref{new connection}) that $\mathfrak{a}_{L/K}$ may be controlled by the idelic twisted unramified characteristic $\mathfrak{c}_{L/K}$ in cases beyond those considered in Corollary \ref{vinatier thm}.

Motivated by this last rather surprising observation, we consider in \S\ref{mixed degree extensions} a family of extensions of smallest possible degree for which the projection of $\mathfrak{c}_{L/K}$ to ${\rm Cl}(G(L/K))$ might not vanish, and hence that a close link between $\mathfrak{a}_{L/K}$ and $\mathfrak{c}_{L/K}$ need not be consistent with the validity of Vinatier's Conjecture.

In all of the cases that we compute, however, we find both that Vinatier's Conjecture is valid and the projection of $\mathfrak{c}_{L/K}$ to ${\rm Cl}(G(L/K))$ vanishes.

At the same time, our methods also give a proof of the central conjecture of \cite{BleyBurns} for a new, and infinite, family of wildly ramified Galois extensions of number fields.

\subsection{The general result}
Recall that for each number field $K$ and finite abstract group $\Gamma$ of odd order
we write ${\rm WR}_K(\Gamma)$ for the set of fields $L$ that are weakly ramified odd degree Galois extensions of $K$
and for which $G(L/K)$ is isomorphic to $\Gamma.$

\begin{theorem}\label{main cor2}
Let $K$ be a number field and $\Gamma$ a finite abstract group whose order is both odd and coprime to the number
of roots of unity in $K$.

Then there exists a finite set ${\rm WR}^\ast_K(\Gamma)$ of Galois extensions $E$
of $K$ which have all of the following properties:

\begin{itemize}\item[(i)] there exists an injective homomorphism of groups $i_E:G({E/K})\to \Gamma$;
\item[(ii)] there exists a unique place $v$ of $K$ that ramifies both wildly and weakly in $E$ and for which there exists a unique place $w$ of $E$ above $v$;
\item[(iii)] all places of $K$ other than $v$ that divide $|\Gamma|$ are completely split in $E/K$;
\item[(iv)] for each $L$ in ${\rm WR}_K(\Gamma)$ and every $E$ in ${\rm WR}^\ast_K(\Gamma)$ there exists an integer $n_{L,E}\in \{0,1\}$ so that in $K_0(\IZ[\Gamma],\IQ^c[\Gamma])$ one has
\[i_{L,\ast}(\Fa_{L/K})=\sum_{E\in {\rm WR}^\ast_K(\Gamma)}n_{L,E}\cdot {\rm i}^\Gamma_{{\rm im}(i_E)}(i_{E,\ast}(\Fa_{E/K})).\]
\end{itemize}
\end{theorem}

\begin{proof} We recall first that for each place $v$ of  $K$ the set $R_v(K, \Gamma)$
of isomorphism classes of Galois extensions $E/K_v$ for which $G(E/K_v)$ is isomorphic to a subgroup of $\Gamma$ is finite.

We next fix a weakly ramified Galois extension $L/K$ for which the group $G := G(L/K)$ is isomorphic to the given group $\Gamma$.
We recall that Theorems \ref{decomp thm} and \ref{special cases prop} combine to imply that there is a finite sum decomposition
\begin{equation}\label{almost there4}
\Fa_{L/K} = \sum_{\ell \in \mathcal W^\Qu_{L/K}}\sum_{v\mid \ell}{\rm i}^G_{G_w,\IQ_{\ell}}(\Fa_{L_w/K_v})
\end{equation}

For each place $v$ in this sum the (weakly ramified) Galois extension $L_w/K_v$ is isomorphic to one of the
Galois extensions $E/K_v$ in the finite set $R_v(K, \Gamma)$.

Further, since we are assuming $|\Gamma|$ is coprime to the number of roots of unity in $K$ a result of Neukirch \cite[Cor. 2, p. 156]{neuk solvable} implies that there exists a finite Galois extension $\tilde E/K$ with both of the following properties:

\begin{itemize}
\item[(P1)] $\tilde E$ has a unique place $\tilde w$ above $v$ and the completion $\tilde E_{\tilde w}/K_v$ is isomorphic to $E/K_v$ (and hence to $L_w/K_v$);
\item[(P2)] if $v'$ is any place of $K$ which divides $|\Gamma|$, and $v'\not= v$, then $v'$ is totally split in $\tilde E/K$.
\end{itemize}

These conditions imply that the global extension $\tilde E/K$ is weakly ramified and that the isomorphism of
$\tilde E_w/K_v$ with $L_w/K_v$ induces a natural identification
\begin{equation}\label{identity}
G(\tilde E/K) \cong G(L_w/K_v) \cong G_w.
\end{equation}

In addition, since $v$ is the only place of $K$ that is not tamely ramified in $\tilde E/K$
the results of Theorem \ref{decomp thm} and \ref{special cases prop}(i) combine to imply
\begin{equation}\label{unique summand}
\Fa_{\tilde E/K} = \Fa_{L_w/K_v}.
\end{equation}

%



We now define ${\rm WR}^\ast_K(\Gamma)$ to be the finite set of extensions $\tilde E/K$ that are
obtained from the above construction as $v$ runs over the places of $K$ that divide $|\Gamma|$. We note that this set satisfies the claimed property (i) as a consequence of the isomorphisms (\ref{identity}),
it satisfies property (ii) and (iii) as a consequence of the properties (P1) and (P2) above and it satisfies property (iv) as a consequence of the equalities (\ref{almost there4}) and (\ref{unique summand}).
\end{proof}

\begin{remark}\label{Yamagishi remark}
The above argument also shows that $|{\rm WR}^\ast_K(\Gamma)| \le \sum_{v \mid |\Gamma|}\tilde\nu(K_v, \Gamma)$ where $\tilde\nu(K_v, \Gamma)$ denotes the number of non-isomorphic Galois extensions
of $K_v$ whose Galois group is isomorphic to a subgroup of $\Gamma$. In this context we recall that if $\Gamma$ is a $p$-group then $\tilde\nu(K_v, \Gamma)$ is explicitly computed by work of Shafarevich \cite{Shafarevics} and Yamagishi \cite{Yamagishi}. We also recall that Pauli and Roblot \cite{PauliRoblot} have developed an algorithm for the computation of all extensions of a $p$-adic field of a given degree. One can therefore use the results of  \cite{Shafarevics} and \cite{Yamagishi} to design an algorithm to compute all $p$-extensions with a given $p$-group (see \cite[\S~10]{PauliRoblot}).
\end{remark}

\begin{remark} For any number $k$ and any finite group $\Gamma$ whose order is both odd and coprime to the number
of roots of unity in $k$, write ${\rm WR}'_k(\Gamma)$ for the set of weakly ramified odd degree Galois extensions $L/K$  with $k \sseq K$ and such that $G(L/K) \simeq \Gamma$ and $K_v = k_{v(k)}$ for each place $v$ of $K$ that ramifies wildly in $L$. Then a closer analysis of the proof of Theorem \ref{main cor2} shows that the stated result remains valid after one replaces each occurrence of $K$ by $k$ and then, in claim (iv), one replaces the terms $L$, ${\rm WR}_k(\Gamma)$ and $n_{L,E}$ by $L/K$, ${\rm WR}'_k(\Gamma)$ and $n_{L/K,E}$ respectively. This stronger version of Theorem \ref{main cor2} makes clear the advantage of the local nature of our computations. \end{remark}

%

\subsection{Applications to Vinatier's Conjecture}\label{vinatier p-cubed}

Vinatier has conjectured that for any weakly ramified odd degree Galois extension $L$ of $\QQ$ the $G(L/\Qu)$-module $\mathcal{A}_{L/\QQ}$ is free  (see \cite[\S1, Conj.]{Vinatier2}) and we now apply our techniques to study this conjecture.

\subsubsection{}We first reformulate the conjecture in terms of the elements $\fra_{L/K}$ (global) and
$\fra_{F/E}$ (local).

If $F/E$ is a Galois extension of $\ell$-adic fields, then we use the decomposition (\ref{decomp}) to view $\fra_{F/E}$ as an element of $K_0(\Ze[G(F/E)], \Qu[G(F/E)])$.

\begin{prop}\label{vinatier lemma} The following are equivalent:
  \begin{itemize}
  \item [(i)] For all odd degree weakly ramified Galois extensions $L/K$ of number fields the $G(L/K)$-module $\calA_{L/K}$ is free.
  \item[(ii)] For all odd degree weakly ramified Galois extensions $L/K$ of number fields the element  $\mathfrak{a}_{L/K}$ projects to zero in $\Cl(G(L/K))$.
  \item[(iii)] For all odd degree weakly ramified Galois extensions $F/E$ of local fields the element $\mathfrak{a}_{F/E}$ projects to zero in $\Cl(G(F/E))$.
  \end{itemize}
\end{prop}

\begin{proof} The equivalence of (i) and (ii) is Lemma \ref{vinatier lemma 2} below and (ii) follows directly from (iii) and Theorem \ref{decomp thm}.

We finally assume (ii) and for a local extension $F/E$ we choose
a number field $K$ and a place $v$ of $K$ such that $K_v$ is isomorphic to $E$ and $|G(F/E)|$ is coprime to the number of roots of unity in $K$.  (Since $G(F/E)$ is of odd order the existence of such a field $K$ is easily implied by the main result of Henniart \cite{Henniart}.)

Then by the construction in the proof of Theorem \ref{main cor2} we find a global extension $\tilde E / K$ with the properties (P1) and (P2).

It follows that $\fra_{\tilde E/K} = \fra_{F/E}$, and hence
that $\fra_{F/E}$ projects to zero in $\Cl(G(F/E))$, as required to prove (iii). \end{proof}

\begin{lemma}\label{vinatier lemma 2} Let $L/K$ be an odd degree weakly ramified Galois extension of number fields of group $G$. Then the  $G$-module $\calA_{L/K}$ is free if and only if the image of $\mathfrak{a}_{L/K}$ in $\mathrm{Cl}(G)$ vanishes.
\end{lemma}

\begin{proof} By Theorem \ref{main result}(iv) one has $\partial_{\Ze, \Qu, G}(\fra_{L/K}) = [\calA_{L/K}]$ in $\Cl(G)$.
 Given this, the equivalence of the stated conditions follows immediately from the fact that, as $G$ has odd order, a finitely generated projective $G$-module is free if and only if its class in ${\rm Cl}(G)$ vanishes.
\end{proof}

\subsubsection{}\label{p cubed section} By \cite{Vinatier1} Vinatier's conjecture is known to be true for extensions $L/\Qu$ with the property that  the decomposition group of each wildly ramified prime is abelian. The family of non-abelian Galois extensions of degree $p^3$, for some odd prime $p$, is thus the family of smallest possible degree for which Vinatier's conjecture is not known to be valid. Such extensions were considered (in special cases) by Vinatier in \cite{Vinatier3}.

In the following result we study the number of corresponding such extensions of the base field $\Qp$. This result (which will be proved at the end of this section) shows that the bounds on the number of such extensions that are discussed in Remark \ref{Yamagishi remark} can be improved if one imposes ramification conditions.

\begin{prop}\label{degree p extensions}
Let $p$ be an odd prime. Then there exist exactly $p$ (non-isomorphic) weakly ramified non-abelian Galois extensions of $\QQ_p$
of degree $p^3$. Exactly one of these extensions has exponent $p$ and the remaining $p-1$ extensions have exponent $p^2$.
\end{prop}

As in the proof of Theorem \ref{main cor2}, for
each odd prime $p$ and each weakly ramified non-abelian Galois extension $F$ of $\Qp$
of degree $p^3$ there exists a weakly ramified Galois extension $N/\Qu$ of degree $p^3$ such that $N$ has a unique $p$-adic place $w$ and the corresponding completion $N_w/\Qp$ is isomorphic to $F/\Qp$. This fact motivates the following definitions.

 For each odd prime $p$ we fix a set $\calF(p)$ of $p$ weakly ramified Galois extensions $N/\Qu$ of degree $p^3$
such that each field $N$ has a unique $p$-adic place $w(N)$ and the corresponding completions $N_{w(N)}/\QQ_p$ give the full set of local extensions that are described in Proposition \ref{degree p extensions}.

For a finite set $P$ of odd primes we define $\calL(P)$ to be the set of Galois extensions of number fields $L/K$ such that
$\calW_{L/K}^\Qu \sseq P$ and for each place $v$ in $\calW_{L/K}$ one has both $K_v = \mathbb{Q}_{\ell(v)}$ and the order of the decomposition subgroup in $G(L/K)$ of any place of $L$ above $v$ divides $\ell(v)^3$.

\begin{theorem}\label{nice explicit} For any finite set of odd primes $P$ the following conditions are equivalent:
\begin{itemize}
\item[(i)] For all $L/K$ in $\calL(P)$ the $G(L/K)$-module $\calA_{L/K}$ is free.

\item[(ii)] For all $N/\Qu$ in the finite set $\cup_{p \in P}\calF(p)$ the $G(N/\Qu)$-module
$\calA_{N/\Qu}$ is free.
\end{itemize}
\end{theorem}

\begin{proof}
Obviously (i) implies (ii). For the reverse implication fix $L/K$ in $\calL(P)$.
By Lemma \ref{vinatier lemma 2} we have to show that the element $\fra_{L/K}$ projects to zero in
$\Cl(G)$. By Theorem \ref{decomp thm} together with Proposition \ref{special cases prop}(i) we have
\[
\fra_{L/K} = \sum_{v \in \calW_{L/K}} i_{G_w, \Qu_{\ell(v)}^c}^G (\fra_{L_w/K_v}).
\]
It is therefore enough to show that each of the terms $ \fra_{L_w/K_v}$ projects to zero in $\Cl(\Ze[G_w])$.
By our assumptions $K_v = \Qp$ for a prime $p \in P$ and
$G({L_w/\Qp})$ is a $p$-group of order at most $p^3$. If   $|G({L_w/\Qp})| \le p^2$ then $L_w/\Qp$ is abelian
and  $\fra_{L_w/\Qp} = 0$ by the relevant case of Theorem \ref{special cases prop}. If $|G_{L_w/\Qp}| = p^3$ then by the definition of $\calL(P)$ the local extension
 $L_w/\Qp$ is the localisation of one of the extensions $N/\Qu$ in $\calF(p)$, so that we have $\fra_{N/\Qu} = \fra_{L_w/\Qp}$.
 The claim now follows from Lemma \ref{vinatier lemma 2}.\end{proof}

In the rest of this section we give the postponed proof of  Proposition \ref{degree p extensions}.

As a first step we recall that there are two isomorphism classes of non-abelian groups of order $p^3$, with respective presentations
\begin{equation}\label{presentations}
\begin{cases} &\langle a,b \mid a^{p^2} = 1 = b^p, b^{-1}ab = a^{1+p}\rangle,\\
 &\langle a, b, c\mid a^p = b^p = c^p = 1, ab = bac, ac = ca, bc = cb\rangle,\end{cases}
\end{equation}
the first having exponent $p^2$ and the second exponent $p$ (see, for example, \cite[\S4.4]{hall}). In both cases the centre $Z(G)$ of the group $G$
has order $p$ (being generated by $a^p$ and $c$ respectively) and the quotient group $G/Z(G)$ is isomorphic to $\ZZ/p\ZZ\times \ZZ/p\ZZ$.

Any weakly ramified non-abelian Galois extension $L$ of $\QQ_p$
of degree $p^3$ must thus contain a subfield $E$ that is Galois over $\QQ_p$ and
such that both $G({L/E})$ is central in $G({L/\QQ_p})$ and $G({E/\QQ_p})$ is isomorphic to $\ZZ/p\times \ZZ/p$. Since
$\QQ_p^\times/(\QQ_p^\times)^p$ has order $p^2$ local class field theory implies 
$E$ is the compositum of the unique subextension $E_1$ of $\QQ_p(\zeta_{p^2})$ of degree $p$ over $\QQ_p$ and of the unique
unramified extension $E_2$ of $\QQ_p$ of degree $p$ (and hence is weakly ramified, as required).
In the sequel we set $G := G(L/\Qp), H := G(L/E), \Gamma := G(E/\QQ_p)$ and $\Delta := G(E/E_1)$.

If $L/E$ is a weakly ramified degree $p$ extension such that $L/\Qp$ is Galois, then $L/\Qp$ is weakly ramified. Indeed,
$G_2 \cap H = H_2 = 1$ and hence $G_2 \simeq G_2H/H$. By Herbrand's theorem we obtain
$G_2H/H = (G/H)_2$, which is trivial since $E/\Qp$ is weakly ramified.
The required fields therefore correspond to weakly ramified degree $p$ extensions $L$ of $E$ which are Galois over
$\QQ_p$.

For each subfield $F$ of $E$ we write $\mathfrak{p}_F$ for the maximal ideal of the valuation ring $\mathcal{O}_F$ of $F$,
$U_F^{(i)}$ for each natural number $i$
for the group $1+ \mathfrak{p}_F^i$ of $i$-th
principal units of $F$ and $\mu_F'$ for the maximal finite subgroup of $F^\times$ of order prime to $p$.
If $L/F$ is abelian we also write ${\rm rec}_{L/F}$ for the reciprocity map $F^\times \to G_{L/F}$.

If $L/E$ is unramified, then the ramification degree
of $L/\QQ_p$ is $p$ so that $L$ contains both $E_1$ and the unramified extension of $\QQ_p$ of degree $p^2$ and so $L$ is abelian over $\QQ_p$.

On the other hand, if $L/E$ is ramified, then the inertia subgroup $G_0$ has order $p^2$.
 In addition, since $L/\Qp$ is assumed to be weakly ramified, the group $G_0 = G_1$ identifies with $G_1/G_2$ and so is isomorphic to a subgroup of $U_{L}^1 /U_L^2$ and therefore has exponent dividing $p$. It follows that $G_0$ is not cyclic and hence that $L/\QQ_p$ is not abelian. We have therefore shown that $L/\Qp$ is abelian if and only if $L/E$ is unramified.

In summary, there is thus a field diagram of the following sort.

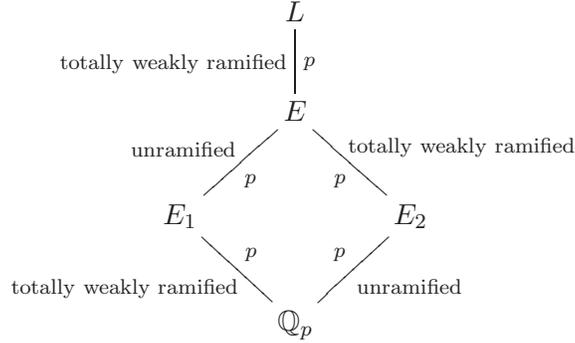
\begin{equation*}
\xymatrix{& L \\
&  E \ar@{-}[u]_{p}^{{\rm totally}\,\, {\rm weakly }\,\, {\rm ramified}} \\
E_1 \ar@{-}[ur]_{p}^{\rm unramified} & & E_2 \ar@{-}[ul]_{{\rm totally}\,\, {\rm weakly }\,\, {\rm ramified}}^p\\
& \QQ_p \ar@{-}[ul]_{p}^{{\rm totally}\,\, {\rm weakly }\,\, {\rm ramified}} \ar@{-}[ur]_{\rm unramified}^p}
\end{equation*}

By an easy exercise one checks that $L/E$ is weakly ramified if and only if the upper ramification subgroup
   $H^2$ vanishes. By local class field theory, the desired extensions $L$ are therefore in bijective correspondence
with subgroups $N$ of $E^\times$ that are $\Gamma$-stable (as $L/\QQ_p$ is Galois), contain
$U_E^{(2)}$ (by \cite[Cor.~3 p. 228]{LocalFields}), contain $E^{\times p}$ (as $E^\times / N$ has exponent $p$) and contain
$I_\Gamma(E^\times)$ (as $\Gamma$ acts trivial on $E^\times /N \simeq Z(G)$), where $I_\Gamma$ denotes the augmentation ideal of $\ZZ[\Gamma]$.

We note next that there are isomorphisms of abelian groups
\begin{equation}\label{composite iso} ( U_E^{(1)}/U_E^{(2)})_{\Gamma} \cong  \left(\left( \mathfrak{p}_E/\mathfrak{p}_E^2\right)_{\Delta}\right)_{\Gamma/\Delta}
\cong  \left(\mathfrak{p}_{E_1}/\mathfrak{p}_{E_1}^2\right)_{\Gamma/\Delta} \cong \left(\ZZ/p\ZZ\right)_{\Gamma/\Delta} = \ZZ/p\ZZ\end{equation}
where the first map is induced by the natural isomorphism $U_E^{(1)}/U_E^{(2)}\cong \mathfrak{p}_E/\mathfrak{p}_E^2$.
The second isomorphism is induced by the field-theoretic trace
${\rm Tr}_{E/E_1}$. Indeed, since $E/E_1$ is unramified, the induced map
$\left(\mathfrak{p}_E/\mathfrak{p}_E^2\right)_{\Delta} \cong \mathfrak{p}_{E_1}/\mathfrak{p}_{E_1}^2$ is
surjective with kernel $\hat H^{-1}(\Delta, \frp_E/\frp_E^2)$, which is trivial since $\mathfrak{p}_E^i$
is $\Zp[\Delta]$-free for each non-negative integer $i$.
The third is induced by
the fact that $\mathfrak{p}_{E_1}/\mathfrak{p}_{E_1}^2\cong \mathcal{O}_{E_1}/\mathfrak{p}_{E_1}$ has order $p$ (since $E_1/\QQ_p$ is totally ramified).

To be explicit we fix a uniformizing parameter $\pi$ of $E_1$ and recall that
$E^\times = \langle \pi\rangle\times \mu_E' \times U_E^{(1)}$.
Any $\gamma \in  \Gamma$ can be written in the form $\gamma = \gamma_1 \gamma_2$ with $\gamma_1 \in G(E/E_1), \gamma_2 \in G(E/E_2)$.
The wild inertia group $\Gamma_1$ is equal to $G(E/E_2)$ and hence we obtain $\pi^{\gamma - 1} = \pi^{\gamma_2 - 1} \in U_{E_1}^{(1)} \sseq  U_{E}^{(1)}$.
In addition, by (\ref{composite iso}) and the fact that ${\rm Tr}_{E/E_1}$ acts as
  multiplication by $p$ on $\mathfrak{p}_{E_1}$, we see that $\pi^{\gamma-1}$ has
  trivial image in $( U_E^{(1)}/U_E^{(2)})_{\Gamma}$.

We set
\[
T := \langle \left(E^\times\right)^p, U_E^{(2)}, I_\Gamma(E^\times) \rangle = \langle \left(E^\times\right)^p, U_E^{(2)}, I_\Gamma(U_E^{(1)}) \rangle
\]
and note that the map
\[ E^\times \to \langle \pi \rangle/\langle \pi^p \rangle \times (U_E^{(1)}/U_E^{(2)})_\Gamma, \,\,\,\,
\pi^a \epsilon y \mapsto ( \pi^a \mod{\langle \pi^p \rangle}, yU_E^{(2)} \mod I_\Gamma ( U_E^{(1)}/U_E^{(2)})),\]
%
%
where $a\in \ZZ$, $\epsilon\in \mu_E'$ and $y \in U_E^{(1)}$, induces an isomorphism of the quotient group $Q:= E^\times / T$ with the direct product $\langle \pi \rangle/\langle \pi^p \rangle \times  (U_E^{(1)}/U_E^{(2)})_\Gamma \cong \ZZ/p\ZZ\times \ZZ/p\ZZ$.

In particular, if we fix an element $u$ of $U_{E}^{(1)}$ that generates $( U_E^{(1)}/U_E^{(2)})_{\Gamma}$, then the order $p$ subgroups of $Q$
 correspond to the subgroups generated by the classes of the elements $u$ and $\pi\cdot u^i$ for $i \in \{0,1,2,\dots , p-1\}$.

In addition, since $L/E$ is ramified if and only if $N$ does not contain $u$, the quotients that we require correspond to the subgroups
$Q_i := \langle \pi\cdot u^i \mod T \rangle$ for $i \in \{0,1,2,\dots , p-1\}$.
The corresponding subgroups $N_i$ of $E^\times$ are given by $N_i := \langle \pi u^i,  T \rangle$ and
we write $L_i$ for the fields that correspond to $N_i$ via local class field theory.

If $i \not= 0$, then $Q_i$ does not contain the class of $\pi$ so $G({L_i/E})$ is generated by
${\rm rec}_{L_i/E}(\pi) = {\rm rec}_{L_i/E_1}({\rm N}_{E/E_1}(\pi)) =  {\rm rec}_{L_i/E_1}(\pi)^p$
 and hence $G({L_i/E_1})$ is cyclic of order $p^2$ (and so $G({L_i/\QQ_p})$ has exponent $p^2$).

Finally we claim that $G({L_0/\QQ_p})$ has exponent $p$. To prove this it is enough,
 in view of the possible presentations (\ref{presentations}), to show $G({L_0/\QQ_p})$ contains two non-cyclic
subgroups of order $p^2$. Hence, since its inertia subgroup $G({L_0/E_2})$ is one such subgroup (as $L_0/\QQ_p$ is weakly ramified),
it is enough to prove $G({L_0/E_1})$ also has exponent $p$.

To do this we note that $G({L_0/E})$ is generated by ${\rm rec}_{L_0/E}(u) = {\rm rec}_{L_0/E_1}({\rm N}_{E/E_1}(u))$
and so ${\rm N}_{E/E_1}(u)$ is an element of order $p$ in $E_1^\times / {\rm N}_{L_0/E_1}(L_0^\times)$. Since
\[
{\rm N}_{L_0/E_1}(L_0^\times) = {\rm N}_{E/E_1}(N_0) = {\rm N}_{E/E_1}(\langle \pi, T \rangle) \sseq \langle \pi^p, U
_{E_1} \rangle
\]
we see that $\pi \not\in {\rm N}_{L_0/E_1}(L_0^\times) $ and $\pi^p \in {\rm N}_{L_0/E_1}(L_0^\times) $.  So it finally remains to
show that $\pi$ and ${\rm N}_{E/E_1}(u)$ generate different subgroups of $E_1^\times / {\rm N}_{L_0/E_1}(L_0^\times)$. But
if ${\rm N}_{E/E_1}(u) \pi^{n}$ were contained in ${\rm N}_{L_0/E_1}(L_0^\times) $ for some integer $n$, then $p$ would divide $n$ since $E/E_1$
is unramified of degree $p$. But this would then imply that ${\rm N}_{E/E_1}(u)$ belongs to ${\rm N}_{L_0/E_1}(L_0^\times) $ which is a contradiction.


This completes the proof of Proposition \ref{degree p extensions}.

\subsubsection{}\label{mixed degree extensions}

Following Lemma \ref{cwr vanish} and Remark \ref{not vanish}, the weakly ramified Galois extensions $L/\Qu$ of smallest degree for which the projection of $\mathfrak{c}_{L/\QQ}$ to ${\rm Cl}(G(L/\QQ))$ might not vanish are non-abelian and of degree $\ell^2 p$ for an odd prime $p$ and an odd prime $\ell$ that divides $p-1$. This motivates us to investigate such extensions numerically (in \S\ref{deg 63 extensions}) and the next result lays the groundwork for such investigations by determining a family of local extensions that satisfies the required conditions.

\begin{prop}\label{degree plsquared extensions}
Let $\ell$ and $p$ be odd primes with $\ell$ dividing $p-1$ . Then there exist exactly $\ell$
(non-isomorphic) weakly ramified non-abelian Galois extensions $L$ of $\QQ_p$ of degree $\ell^2p$ with $G({E/\Qp})\cong \Ze/l\Ze \times \Ze/l\Ze$, where $E := L^C$ and $C$ is the unique Sylow-$p$-subgroup of $G({L/\Qp})$.
\end{prop}

\begin{proof} Let $L/\Qp$ be an extension with the stated conditions and set $G := G({L/\Qp})$.

As $\ell$ divides $p-1$ the $\ell$-th roots of unity are contained in $\Qp$ and $\Qu_p^\times / (\Qu_p^\times)^\ell \simeq   \Ze/\ell\Ze \times \Ze/\ell\Ze$, so that
$E$ is the maximal abelian extension of $\Qp$ of exponent $\ell$. Explicitly, $E = E_1E_2$ where $E_1$ is the unramified extension of degree $\ell$ and
$E_2 := \Qp(\sqrt[\ell]{p})$. By local class field theory $L$ corresponds to a subgroup $X$ of $E^\times$ such that $X$
is stable under the action of $\Gamma := G/C$, $|E^\times/X| = p$ and $U_E^{(2)} \sseq X$ (by \cite[Cor.~3 p. 228]{LocalFields}).

Let $H$ be a subgroup of $\Gamma$ such that $|H| = \ell$. Since $H$ is cyclic the extension $L/E^H$ is abelian if and only if $H$ acts trivially on $E^\times/X$.
As a consequence $\Delta := G({E/E_1})$ acts non trivially on  $E^\times/X$ since otherwise $G(L/E_1) = G_0(L/\Qp)$ would be abelian and by
\cite[Cor.~2, p. 70]{LocalFields} this contradicts $G_2(L/E_1) = 1$.

Since $p \nmid |\Gamma|$ the $\Fp[\Gamma]$-module  $E^\times/X$ decomposes as $E^\times/X = \bigoplus_\phi e_\phi ( E^\times/X )$
where $\phi$ runs over the $\Fp$-valued abelian characters of  $\Gamma$ and $e_\phi$ denotes the usual idempotent in $\Fp[\Gamma]$.
In addition, since $|E^\times/X|= p$, exactly one of the components, for $\phi = \phi_0$ say,  is non-trivial.

Since $H_0 := \ker(\phi_0)$ acts trivially on $e_{\phi_0}(E^\times / X)$ we deduce that $H_0 \ne \Delta$. Then, writing
$T_H := \sum_{h \in H}$ for any subgroup
$H$ of $\Gamma$, one has $T_H(E^\times/X) = (E^\times/X)^H$ and so, since $(E^\times/X)^\Gamma \sseq (E^\times/X)^\Delta = 0$, we deduce $H_0 \ne \Gamma$.


We claim that $X$ contains $\langle \mu_E', \sqrt[\ell]{p}, U_E^{(2)}, I_{H_0}(U_E^{(1)}) \rangle$.
To see this note $\mu_E' \sseq X$ as $(\mu_E')^p = \mu_E'$. Since $T_\Delta(E^\times/X) = 0$ we obtain
$T_\Delta(\sqrt[\ell]{p}) = N_{E_2/\Qp}(\sqrt[\ell]{p}) = p  = (\sqrt[\ell]{p})^\ell \in X$.
As $\ell\not=p$ it follows that $\sqrt[\ell]{p} \in X$. Finally, as $L/E^{H_0}$ is abelian, $X$ must contain
$I_{H_0}(E^\times)$, and hence also  $I_{H_0}(U_E^{(1)})$, as required.

We will show below that for any subgroup $H$ of $\Gamma$ with $|H| = \ell$ and $H \ne \Delta$ the subgroup
\[
X(H) :=  \langle \mu_E', \sqrt[\ell]{p}, U_E^{(2)}, I_{H}(U_E^{(1)}) \rangle
\]
is both stable under $\Gamma$ and satisfies $|E^\times / X(H)| = p$.

This will show, in particular, that $X = X(H_0)$. Conversely, since each subgroup $X(H)$ corresponds by local class field theory (and \cite[p.~70, Cor.~2]{LocalFields}) to a weakly ramified extension $L/\Qp$ as in the proposition, we will also have proved that the extensions $L$ in the proposition correspond uniquely to the subgroups $H$ of $\Gamma$ with $|H| = \ell$ and $H \ne \Delta$.

It thus remains to show that for each subgroup $H$ as above the subgroup $X(H)$ is stable under $\Gamma$ and such that $|E^\times / X(H)| = p$.

Since $\gamma( \sqrt[\ell]{p}) \equiv  \sqrt[\ell]{p} \pmod{\mu_E'}$ for all $\gamma \in \Gamma$ it is immediate that $X(H)$ is $\Gamma$-stable.
The extension $E/E^H$ is unramified and therefore
\[
( U_E^{(1)} / U_E^{(2)})^H \simeq ( U_E^{(1)} / U_E^{(2)} )_H \simeq (  \frp_E / \frp_E^2 )_H
\simeq \frp_{E^H} /  \frp_{E^H}^2 \simeq \Ze / p\Ze,
\]
where the first isomorphism holds since each $U_E^{(n)}$ is $H$-cohomologically trivial,
the second is canonical and the third is  induced by the trace map $\tr_{E/E^H}$.
On the other hand, $( U_E^{(1)} / U_E^{(2)})_H = U_E^{(1)} / I_H(U_E^{(1)})U_E^{(2)}$ and so the decomposition $E^\times = \langle\sqrt[\ell]{p}\rangle \times \mu_E' \times U_E^{(1)}$ implies that the quotient
$E^\times / X(H) \simeq U_E^{(1)} / I_H(U_E^{(1)})U_E^{(1)}$ is isomorphic to $\Ze / p\Ze$, as required. \end{proof}

\begin{remark}\label{distinguish extensions} Assume the situation of Proposition \ref{degree plsquared extensions}.
  Then the extension $E/\Qp$ has $\ell$ subextensions $F_1, \ldots, F_\ell$ (corresponding to the subgroups $H$ of $\Gamma$ with
$|H| = \ell$ and $H \ne \Delta$) that are ramified over $\Qp$. For each such $F_i$ there exists precisely one extension
$L/\Qp$ that satisfies the assumptions of  Proposition \ref{degree plsquared extensions} and is also such that $L/F_i$ is abelian.
\end{remark}

\begin{remark}
Our primary motivation for obtaining the explicit descriptions of wildly and weakly ramified non-abelian Galois extensions that are
given above was to assist with attempts to make numerical investigations of the conjectures that we have discussed. However,
such explicit descriptions are of course interesting in their own right. In this context we recall that Vinatier \cite[Cor.~2.2]{Vinatier1}
has shown that for any positive multiple $n$ of $p$ there are exactly $p$ non isomorphic abelian wildly and weakly ramified extensions
of $\Qp$ of degree $n$ and, moreover, that these extensions can be described explicitly.
\end{remark}


\section{Numerical examples}\label{numerical examples}

In this section we investigate numerically, and thereby prove, Vinatier's conjecture for two new, and infinite, families of non-abelian weakly ramified Galois extensions of $\QQ$.

At the same time we shall also explicitly compute both sides of the equality in Conjecture \ref{local conj}
for all weakly ramified non-abelian Galois extensions of $\Qu_3$ of degree $27$,
thereby verifying this conjecture, and hence also Breuning's local epsilon constant conjecture, in this case.

\subsection{Extensions of degree $27$}

\subsubsection{} We first compute explicitly a set $\calF(3)$ as in \S\ref{p cubed section}. To do this we have to find $3$ weakly ramified Galois extensions $L$ of $\QQ$ of degree $27$ with a unique $3$-adic place $w$ and such that $L_w/\Qp$ runs over all extensions as in Proposition \ref{degree p extensions}.

In the following $p$ denotes $3$. We shall also only consider Galois extensions $F/\Qu$ that have a unique place $w$ above $p$ and so we write  $F_p$ in place of $F_w$.

We let $E_1$ be the subextension of $\Qu(\zeta_{p^2})$ of degree $p$ and $E_2$ an abelian extension of degree $p$ such that $p$ is inert in $E_2$.
We set $E := E_1E_2$ and let $\frp$ denote the unique prime ideal of $\OE$ above $p$. We write $\Gamma$ for the Galois group of $E/\Qu$.

Set $Q_2 := \{ \alpha \in \left(\OE / \frp^2 \right)^\times \mid \alpha \equiv 1 \pmod{\frp}\}$ and note that $(Q_2)_\Gamma \simeq
( U_{E_p}^{(1)} / U_{E_p}^{(2)})_\Gamma \simeq \Ze/p\Zp$
 by (\ref{composite iso}). Let $u \in \OE$ be such that the class of $u$ generates $( Q_2 )_\Gamma$ and let $\pi \in \OE$ be a uniformizing element for $\frp$.

By algorithmic global class field theory we compute ray class groups $\mathrm{cl}(q\frp^2)$ of conductor $q\frp^2$ for small positive integers $q$ with
$(q,p) = 1$ and search for subgroups $U \le \mathrm{cl}(q\frp^2)$ of index $p$ which are invariant under $\Gamma$ and such that the corresponding
extension $L/E$ is ramified at $\frp$. Each such $U$ corresponds to a Galois extension $L/\Qu$ whose completion at $p$ is one of the extensions of
Proposition \ref{degree p extensions}. As shown in the proof of Proposition \ref{degree p extensions} the local extensions $L_p/\Qp$ are in one-to-one correspondence with the
elements $\pi u^b$ for $b  \in \{0,1,2, \ldots, p-1\}$. More precisely, there is exactly one $b$ such that $\rec_{L_p/\Qp}(\pi u^b) = 1$.
Thus we have to find extensions $L/\Qu$ such that the resulting integers $b$ range from $0$ to $p-1$. In order
to compute $\rec_{L_p/\Qp}(\pi u^b)$ we compute $\xi \in \OE$ such that $\xi \equiv \pi \pmod{q}$ and  $\xi \equiv u^{-b} \pmod{\frp^2}$.
Then class field theory shows that $\rec_{L_p/\Qp}(\pi u^b) = \rec_{L/\Qu}(\xi\OE)$ which can be computed globally.

This approach is implemented in MAGMA. For $E_2$ we used the cubic subextension of $\Qu(\zeta_{19})$ and found a set of $3$ global
extensions $L/\Qu$ by taking $q \in \{5, 19 \}$.
The results of these computations can be reproduced using the MAGMA implementation
which can be downloaded from the first author's homepage.

\subsubsection{}\label{subsub deg 27} Using results of \cite{BleyBoltje} one can explicitly compute $\Cl(G)$ as an abstract group for each finite group $G$. In particular, for the two non-abelian groups of order $27$ one finds in this way that $\Cl(G)$ is cyclic of order $9$.

For each of the extensions $L/\Qu$ computed in the last section we can use the algorithm described in \cite[\S~5]{BleyWilson}
to compute the logarithm of $[\calA_{L/\Qu}]$ in $\Cl(G)$ with $G := G(L/\Qu)$. Since $G$ is of odd order,
$\calA_{L/\Qu}$ is a free $G$-module if and only if $[\calA_{L/\Qu}]$ is trivial.

In a little more detail, we first compute a normal basis element $\theta \in \calO_L$ and the $G$-module $\calA_\theta \sseq \QG$ such
that $\calA_\theta \theta = \calA_{L/\Qu}$. Then $\calA_\theta \simeq \calA_{L/\Qu}$ and the element $[\calA_\theta, \id, \ZG] \in K_0(\ZG, \QG)$
projects to the class of $\calA_{L/\Qu}$ in $\Cl(G)$. The algorithm in \cite{BleyWilson} now solves the discrete logarithm problem
for $[\calA_{\theta, \ell}, \id, \ZlG]$ in $K_0(\ZlG, \QlG)$ for each of the primes $\ell$ dividing the generalized index
$[\calA_\theta : \ZG]$ and then uses the recipe described
in \cite[\S~5]{BleyWilson} to compute the logarithm of $[\calA_{L/\Qu}]$ in $\Cl(G)$.

However, for an arbitrary choice of $\theta$ the algorithm will in general fail because of efficiency problems since this set of primes $\ell$
is often too large and contains primes $\ell$ which are much too big. We therefore first compute a maximal order $\calM$ in $\QG$
containing $\ZG$ and an element
$\delta \in \QG$ such that $\calM\calA_\theta = \calM\delta$. This is achieved by the method described in Steps (1) to (5) of
Algorithm 3.1 in  \cite{BleyJohnston}. We then  set $\theta' := \delta(\theta)$ and start over again by computing
$\calA_ {\theta'}$ such that $\calA_{\theta'}\theta' = \calA_{L/\Qu}$.
Then one has $\calM\theta' = \calM\calA_\theta\theta = \calM\calA_{L/\Qu} = \calM\calA_{\theta'}\theta'$.

Localizing at prime divisors $\ell$ of $G$ we obtain $\Ze_{(\ell)}[G]\theta' = \calA_{\theta', (\ell)}\theta'$ and hence
$\Ze_{(\ell)}[G] = \calA_{\theta', (\ell)}$. It follows that we only need to solve the discrete logarithm problem
for $[\calA_{\theta', \ell}, \id, \ZlG]$ in $K_0(\ZlG, \QlG)$
for primes $\ell$ dividing $|G|$.

The computations show that for each of the $3$ extensions computed in the previous section the $G$-module $\calA_{L/\Qu}$ is free.
As a consequence of these computations and Theorem \ref{nice explicit} we obtain the following result.

\begin{theorem}\label{degree 27 vinatier} For all extensions in $\calL(3)$ the $G$-module $\calA_{L/K}$ is free. In particular, Vinatier's conjecture holds for all non-abelian extensions $L/\Qu$ of degree $27$.
\end{theorem}

\subsubsection{}\label{subsub comp for fra} We now show how to compute $\fra_{L/\Qu}$
for the extension $L/\Qu$ in $\calF(3)$. By Theorems \ref{decomp thm} and \ref{special cases prop} we have $\fra_{L/\Qu} = \fra_{L_p/\Qp}$ and
 both $\fra_{L_p/\Qp}$ {and the right hand side of the equality in Conjecture \ref{local conj}} can be computed by adapting the methods of \cite{BleyDebeerst}.
In the following we indicate where special care has to be taken to improve the performance of the general implementation used
to obtain the results of  \cite{BleyDebeerst}

For the computation of $[\calA_{L/\Qu}, \kappa_L, H_L]$ we choose a normal basis element $\theta$ and write
\[
[\calA_{L/\Qu}, \kappa_L, H_L] = [\ZG\cdot\theta, \kappa_L, H_L] + [p\calA_{L/\Qu}, \id, \ZG\cdot\theta] + \delta_{G}(\Nrd_\QG(p)).
\]
For computational reasons we proceed as in \S\ref{subsub deg 27} and choose $\theta$ such that $\calM\theta = \calM\calA_{L/\Qu}$.
The second and the third term are straightforward to compute and the first term is given by norm resolvents (see, for example,  \cite[(13)]{BleyDebeerst}).

For the computation of $\delta_{G,p}(j_p((\psi_{2,*}-1)(\tau_{L_p/\Qp}'))$ we first digress to
describe the character theory of non-abelian groups of order $p^3$.

The centre $Z = Z(G)$ of any such group $G$ is equal to the  commutator subgroup of $G$ and the quotient $G/Z$ is isomorphic to $\Ze/p\Ze \times \Ze/p\Ze$ so that $G$ has $p^2 $ linear characters of order dividing $p$.

It is also easy to see that $G$ has normal subgroups $A$ that are isomorphic to $\Ze/p\Ze \times \Ze/p\Ze$ and contain $Z$.  We fix such a subgroup $A$ and for each non-trivial character $\lambda$ of $Z$ we choose a character $ \psi \in \widehat A$ which restricts to give $\lambda$ on $Z$. Then  it can be shown that $\ind_A^G(\psi)$ depends only on $\lambda$ and
does not depend on the choice of $\psi$. In addition, it is an irreducible character of $G$ of degree $p$. Since $(p-1)p^2 + p^2 = p^3$ we have found all irreducible characters of $G$.

In particular, for $p=3$ the characters of $G$ comprise the trivial character, $8$ linear characters of order $3$ and two characters of degree $3$.

Returning now to the computation of local Galois Gauss sums we essentially proceed as described in \cite[\S~2.5]{BleyBreuning} but for reasons of
efficiency must take care in the 'Brauer induction step' of loc. cit.

The computation of $\tau(\Qp, \chi)$ for abelian characters is clear. Let now $\chi = \ind_A^G(\psi)$ be one of the characters of degree $p$. We set $M := L^A$ and $N := L^{\ker(\psi)}$ and use the equalities
\begin{align*}
  \tau(\Qp, \chi) =\, &\tau(\Qp, \ind_A^G(\psi - 1_A))\cdot \tau(\Qp,  \ind_A^G(1_A))\\
   =\, &\tau(M_p, \psi - 1_A)\cdot \tau(\Qp,  \ind_A^G(1_A)) \\
   =\, &\tau(M_p, \psi)\cdot \prod_{\varphi \in \widehat G \atop{\varphi|_A = 1_A}}\tau(\Qp, \varphi).\end{align*}

The problematic part is the computation of $\tau(M_p, \psi)$. To explain why we write $\frf(\psi)$ for the conductor of $\psi$ and choose
$c \in M_p$ such that $\frf(\psi) \frD_{M_p/\Qp} = c\calO_{M_p}$. Then, by the definition of local Gauss sums, one has
\[
\tau(M_p, \psi) = \sum_x \psi\left( \rec_{N_p/M_p}(x/c) \right) \psi_{\mathrm{add}}\left( x/c \right)
\]
where $\psi_{\mathrm{add}}$ denotes the standard additive character and $x$ runs over a set of representatives of $\calO_{M_p}^\times$ modulo
$U_{M_p}^{(2)}$. From an algorithmic point of view it is now important to choose the subgroup $A$ such that $L^A/\Qu$ is totally ramified
(e.g. we may take $A = G(L/E_1)$) because then
$\calO_{M_p}^\times / U_{M_p}^{(2)}$ has order $6$ as compared to order $702$ if $M_p/\Qp$ were the unique
unramified extension of degree $3$.

From the explicit description of the unramified characteristic in (\ref{explicit unram char})
it is now easy to compute $\tau'(\Qp, \chi) = \tau(\Qp, \chi)y(\Qp, \chi)^{-1}$ for
all $\chi \in \widehat G$ and based on this it is straightforward by the methods of \cite{BleyWilson} to compute the term
 $\delta_{G,p}(j_p((\psi_{2,*}-1)(\tau_{L_p/\Qp}'))))$.

Our computations show that for all extensions $L/\Qu$ in $\calF(3)$ the element $\fra_{L_p/\Qp}$ is equal to the twisted unramified characteristic $\frc_{L_p/\Qp}$ defined in (\ref{twisted unram char def}). Combining this with Theorems \ref{decomp thm} and  \ref{special cases prop} we obtain the following result.

\begin{theorem} \label{new connection} If $L/K$ belongs to $\calL(3)$, then $\fra_{L/K} =  \frc_{L/K}$, where $\frc_{L/K}$ is as defined in (\ref{idelic twisted unram char def}). \end{theorem}

\begin{remark} The equality $\fra_{L/K} =  \frc_{L/K}$ in Theorem \ref{new connection} combines with the results of Theorems \ref{main result}(iv) and \ref{degree 27 vinatier} to imply that the image of $\frc_{L/K}$ in $\Cl(G)$ vanishes. Under the stated conditions, this fact also follows directly from Lemma \ref{cwr vanish}. Conversely, the results of Theorem \ref{new connection}, Theorem \ref{main result}(iv) and
Lemma \ref{cwr vanish} combine to give an alternative proof of Therem \ref{degree 27 vinatier}.\end{remark}

\begin{remark} By adapting the methods implemented for \cite{BleyDebeerst} one can also compute
{the right hand side of the equality in Conjecture \ref{local conj}} for all  extensions $L/\Qu$ in $\calF(3)$. These computations show that
\[\mathcal{E}_{L_p/\Qp} - J_{2, L_p/\Qp} - \frc_{L_p/\Qp} - M_{L_p/\Qp} = \frc_{L_p/\Qp},\]
and thereby verify Conjecture \ref{local conj}, and hence also Breuning's conjecture, for these extensions. Combining this fact with \cite[Th.~4.1]{mb2} and \cite[Cor.~7.6]{BleyBurns} one also finds that the central conjecture of \cite{BleyBurns} is valid for all $L/K$ in $\calL(3)$ for which $G_w$ has order $27$ and exponent $3$ for each wildly ramified place $w$ of $L$.\end{remark}

\subsection{Extensions of degree $63$}
\label{deg 63 extensions}

\subsubsection{}\label{explicit set ellp} Let $\ell$ and $p$ be odd primes with $\ell$ dividing $p-1$. We now sketch how to compute a set of Galois extensions
$L/\Qu$ of degree $\ell^2p$ such that $L/\Qu$ is at most tamely ramified outside $p$ and the
extensions $L_w/\Qp$ cover the set of local extensions of Proposition \ref{degree plsquared extensions} (where as usual $w$ denotes the unique place
of $L$ above $p$).

We use a simple heuristic approach which is motivated by the proof of Proposition \ref{degree plsquared extensions}
and which works well for $\ell = 3$ and $p=7$.

We fix a cyclic extension $E_1/\Qu$ of degree $\ell$ such that $p$ is inert and $\ell$ is unramified. Let $E_2$
denote the unique subextension of $\Qu(\zeta_p)$ of degree $\ell$. Let $E := E_1E_2$ denote the compositum of $E_1$ and
$E_2$ and let $F_1, \ldots, F_\ell$ be the subextensions of $E/\Qu$ of degree $\ell$ which are ramified at $p$.
Then the completions $F_{i,p}$ of the $F_i$ at the unique prime above $p$ range over the set of totally ramified
cyclic extensions of $\Qp$ of degree $\ell$. By Remark \ref{distinguish extensions} the extensions $N/\Qp$ which are non-abelian, wildly and weakly ramified can be distinguished by the unique subfield $F_{i,p}$ such that $N/F_{i,p}$ is abelian.

Let now $\frp$ denote the unique prime of $\OE$ above $p$. We then compute ray class groups $\cl(q\frp^2)$ for small
rational integers $q$ with $(q, \ell p)=1$ and search for subgroups $U$ of $\cl(q\frp^2)$ of index $p$ which are invariant
under $G(E/\Qu)$ and such that the corresponding
extension $L/E$ is both wildly and weakly ramified above $\frp$. Then $L_p / F_{i,p}$ is abelian, if and only if
$G(E/F_i)$ acts trivially on the quotient $\cl(q\frp^2)/U$ (or equivalently, $I_{G(E/F_i)}\cl(q\frp^2) \sseq U$).
A search using MAGMA produces these extensions.
The results can be reproduced with the MAGMA programs which can be
downloaded from the first author's homepage.

\subsubsection{} We now fix $\ell := 3$ and $p :=7$ and  apply class group methods to verify Vinatier's conjecture for the three extensions $L/\Qu$  described in the previous section. The principal approach is exactly the same as described in \S\ref{subsub deg 27}.

For the locally free class group of a non-abelian group $G$ of order $63$ such that $G/C \simeq \Ze/3\Ze \times \Ze/3\Ze$ (where $C$
denotes the Sylow-$7$-subgroup) one finds that $\Cl(G)$ is isomorphic to $\Ze/3\Ze \times \Ze/3\Ze \times \Ze/12\Ze$.

Our computations show that for each of the three extensions $L$ computed in the previous subsection Vinatier's conjecture is valid.
Taken in conjunction with Theorem \ref{decomp thm} and Lemma \ref{vinatier lemma 2} this fact implies the following result.

\begin{theorem}\label{last result} Let $L/K$ be a weakly ramified odd degree Galois extension of number fields for which at
each wildly ramified place $v$ of $K$ one has $K_v = \Qu_7$, $|G_w| = 63$ and $G_w/C_w$ is isomorphic to $\Ze/3\Ze \times \Ze/3\Ze$,
where $w$ denotes a fixed place of $L$ above $v$, $G_w$ its decomposition subgroup in $G(L/K)$ and $C_w$ is the Sylow-$7$-subgroup of $G_w$. Then $\calA_{L/K}$ is a free
$G(L/K)$-module.
\end{theorem}


\begin{remark} One can also use numerical methods to show that for each of the extensions $L/K$ considered in Theorem \ref{last result} the projection of $\frc_{L/K}$ to $\Cl(G)$ vanishes.
\end{remark}

\subsection{$\fra_{L/K}$ and idelic twisted unramified characteristics}\label{second conj sect} An extension of the methods used in \S\ref{subsub comp for fra} also allowed us to numerically compute the element $\fra_{L_p/\Qp}$ for one of the three extensions $L/\Qu$ of degree $63$ discussed in \S\ref{explicit set ellp}, so that $\ell = 3$ and $p=7$. (For the other two extensions that occur in this setting, however, the necessary computations became too complex and did not finish in reasonable time.)

In particular, in this respect it is useful to note that groups of order $\ell^2p$ are monomial, and hence that one can proceed as in \S\ref{subsub comp for fra} for the computation of the local Galois Gauss sums.

These numerical computations showed that $\fra_{L_p/\Qp} = \frc_{L_p/\Qp}$. 
 Taken in conjunction with Theorem \ref{new connection}, Corollary \ref{vinatier thm} and the observation in Remark \ref{simpler variant}, this fact motivates us to make the following remarkable conjecture.

\begin{conj}\label{second local conj} For any weakly ramified odd degree Galois extension of number fields $L/K$ one has $\fra_{L/K} = \frc_{L/K}$.  \end{conj}


\begin{remark}\label{last rem} Upon comparing Conjectures \ref{local conj} and \ref{second local conj} one obtains, for each odd prime $\ell$ and each weakly ramified odd degree Galois extension of $\ell$-adic fields $F/E$, an explicit conjectural formula
\[ \delta_{\Gamma, \ell}(J_{2,F/E}) = {\mathcal{E}}_{F/E} - 2c_{F/E} - M_{F/E}\]
that computes Galois-Jacobi sums in terms of fundamental classes and twisted unramified characteristics.
\end{remark}

\bibliographystyle{plain}

\begin{thebibliography}{9}

\bibitem{AgboolaBurns} A. Agboola, D. Burns,
\newblock On twisted forms and relative $K$-theory,
\newblock Proc. London Math. Soc. \textbf{92} (2006) 1-28.


\bibitem{BleyBurns} W. Bley, D. Burns,
\newblock Equivariant Epsilon Constants, Discriminants and \'{E}tale Cohomology,
\newblock Proc. London Math. Soc. \textbf{87} (2003) 545-590.

\bibitem{BleyBreuning} W. Bley, M. Breuning,
\newblock Exact algorithms and Epsilon Constant Conjectures,
\newblock  Illinois Journal of Mathematics {\bf 52}, No. 3 (2009), 773-797.

\bibitem{BleyDebeerst} W. Bley, R. Debeerst,
\newblock Algorithmic proof of the epsilon constant conjecture,
\newblock Math. Comp. {\bf 82} (2013), 2363-2387.

\bibitem{BleyCobbe} W. Bley, A. Cobbe,
\newblock Equivariant epsilon constant conjectures for weakly ramified extensions,
\newblock Math. Z. {\bf 283} (3) (2016), 1217-1244.

\bibitem{BleyBoltje} W. Bley, R. Boltje,
\newblock Computation of locally free class groups,
\newblock in: F. Hess, S. Pauli, M. Pohst (Eds.), Algorithmic Number Theory, Lecture Notes in Computer Science 4076 (2006), Springer, 72--86.

\bibitem{BleyWilson} W. Bley, S. M. J. Wilson,
\newblock Computations in relative algebraic $K$-groups,
\newblock London Math. Soc. J. Comput. Math. {\bf 12} (2009), 166-194.

\bibitem{BleyJohnston} W. Bley, H. Johnston,
\newblock Computing generators of free modules over orders in group algebras,
\newblock J. Algebra (Computational Section) 320 (2008), 836--852.

\bibitem{BK} S. Bloch, K. Kato,
\newblock $L$-functions and Tamagawa numbers of motives,
\newblock in The Grothendieck Festschrift Vol I, Progress in Math. Vol 86,
Birkh\"{a}user (1990), 333-400.

\bibitem{BreuPhd} M. Breuning,
\newblock Equivariant Epsilon Constants for Galois extensions of Number Fields and $p$-adic fields,
\newblock PhD thesis, King's College London, 2004.

\bibitem{mb2} M. Breuning,
\newblock Equivariant local epsilon constants and \'etale cohomology,
\newblock J. London Math. Soc. (2) {\bf 70} (2004), 289-306.




\bibitem{Burns} D. Burns,
\newblock Equivariant Whitehead Torsion and Refined Euler Characterisics,
\newblock CRM Proceedings and Lecture Notes \textbf{36} (2004) 239-255.

\bibitem{BurnsChinburg} D. Burns, T. Chinburg,
\newblock Adams operations and integral Hermitian-Galois representations,
\newblock Amer. J. Math. {\bf 118} (1996) 925-962.

\bibitem{cn-t2} Ph. Cassou-Nogu\`es, M. J. Taylor,
\newblock Constante de l'equation fonctionnelle de la fonction $L$ d'Artin d'une representation symplectique et moderee,
\newblock Ann. Inst. Fourier (Grenoble) {\bf 33} (1983) 1-17.


\bibitem{cn-t} Ph. Cassou-Nogu\`es, M. J. Taylor,
\newblock Operations d'Adams et groupe des classes d'algebre de groupe,
\newblock J. Algebra {\bf 95} (1985) 125-152.


\bibitem{cv} L. Caputo, S. Vinatier,
\newblock Galois module structure of the square root of the inverse different in even degree tame extensions of number fields,
\newblock J. Algebra {\bf 468} (2016) 103-154.

\bibitem{cpt} T.~Chinburg, G.~Pappas, M. J.~Taylor,
\newblock $\epsilon$-constants and equivariant Arakelov-Euler characteristics,
\newblock Ann. Scient. \'Ec. Norm. Sup. {\bf 35} (2002), 307--352.

\bibitem{cpt2} T.~Chinburg, G.~Pappas, M. J.~Taylor,
\newblock Duality and Hermitian Galois module structure,
\newblock Proc. London Math. Soc. {\bf 87} (2003) 54-108.

\bibitem{CurtisReinerII} C. Curtis, I. Reiner,
\newblock Methods of Representation Theory Vol. 2,
\newblock Wiley, New York, (1987).


\bibitem{Erez} B. Erez,
\newblock The Galois Structure of the Square Root of the Inverse Different,
\newblock Math. Z. \textbf{208} (1991) 239-255.

\bibitem{ErezTaylor} B. Erez, M. J. Taylor,
\newblock Hermitian Modules in Galois Extensions of Number Fields and Adams Operations,
\newblock Ann. Math.  \textbf{135} (1992) 271-296.

\bibitem{Frohlich} A. Fr\"{o}hlich,
\newblock Galois Structure of Algebraic Integers,
\newblock Ergebnisse Math. 3, Springer-Verlag, Berlin, 1983.

\bibitem{fro2} A. Fr\"ohlich,
\newblock Classgroups and Hermitian Modules,
\newblock Birkh\"auser, Boston, 1984.

\bibitem{hahn} C. Hahn,
\newblock On the Square Root of the Inverse Different via Relative Algebraic K-Theory,
\newblock PhD Thesis, King's College London, 2016.

\bibitem{hall} M. Hall,
\newblock The theory of groups,
\newblock Macmillan, New York, 1963.

\bibitem{Henniart} G. Henniart,
\newblock Rel\`evement global d'extensions locales: quelques probl\`emes de plongement,
\newblock Math. Ann. {\bf 319} (2001) 75--87.


\bibitem{neuk solvable} J. Neukirch,
\newblock On solvable extensions of number fields,
\newblock Invent. Math. {\bf 53} (1979) 135--164.

\bibitem{Neukirch} J. Neukirch,
\newblock Algebraic Number Theory, Springer-Verlag, Berlin, (1999).

\bibitem{PauliRoblot} S. Pauli, X. Roblot,
\newblock On the computation of all extensions of a $p$-adic field of a given degree
\newblock Math.Comp. \textbf{70} (2001) 1641-1659.

\bibitem{PV} E. J. Pickett, S. Vinatier,
\newblock Self-Dual Integral Normal Bases and Galois Module Structure,
\newblock Compos. Math. \textbf{149} (2013) 1175-1202.


\bibitem{LocalFields}J.~P.~Serre,
\newblock Local Fields,
\newblock volume {\bf 67} of Graduate texts in Mathematics, Springer-Verlag, New York 1979.

\bibitem{Shafarevics} I. R. Shafarevich,
\newblock On p-extensions
\newblock Mat.Sb.,Nov.Ser. \textbf{20(62)} (1947), 351-363; English Transl., AMS Transl. II. Ser. \textbf{4} (1956), 59-72.

\bibitem{swan}R. G. Swan,
\newblock Algebraic $K$-theory,
\newblock Lecture Note in Math. {\bf 76}, Springer, 1968.


\bibitem{taylor} M. J. Taylor,
\newblock Classgroups of Group Rings,
\newblock London Math. Soc. Lecture Notes Series {\bf 91},
Cambridge Univ. Press (1984).

\bibitem{U} S. V. Ullom,
\newblock Normal bases in Galois extensions of number fields,
\newblock Nagoya Math. J. {\bf 34} (1969), 153-167.

\bibitem{Vinatier1} S. Vinatier,
\newblock Structure galoisienne dans les extensions faiblement ramifies de $\mathbb{Q}$,
\newblock J. Number Theory \textbf{91} (2001), 126-152.

\bibitem{Vinatier3} S. Vinatier,
\newblock Une famille infinie d'extensions faiblement ramifi\'ees,
\newblock Math. Nachr. \textbf{243} (2002), 165-187.

\bibitem{Vinatier2} S. Vinatier,
\newblock Sur la racine carr\'ee de la codiff\'erente,
\newblock J. Th\'eor. Nombres Bordeaux \textbf{15} (2003) 393-410.

\bibitem{Yamagishi} M. Yamagishi,
\newblock On the number of Galois $p$-extensions of a local field,
\newblock Proc. Amer. Math. Soc. \textbf{123} (1995) 2375-2380.

\end{thebibliography}

\vskip 1truein


\end{document}